\documentclass[11pt,amstex]{amsart}%[12pt,amstex]
\usepackage{hyperref}
 \usepackage[capitalize]{cleveref}
\usepackage{enumerate,color}
\usepackage{amssymb}
\usepackage{mathtools}
\usepackage[utf8]{inputenc}
\usepackage[T1]{fontenc}
\usepackage[colorinlistoftodos]{todonotes}
\usepackage[mathscr]{eucal}
\usepackage{amssymb}
\usepackage{stackengine}

\usepackage[a4paper, hmargin={2.5cm,2.5cm}, centering]{geometry}

\everymath{\displaystyle}

\everymath{\displaystyle}

\usepackage{xcolor}

\newtheorem{lemma}{Lemma}[section]
\newtheorem{theorem}[lemma]{Theorem}
\newtheorem*{theorem*}{Theorem}
\newtheorem{corollary}[lemma]{Corollary}

\newtheorem{proposition}[lemma]{Proposition}
\newtheorem{remark}[lemma]{Remark}
\newtheorem*{proposition*}{Proposition}

\newtheorem{conjecture}[lemma]{Conjecture}

\theoremstyle{remark}

\newcommand{\E}{{\mathbb E}}
\newcommand{\F}{{\mathbb F}}
\newcommand{\N}{{\mathbb N}}

\newcommand{\Q}{{\mathbb Q}}
\newcommand{\R}{{\mathbb R}}

\newcommand{\Z}{{\mathbb Z}}

\newcommand{\norm}[1]{\left\Vert #1\right\Vert}
\newcommand{\nnorm}[1]{\lvert\!|\!| #1|\!|\!\rvert}

\theoremstyle{definition}
\newtheorem*{definition*}{Definition}
\newtheorem*{conjecture*}{Conjecture}
\newtheorem*{remark*}{Remark}
\newtheorem*{remarks*}{Remarks}
\newtheorem*{claim*}{Claim}
\newtheorem*{convention}{Convention}

\newtheorem{example}{Example}

\def \c {\bold{c}}

\def \E {\overline{\mathbb{E}}}

\def \g {\bold{g}}
\def \h {\bold{h}}

\def \p {\bold{p}}
\def \q {\bold{q}}
\def \u {\bold{u}}
\def \U {\bold{U}}
\def \v {\bold{v}}

\def \X {\bold{X}}

\def \I {\mathcal{I}}

\def \vl {\varlimsup_{N\to \infty}}
\def \F {\sup_{\substack{  (I_{N})_{N\in\mathbb{N}} \\ \text{ F\o lner seq.} }}}
\def \ei {\mathbb{E}_{n\in I_{N}}}

\title[Seminorms and joint ergodicity]{Seminorms for multiple averages along polynomials and applications to joint ergodicity} 
\author{Sebasti{\'a}n Donoso, Andreas Koutsogiannis and Wenbo Sun}
\address[Sebasti{\'a}n Donoso]{Instituto de Ciencias de la Ingenier\'ia, Universidad de O'Higgins, Av. Libertador Bernardo O’Higgins 611, Rancagua, 2841959, Chile}
\email{sebastian.donoso@uoh.cl}
\curraddr{\textsc{Departamento de Ingenier\'{\i}a Matem\'atica and Centro de Modelamiento Matem\'atico, Universidad de Chile \& IRL 2807 - CNRS, Beauchef 851, Santiago, Chile.}} \email{sdonoso@dim.uchile.cl}

\address[Andreas Koutsogiannis]{Department of mathematics, The Ohio State University, 231 West 18th Avenue, Columbus, OH 43210-1174, USA}
 \email{koutsogiannis.1@osu.edu}
\curraddr{\textsc{Department of Mathematics, Aristotle University of Thessaloniki, Thessaloniki, 54124, Greece}}
\email{akoutsogiannis@math.auth.gr}
\address[Wenbo Sun]{Department of mathematics, The Ohio State University,  231 West 18th Avenue, Columbus, OH 43210-1174, USA} \email{sun.1991@osu.edu}
\curraddr{\textsc{Department of Mathematics, Virginia Tech, 225 Stanger Street, Blacksburg, VA, 24061, USA}}
\email{swenbo@vt.edu}

\thanks{The first author is supported by Fondecyt Iniciaci\'on en Investigaci\'on Grant 11160061 and grant Conicyt-PIA Program AFB 170001.}

\subjclass[2010]{Primary: 37A05; Secondary: 37A30, 28A99}

\keywords{Multiple averages, polynomials, characteristic factors, concatenation theorem}

\begin{document}

\begin{abstract}

Exploiting the recent work of Tao and Ziegler on a concatenation theorem on factors, 
we find explicit characteristic factors for multiple averages along polynomials on systems with commuting transformations, and use them to study criteria of joint ergodicity for sequences of the form $(T^{p_{1,j}(n)}_{1}\cdot\ldots\cdot T^{p_{d,j}(n)}_{d})_{n\in\mathbb{Z}},$ $1\leq j\leq k$, where  $T_{1},\dots,T_{d}$ are commuting measure preserving transformations on a probability measure space and $p_{i,j}$ are integer polynomials. To be more precise,
 we provide a sufficient condition for such sequences to be jointly ergodic, giving also a characterization for sequences of the form $(T^{p(n)}_{i})_{n\in\mathbb{Z}}, 1\leq i\leq d$ to be jointly ergodic, answering a question due to Bergelson.
\end{abstract}

\maketitle

\section{introduction}

\subsection{Characteristic factors for multiple averages}

Let $\X=(X,\mathcal{B},\mu,T)$ be a measure preserving $\mathbb{Z}$-system.\footnote{ By this we mean that $(X,\mathcal{B},\mu)$ is a probability space and $T$ is an invertible measure preserving transformation, \emph{i.e.}, $\mu(T^{-1}A)=\mu(A)$ for all $A\in\mathcal{B}$. We also denote such a system as $(X,\mathcal{B},\mu,(S_g)_{g\in \mathbb{Z}})$ later in this paper, where $S_n=T^n,$ \emph{i.e.}, the composition of $T$ with itself $n$ times if $n\geq 0$ (and the composition of $T^{-1}$ $-n$ times if $n<0$).} When $T$ is ergodic (\emph{i.e.}, the measure of any $T$-invariant set is $0$ or $1$), the von Neumann ergodic theorem (see for example \cite[Theorem~2.21]{ET}) asserts that for all $f\in L^{2}(\mu)$, the $L^{2}(\mu)$ limit of the ``time average''  $\frac{1}{N}\sum_{n=0}^{N-1}T^{n}f$ equals to the ``natural'' one, namely the ``space limit'' $\int_{X}f\,d\mu$. 

 In the past decades,  the $L^2$-limit behavior of the ``multiple averages'' became a central topic in ergodic theory. Several  authors have studied averages for a single transformation $T$, as 
\begin{equation}\label{i1}
\frac{1}{N}\sum_{n=0}^{N-1}T^{p_{1}(n)}f_{1}\cdot\ldots\cdot T^{p_{k}(n)}f_{k},
\end{equation}  
 averages for several (usually commuting) $T_i$'s, as
\begin{equation}\label{i1*}
\frac{1}{N}\sum_{n=0}^{N-1}T_1^{p_{1}(n)}f_{1}\cdot\ldots\cdot T_k^{p_{k}(n)}f_{k}
\end{equation}  	
and even more general averages as

\begin{equation}\label{i1**}\frac{1}{N}\sum_{n=0}^{N-1}\prod_{i=1}^m T_i^{p_{i,1}(n)}f_1\cdot\ldots\cdot \prod_{i=1}^m T_i^{p_{i,k}(n)}f_k \end{equation}
for some $m, k\in\mathbb{N}^\ast$, suitable integer valued sequences $(p_i(n))_{n\in \mathbb{N}},$ $(p_{j,i}(n))_{n\in\mathbb{N}}$ and $f_i\in L^{\infty}(\mu),$ $1\leq i\leq k,$ $1\leq j\leq m$.\footnote{ Even though the expressions \eqref{i1}, \eqref{i1*} and \eqref{i1**} are stated for general, suitable, integer valued sequences, we use this notation as we will only deal with (integer) polynomial ones in what follows.} Fruitful results have been obtained, which include, but are not limited to \cite{A,BD,B,BLS,CFH,FK,Ho,HK,Jo,KK,K1,L9,T,W,Z2}.
In particular, it was proved by  Walsh \cite{W} (following the ideas of Tao \cite{T}) that the multiple (uniform) averages, as in \eqref{i1**}, converge in the $L^2$ sense for any integer valued polynomials $p_i$ when $T_1,\ldots,T_m$ span a nilpotent group. 
However, the result in \cite{W} does not give any description or information about the limit.
In general, very little is known about the limit of multiple averages.

 The existing results employ the idea of \emph{characteristic factors}, which intends to reduce the average under study to a more tractable one. For a single transformation $T$ and for linear $p_i$'s, the main content of \cite{HK} is the introduction of some seminorms that control the behavior of the average \eqref{i1} and are characterized by  \emph{nilsystems}. These seminorms were also used by Leibman (in \cite{L9}) to bound the limit of \eqref{i1} for polynomial $p_i$'s (always in the context of a single transformation).  For several commuting transformations, Host (in \cite{Ho}) introduced similar seminorms to bound the limit of \eqref{i1*} for linear $p_i$'s but in that case there was still no clear connection to nilsystems (see also \cite{Sun,TZ} for slight generalizations of these seminorms).  When considering non linear polynomials $p_i$'s, even less is known and even simple cases can be very intricate. For instance,  Austin in \cite{A1,A2} found precise characteristic factors for some specific cases of quadratic polynomials for $k=2$ (and linear polynomials for $k=3$).

In this paper, under a further development of  a recent result by Tao and Ziegler (\cite{TZ}) on concatenation (intersection) of factors, we provide an upper bound for the limit of \eqref{i1**} for  any $m,k \in \N^{\ast} $ and polynomials $p_{i,j}$ taking integer values at integers by using some seminorms on the system (generically called \emph{Host-Kra seminorms}), which, to the best of our knowledge, has never been studied before in this generality. We state here a simplified more aesthetic one-parameter version of our main result, and refer the readers to Theorem~\ref{T:2} below for the result in its full generality:

\begin{theorem}[Bounding multiple averages along polynomials by seminorms]\label{T:3}
	Let $d,k,K\in\mathbb{N}^\ast$ and $p_{1},\dots,p_{k}\colon\mathbb{Z}\to\mathbb{Z}^{d}$ be  a family of polynomials of degrees at most $K$ such that $p_{i}, p_{i}-p_{j}$ are not constant for all $1\leq i,j\leq k$, $i\neq j$, where $p_{i}(n)=\sum_{0\leq v\leq K}b_{i,v}n^{v}$ for some $b_{i,v}\in \mathbb{Q}^{d}$.
	Denote the set of the coefficients and pairwise differences of the coefficients (excluding ${\bf 0}$) of the polynomials with 
	\[ R=\bigcup_{0< v\leq K}\{b_{i,v}, b_{i,v}-b_{i',v}\colon 1\leq i, i'\leq k\}\backslash\{{\bf 0}\}.\]
	Let $(X,\mathcal{B},\mu, (T_{g})_{g\in\mathbb{Z}^{d}})$ be a $\mathbb{Z}^{d}$-system (see Section \ref{def} for the definition). If the Host-Kra seminorm $\|f_i \|_{{\{G(r)^{\times \infty}\}_{r\in R}}}$ (see Section \ref{s:2} for definitions) of $f_{i}$ equals to 0 for some $1\leq i\leq k$, then  
	\[ \lim_{N-M\to \infty} \frac{1}{N-M}\sum_{n=M}^{N-1}T_{p_{1}(n)}f_1\cdot\ldots\cdot T_{p_{k}(n)}f_k = 0.\]	
\end{theorem}

\begin{remark*}
Unlike the conventional ``finite-step'' Host-Kra seminorms, the seminorms $\|\cdot \|_{{\{G(r)^{\times \infty}\}_{r\in R}}}$ that we use are ``infinite-step'' ones. It is an interesting question to ask whether one can replace the ``infinite-step'' seminorms in the main theorems of this paper by ``finite-step'' ones.
\end{remark*}

\subsection{The joint ergodicity property}
An interesting application of Theorem \ref{T:3} and its stronger version, Theorem \ref{T:2}, is that they can be used to study problems of joint ergodicity, allowing us to answer a question due to Bergelson.
Back to the description of the limit of (\ref{i1**}), there are interesting cases where the limit has a ``simple'' description. In \cite{B},  Bergelson showed that if  $(X,\mathcal{B},\mu,T)$ is a weakly mixing system (meaning that $T\times T$ is ergodic for $\mu\times \mu$)\footnote{ In this case we also say that $T$ is a weakly mixing transformation.} and $p_{1},\dots,p_{k}$ are polynomials such that $p_{i}, p_{i}-p_{j}$ are non-constant for all $1\leq i,j\leq k, i\neq j$, then the $L^{2}(\mu)$ limit of (\ref{i1}) is the ``expected'' one, namely the ``multiple space limit'' $\prod_{i=1}^{k}\int_{X}f_{i}\,d\mu$.\footnote{ This result was previously obtained by  Furstenberg (in \cite{Fu0}) in the special case where $p_i(n)=in$, $i=1,\ldots,k$.}   
One can think of this result as a strong independence
property of the sequences $(T^{p_{i}(n)})_{n\in\mathbb{Z}}, 1\leq i\leq k$ in the weakly mixing case. This naturally leads to the following definition of joint ergodicity, in which we demand the average to converge to the expected limit. 

\begin{definition*}
	Let $d,k,L\in\mathbb{N}^\ast$, $p_{1},\dots,p_{k}\colon\mathbb{Z}^{L}\to\mathbb{Z}^{d}$ be functions, and $(X,\mathcal{B},\mu, (T_{g})_{g\in \mathbb{Z}^{d}})$ be a $\mathbb{Z}^{d}$-system.  We say that the tuple $(T_{p_{1}(n)},\dots,T_{p_{k}(n)})_{n\in\mathbb{Z}^{L}}$ is \emph{jointly ergodic} for $\mu$ if for every $f_{1},\dots,f_{k}\in L^{\infty}(\mu)$ and every F{\o}lner sequence $(I_{N})_{N\in\mathbb{N}}$ of $\mathbb{Z}^{L}$,\footnote{ Let $H$ be a subgroup of $\mathbb{Z}^L.$ A sequence of finite subsets $(I_{N})_{N\in\mathbb{N}}$ of  $H$ with the property that for all $g\in H,$ $\lim_{N\to\infty}\vert I_{N}\vert^{-1}\cdot\vert (g+I_{N})\triangle I_{N}\vert=0,$  is called \emph{F\o lner sequence} in $H$.} we have that
	\begin{equation}\label{444}
	\lim_{N\to\infty}\frac{1}{\vert I_{N}\vert}\sum_{n\in I_{N}}T_{p_{1}(n)}f_{1}\cdot\ldots\cdot T_{p_{k}(n)}f_{k}=\int_{X}f_{1}\,d\mu\cdot\ldots\cdot \int_{X}f_{k}\,d\mu,
	\end{equation}	
	where the limit is taken in $L^{2}(\mu)$. When $k=1$, we say that $(T_{p_{1}(n)})_{n\in\mathbb{Z}^{L}}$ is \emph{ergodic} for $\mu$ instead.\footnote{ The main reason we change from single-variable $p_i$'s to multi-variable ones and give the definition in this generality is technical. More specifically, we will deal with multi-variable integer valued polynomials, since our arguments, even for single-variable polynomials, naturally lead to multi-variable ones (for details, see the ``dimension-increment'' method, explained before Proposition~\ref{ext} below).}
\end{definition*}

For $d,L\in\mathbb{N}^\ast$, we say that $q\colon \mathbb{Z}^L\to\mathbb{Z}^d$ is an \emph{integer-valued polynomial} 
if $q=(q_1,\ldots,q_d),$ where each $q_i$ is an integer polynomial (meaning that it takes integer values at integers) of $L$ variables. The polynomial $q$ is \emph{non-constant} if some $q_i$ is non-constant. A family of polynomials $p_{1},\dots,p_{k}\colon\mathbb{Z}^{L}\to\mathbb{Z}^{d}$  is \emph{non-degenerate} if they are \emph{essentially non-constant} (meaning that each $p_{i}$ is not a constant polynomial) and \emph{essentially distinct} (meaning that $p_{i}-p_{j}$ is essentially non-constant for all $1\leq i,j\leq k, i\neq j$).\footnote{ Throughout this paper, when we write ``a polynomial $p\colon\mathbb{Z}^{L}\to\mathbb{Z}^{d},$'' we implicitly assume that $p$ is integer-valued, hence, in general, $p$ has rational coefficients.} Using this new language, it follows from \cite{B} that if $T$ is weakly mixing 
and $p_{1},\dots,p_{k}\colon\mathbb{Z}\to\mathbb{Z}$ is a non-degenerate family of polynomials, then $(T^{p_{1}(n)},\dots,T^{p_{k}(n)})_{n\in\mathbb{Z}}$ is jointly ergodic for $\mu$.
Later, it was proved by Frantzikinakis and Kra (in \cite{FK}) that if $p_{1},\dots,p_{k}\colon\mathbb{Z}\to\mathbb{Z}$ is an independent family of polynomials (\emph{i.e.}, every linear combination along integers of the $p_i$'s is non-constant) and $T$ is \emph{totally ergodic} (\emph{i.e.}, $T^{n}$ is ergodic for all $n\in\mathbb{Z}\backslash\{0\}$), then the tuple $(T^{p_{1}(n)},\dots,T^{p_{k}(n)})_{n\in\mathbb{Z}}$ is jointly ergodic for $\mu$ (for integer parts of real valued strongly independent polynomials, see \cite{KK}). 
By combining existing results, we have the following proposition:
\begin{proposition}\label{00}
	Let $d,k,L\in\mathbb{N}^\ast$ and $p_{1},\dots,p_{k}\colon\mathbb{Z}^{L}\to\mathbb{Z}^{d}$ be  a non-degenerate family of polynomials. Let $(X,\mathcal{B},\mu, (T_{g})_{g\in\mathbb{Z}^{d}})$ be  a $\mathbb{Z}^{d}$-system such that:
	\begin{itemize}
		\item[(i)] $T_{g}$ is ergodic for $\mu$ for all $g\in \mathbb{Z}^{d}\setminus\{{\bf 0}\}$; and 
		\item[(ii)] $(T_{p_{1}(n)}\times\dots\times T_{p_{k}(n)})_{n\in\mathbb{Z}^{L}}$ is ergodic for $\mu^{\otimes k}.$\footnote{ $\mu^{\otimes k}$ is the product measure $\mu\otimes \cdots \otimes\mu$ on $X^{k}.$} 
	\end{itemize}	
	Then 	
	$(T^{p_{1}(n)},\dots,T^{p_{k}(n)})_{n\in\mathbb{Z}^{L}}$ is jointly ergodic for $\mu$. 
\end{proposition}

Proposition~\ref{00} follows from \cite[Proposition~2.10]{Jo}, \cite[Theorem~10.1]{HK} (see also Theorem~\ref{ppq} below), and a variation of \cite[Theorem 3.9]{AA} (see Subsection \ref{s5.2} for details). We leave the details of the proof to the interested readers. 

We remark that in all the aforementioned results, one needs to postulate rather strong assumptions for the system. More specifically that either the transformation is weakly mixing or that infinitely many transformations $T_{g}$ are ergodic.
It is then natural to ask if one can obtain joint ergodicity results under weaker conditions, \emph{e.g.}, assuming that only finitely many transformations (or sequences of transformations with specific iterates) are ergodic, and finally, if there are any cases in which the sufficient condition is also necessary. In this direction, it is worth mentioning two results related to our study. 

Let $d\in\mathbb{N}^\ast$ and $(X,\mathcal{B},\mu,T_{1},\dots,T_{d})$ be a measure preserving system with commuting transformations.\footnote{ Here, as in the expression~\eqref{i1*}, $(X,\mathcal{B},\mu,T_{1},\dots,T_{d})$ can be understood as an abbreviation for the $\mathbb{Z}^{d}$-system $(X,\mathcal{B},\mu,(S_{g})_{g\in\mathbb{Z}^{d}})$, where $T_{1}=S_{(1,0,\dots,0)}, T_{2}=S_{(0,1,0,\dots,0)},\dots,T_{d}=S_{(0,\dots,0,1)}$.}
It was proved by Berend and Bergelson (in \cite{BD}) that the tuple $(T^{n}_{1},\dots,T^{n}_{d})_{n\in\mathbb{Z}}$ is jointly ergodic for $\mu$ if and only if $T_{i}T^{-1}_{j}$ is ergodic for $\mu$ for all $1\leq i,j\leq d, i\neq j,$ and $T_{1}\times\dots\times T_{d}$ is ergodic for $\mu^{\otimes d}$. Recently, it was proved by Bergelson, Leibman and Son (in \cite{BLS}) that if $p_{1},\dots,p_{d}\colon\mathbb{Z}\to\mathbb{Z}$ are generalized linear functions (\emph{i.e.}, functions of the form $p(n)=[\alpha_{1}n+\alpha_{2}],\; [\alpha_{3}[\alpha_{1}n+\alpha_{2}]]$, \emph{etc.},   where $[\cdot]$ denotes the integer part, or floor, function), then the tuple $(T^{p_{1}(n)}_{1},\dots,T^{p_{d}(n)}_{d})_{n\in\mathbb{Z}}$ is jointly ergodic for $\mu$ if and only if the sequence $(T^{p_{i}(n)}_{i}T^{-p_{j}(n)}_{j})_{n\in\mathbb{Z}}$ is ergodic for $\mu$ for all $1\leq i,j\leq d, i\neq j$, and the sequence $(T^{p_{1}(n)}_{1}\times\dots\times T^{p_{d}(n)}_{d})_{n\in\mathbb{Z}}$ is ergodic for $\mu^{\otimes d}$.
Note that both results, while being characterizations, hold under only the ergodicity assumption for finitely many transformations and sequences of transformations.

In this paper, we study joint ergodicity properties for sequences of transformations with polynomial iterates.
The following is our first application of Theorems \ref{T:3} and \ref{T:2}:

\begin{theorem}\label{t2}
	Let $d,k,K,L\in\mathbb{N}^\ast$ and $p_{1},\dots,p_{k}\colon\mathbb{Z}^{L}\to\mathbb{Z}^{d}$ be  a non-degenerate family of polynomials of degrees at most $K$. Suppose that $p_{i}(n)=\sum_{v\in\mathbb{N}^{L},\vert v\vert\leq K}b_{i,v}n^{v}$ for some $b_{i,v}\in \mathbb{Q}^{d}$.\footnote{ For $n=(n_{1},\dots,n_{L})\in\mathbb{Z}^{L}$ and $v=(v_{1},\dots,v_{L})\in\mathbb{N}^{L}$, $n^{v}$ denotes the quantity $n^{v_{1}}_{1}\cdot\ldots\cdot n^{v_{L}}_{L}$, and $\vert v\vert=v_{1}+\dots+v_{L}$.} 
		Denote the set of the coefficients and pairwise differences of the coefficients (excluding ${\bf 0}$) of the polynomials with 
		\begin{equation}\label{R}
		    R=\bigcup_{0<\vert v\vert\leq K}\{b_{i,v}, b_{i,v}-b_{i',v}\colon 1\leq i, i'\leq k\}\backslash\{{\bf 0}\}.
		\end{equation}
	Let $(X,\mathcal{B},\mu, (T_{g})_{g\in\mathbb{Z}^{d}})$ be a $\mathbb{Z}^{d}$-system such that:
	\begin{itemize}
		\item[(i)] For all $r\in R$, denoting $G(r)\coloneqq \text{span}_{\mathbb{Q}}\{r\}\cap\Z^{d}$ (see also the relation (\ref{G}) in the corresponding definition in Subsection~\ref{Sub:2.5}), the action $(T_{g})_{g\in G(r)}$ is ergodic for $\mu$;\footnote{ For a subgroup $H$ of $\Z^{d}$, $(T_{g})_{g\in H}$ is \emph{ergodic} for $\mu$ if every $A\in\mathcal{B}$ which is invariant under $T_{g}$ for all $g\in H$ is of $\mu$-measure $0$ or $1$.} and
		\item[(ii)] $(T_{p_{1}(n)}\times\dots\times T_{p_{k}(n)})_{n\in\mathbb{Z}^{L}}$ is ergodic for $\mu^{\otimes k}$.
	\end{itemize}	
	Then  	
	$(T_{p_{1}(n)},\dots,T_{p_{k}(n)})_{n\in\mathbb{Z}^{L}}$ is jointly ergodic for $\mu$. 
\end{theorem}

We remark that Theorem~\ref{t2} is stronger than Proposition~\ref{00} since we only require finitely many $T_{g}$'s to be ergodic, \emph{i.e.}, those $g$'s belonging to $R$, and the set $R$ has an explicit expression.

\begin{example}\label{Eg:1}
	Let $(X,\mathcal{B},\mu,T_{1},T_{2})$ be a system with two commuting transformations and assume that $(T^{n^{2}+n}_{1}\times T^{n^{2}}_{2})_{n\in\mathbb{Z}}$ is ergodic for $\mu\times\mu$. Then Theorem~\ref{t2} implies that if $T_{1}, T_{2}, T_{1}T^{-1}_{2}$ are ergodic for $\mu$, then $(T^{n^{2}+n}_{1}, T^{n^{2}}_{2})_{n\in\mathbb{Z}}$ is jointly ergodic for $\mu$.

	Conversely, the joint ergodicity of $(T^{n^{2}+n}_{1}, T^{n^{2}}_{2})_{n\in\mathbb{Z}}$
	implies the ergodicity of $(T^{n^{2}+n}_{1})_{n\in\mathbb{Z}}$ and $(T^{n^{2}}_{2})_{n\in\mathbb{Z}}$ for $\mu$, which in turn implies the ergodicity of $T_{1}$ and $T_{2}$ for $\mu$. However, the fact that $(T^{n^{2}+n}_{1}, T^{n^{2}}_{2})_{n\in\mathbb{Z}}$ is jointly ergodic for $\mu$ does not necessarily imply that $ T_{1}T^{-1}_{2}$ is ergodic (take for instance $T_1=T_2=T$ where $T$ is a weakly mixing transformation).
\end{example}	

Throughout this paper, Example \ref{Eg:1}
will be our main example via which we demonstrate how the main steps of our method work. Note that annoyingly enough, the expression of the limit of the average of the sequence $T_1^{n^2+n}f_1\cdot T_2^{n^2}f_2$ for bounded $f_1$ and $f_2$ cannot be immediately found from known results, despite the fact that the polynomials $p_1(n)=n^2+n$ and $p_2(n)=n^2$ are essentially distinct.  

\medskip

 The second application of Theorems \ref{T:3} and \ref{T:2} is the following theorem,
 which provides necessary and sufficient conditions for joint ergodicity of the polynomial sequences $T_{i}^{p(n)}, 1\leq i\leq d.$
 This
 generalizes the result from \cite{BD}  and answers a question due to Bergelson:\footnote{ Personal communication.}

\begin{theorem}\label{t1}
	Let  $d,L\in\mathbb{N}^\ast$, $p\colon\mathbb{Z}^{L}\to\mathbb{Z}$ be a polynomial and $(X,\mathcal{B},\mu,T_{1},\dots,T_{d})$ be a system with commuting transformations. Then $(T_{1}^{p(n)},\dots,T_{d}^{p(n)})_{n\in\mathbb{Z}^{L}}$ is jointly ergodic for $\mu$ if and only if both of the following conditions are satisfied:
	\begin{itemize}
		\item[(i)] $T_{i}T^{-1}_{j}$ is ergodic for $\mu$ for all $1\leq i,j\leq d,$ $i\neq j$; and
		\item[(ii)] $((T_{1}\times \dots\times T_{d})^{p(n)})_{n\in\mathbb{Z}^{L}}$ is ergodic for $\mu^{\otimes d}$.
	\end{itemize}
\end{theorem}

As an immediate example, for a system $(X,\mathcal{B},\mu,T_{1},T_{2})$ with two commuting transformations, the sequence $(T^{n^{2}}_{1}, T^{n^{2}}_{2})_{n\in\mathbb{Z}}$ is jointly ergodic for $\mu$ if and only if $T_{1}T^{-1}_{2}$ is ergodic for $\mu$ and $(T^{n^{2}}_{1}\times T^{n^{2}}_{2})_{n\in\mathbb{Z}}$ is ergodic for $\mu\times\mu$.

One might wonder if there are better descriptions of condition (ii) of Theorem \ref{t1}. In Section~\ref{s233}, we provide several criteria and equivalent conditions to (ii), related to the eigenvalues of the system. 

Based on the work of \cite{BD,BLS} and the main results of this paper, we have a natural conjecture:

\begin{conjecture}
	Let $d,k,L\in\mathbb{N}^\ast$, $p_{1},\dots,p_{k}\colon\mathbb{Z}^{L}\to\mathbb{Z}^{d}$ be polynomials and
	$(X,\mathcal{B},\mu,(T_{g})_{g\in\mathbb{Z}^{d}})$ be a $\mathbb{Z}^{d}$-system. Then $(T_{p_{1}(n)},\dots,T_{p_{k}(n)})_{n\in\mathbb{Z}^{L}}$ is jointly ergodic for $\mu$
	if and only if both of the following conditions are satisfied:
	\begin{itemize}
		\item[(i)] $(T_{p_{i}(n)-p_{j}(n)})_{n\in\mathbb{Z}^{L}}$ is ergodic for $\mu$ for all $1\leq i,j\leq k,$ $i\neq j$; and
		\item[(ii)] $(T_{p_{1}(n)}\times\dots\times T_{p_{k}(n)})_{n\in\mathbb{Z}^{L}}$ is ergodic for $\mu^{\otimes k}$.
	\end{itemize}
\end{conjecture}

\subsection{Method and Organization} Section~\ref{s:2} contains all the background material and Section~\ref{s233} the conditions equivalent to (ii) of Theorem \ref{t1} (see Proposition \ref{234}).

In order to prove the joint ergodicity results of this paper, we introduce a characterization theorem (Theorem \ref{T:2}, the stronger version of Theorem \ref{T:3}) in Section \ref{s:4}, which allows us to study joint ergodicity properties under the assumption that all the functions $f_{1},\dots,f_{k}$ are measurable with respect to certain Host-Kra characteristic factors (see Section \ref{s:2} for definitions).

Once Theorem \ref{T:2} is proven, a straightforward argument using results from \cite{AA, HK} yields the main results of this paper (see Subsection \ref{s5.2} for details).
The proofs of Theorems~\ref{t2}  and ~\ref{t1}, under the assumption of the validity of Theorem~\ref{T:2}, are enclosed in  Section~\ref{s:4} as well. In the same section, we also introduce the two main ingredients for proving Theorem~\ref{T:2}, namely Propositions~\ref{pet} (which we prove in Section~\ref{s:pet}) and \ref{pet3} (which we prove in Section~\ref{s:pet3}).

To obtain the characterization theorem (Theorem \ref{T:2}), we employ the, by now classical, ``PET induction'' (first introduced in \cite{B}), which allows us to convert the average in (\ref{444}) to a special case where every $p_{i}(n)$ is a linear function by repeatedly applying the van der Corput lemma (Lemma~\ref{lemma:iteratedVDC}). Adaptations of this method have been extensively studied in the past in \cite{CFH,Jo,L9} too. We explain it in detail in Section~\ref{s:3} tailored to our purposes. 

There are two major difficulties to carry out the PET induction in proving Theorem~\ref{T:2} though. The first is that although PET induction variations used in the past allow us to eventually reduce the left hand side of (\ref{444}) to an expression with linear iterates, they provide no information on the coefficients of these iterates, which is a crucial detail in describing the set $R$ defined in Theorem~\ref{t2}. To overcome this difficulty, we introduce a new alteration of this technique in Section~\ref{s:pet} (see the proof of Proposition~\ref{pet}) which allows us to keep track of the coefficients of the polynomials when we iteratively apply van der Corput (vdC) operations.

The second, and perhaps the most important problem, is how to bound the  left hand side of (\ref{444}) by some  Host-Kra-type seminorm of each function $f_i$. It turns out that for a general non-degenerate family of polynomials $p_{1},\dots,p_{k}\colon\mathbb{Z}^{L}\to\mathbb{Z}^{d},$
we can use the PET induction to bound the left hand side of (\ref{444}) by an averaged Host-Kra seminorm, as the right hand sides of \eqref{30} and \eqref{35} (see Section~\ref{s:4}). The problem-goal now is to bound such an averaged seminorm effectively by a single one. In the past, in analogous situations, issues like these were resolved under additional restrictions, such as the assumption that $d=1$ (\cite{B}), that all $T_{g}$'s are ergodic (\cite{FK,Jo}), or that $p_{1},\dots,p_{k}$ have different (and positive) degrees (\cite{CFH}). In this paper, we address this difficulty in Section~\ref{s:pet3} (see the proof of Proposition \ref{pet3}) in its full generality.
Our method is based on the recent
work of Tao and Ziegler on a concatenation theorem (\cite{TZ}).

\subsection*{Acknowledgements.}
We thank Vitaly Bergelson for bringing the problem that we are addressing in Theorem~\ref{t1} to our attention, and also for providing useful advice.  We also thank Andreu Ferr\'e Moragues and Nikos Frantzikinakis for pointing out a mistake in the initial version of the article regarding the deduction of Theorems~\ref{t2} and \ref{t1} from Proposition~\ref{T:2}.  Thanks also go to the anonymous referee for providing helpful comments and suggestions.  Finally, the second author thanks the Center for Mathematical Modeling (CMM) of the University of Chile, where this work started,  while the first author thanks The Ohio State University, where this work was completed, for their hospitality.

\subsection{Definitions and notations}\label{def}
We denote with $\mathbb{N}^\ast,$ $\mathbb{N},$ $\mathbb{Z},$ $\mathbb{Q},$ $\mathbb{R},$ $\mathbb{C}$ and $\mathbb{S}^1$ the sets of positive integers, non-negative integers, integers, rational numbers, real numbers, complex numbers and complex numbers of magnitude $1$ respectively. If $X$ is a set, and $d\in \mathbb{N}^\ast$, $X^d$ denotes the Cartesian product $X\times\cdots\times X$ of $d$ copies of $X$. 

We say that a tuple $(X,\mathcal{B},\mu,(T_{g})_{g\in \Z^d})$ is a \emph{$\Z^d$-measure preserving system} (or a \emph{$\Z^d$-system}) if $(X,\mathcal{B},\mu)$ is a probability space and $T_{g}\colon X\to X$ are measurable, measure preserving transformations on $X$ such that $T_{(0,\dots,0)}=id$ and $T_{g}\circ T_{h}=T_{g+h}$ for all $g,h\in \Z^{d}$. The system is \emph{ergodic} if for any $A\in\mathcal{B}$ such that $T_{g}A=A$ for all $g\in \Z^{d}$, we have that $\mu(A)\in \{0,1\}$.

We say that $(Y,\mathcal{D},\nu,(S_{g})_{g\in \Z^d})$ is a \emph{factor} of $(X,\mathcal{B},\mu,(T_{g})_{g\in \Z^d})$ if there exists a measurable map $\pi \colon (X,\mathcal{B},\mu)\to (Y,\mathcal{D},\nu)$ such that $\mu(\pi^{-1}(A))=\nu(A)$ for all $A\in \mathcal{D}$, and that $\pi\circ T_{g}=S_{g}\circ \pi$ for all $g\in\Z^{d}$.
A factor $(Y,\mathcal{D},\nu,(S_{g})_{g\in \Z^d})$ of $(X,\mathcal{B},\mu,(T_{g})_{g\in \Z^d})$ can be identified as a sub-$\sigma$-algebra $\mathcal{B}'$ of $\mathcal{B}$ or a subspace $V$ of $L^{2}(\mu)$ by setting $\mathcal{B}'\coloneqq \pi^{-1}(\mathcal{D})$ or $V\coloneqq L^{2}(\nu)\circ\pi$.
Given two $\sigma$-algebras $\mathcal{B}_1$ and $\mathcal{B}_2$, their \emph{joining} $\mathcal{B}_1\vee \mathcal{B}_2$ is the $\sigma$-algebra generated by $B_1\cap B_2$ for all $B_1\in \mathcal{B}_1$ and $B_2\in\mathcal{B}_2$, \emph{i.e.}, the smallest $\sigma$-algebra containing both $\mathcal{B}_1$ and $\mathcal{B}_2$. This definition extends to a countable collection of $\sigma$-algebras $\mathcal{B}_i$, $i\in \N,$ which we denote by $\bigvee_{i=0}^\infty \mathcal{B}_i.$

For simplicity all functions in $L^{\infty}(\mu),$ throughout the paper, are assumed to be real valued. All our results are easily extended to complex valued functions as well. 

We will denote with $e_i$ the vector which has $1$ as its $i$th coordinate and $0$ elsewhere. We use in general lower-case letters to symbolize both numbers and vectors but bold letters to symbolize vectors of vectors to highlight this exact fact, in order to make the content more reader-friendly. The only exception to this convention is the vector ${\bf 0}$ (\emph{i.e.}, the vector with coordinates only $0$'s) which we always symbolize in bold.

\subsubsection{Notation on averaging}
Throughout this article, we use the following notations about averages.
Let  $(a(n))_{n\in \Z^L}$ be a sequence of real numbers, or a sequence of measurable functions on a probability space $(X,\mathcal{B},\mu)$.
Denote 
\begin{align*}
\mathbb{E}_{n\in A}a(n)  &\coloneqq  \frac{1}{|A|} \sum_{n \in A } a({n}),  \;\; \text{ where A is a finite subset of} \;\;\mathbb{Z}^L, \\
\E^{\square}_{n\in \Z^{L}}a(n)  &\coloneqq \vl \mathbb{E}_{n \in [-N,N]^L } a(n),  \text{\footnotemark}\\ 
\E_{n\in\mathbb{Z}^{L}}a(n) &\coloneqq \sup_{\substack{  (I_{N})_{N\in\mathbb{N}} \\ \text{ F\o lner seq.} }}\varlimsup_{N\to\infty} \mathbb{E}_{n\in I_N} a(n),
\\
\mathbb{E}^{\square}_{n\in \Z^{L}}a(n)  &\coloneqq \lim_{N\to\infty} \mathbb{E}_{n \in [-N,N]^L } a(n) \;\; \text{ (provided that the limit exists)}, \\ 
\mathbb{E}_{n\in\mathbb{Z}^{L}}a(n) &\coloneqq \lim_{N\to\infty} \mathbb{E}_{n\in I_N} a(n) \text{ (provided that the limit exists for all F\o lner sequences $(I_{N})_{N\in\N}$)}.
\end{align*}
\footnotetext{ We use the symbol $\square$ to highlight the fact that the average is along the boxes $[-N,N]^{L}.$}
It is worth noticing that if the limit $\lim_{N\to\infty} \mathbb{E}_{n\in I_N} a(n)$ exists for all F\o lner sequences, then this limit does not depend on the particular F\o lner sequence. Also, along the paper, we use the notation $(I_{N})_{N\in\N}$ to denote a F\o lner sequence in $\Z^L$. 

We also consider \emph{iterated} averages. Let $(a(h_{1},\dots,h_{s}))_{h_{1},\dots,h_{s}\in \Z^L}$ be a multi-parameter sequence. We denote
\[\E_{h_{1},\dots,h_{s}\in \mathbb{Z}^{L}}a(h_{1},\dots,h_{s})\coloneqq \E_{h_{1}\in\mathbb{Z}^{L}}\ldots\E_{h_{s}\in\mathbb{Z}^{L}}a(h_{1},\dots,h_{s})\]
and adopt similar conventions for $\mathbb{E}_{h_{1},\dots,h_{s}\in \mathbb{Z}^{L}}$, $\E^{\square}_{h_{1},\dots,h_{s}\in \mathbb{Z}^{L}}$ and $\mathbb{E}^{\square}_{h_{1},\dots,h_{s}\in \mathbb{Z}^{L}}$ respectively.

\begin{convention}	
Throughout this paper, all the limits of measurable functions on a measure preserving system are taken in $L^{2}$ (unless otherwise stated). Even though all the expressions with polynomial iterates that we will encounter converge (in $L^2$) by \cite{W}, we don't a priori postulate any existence of such limits throughout the whole article. 
\end{convention}

\section{Background material}\label{s:2}

\subsection{The van der Corput lemma} The main tool in reducing the complexity of polynomial families and running the PET induction is the van der Corput lemma (and its variations), whose original proof can be found in \cite{B}. We state a convenient for us version that can be easily deduced from the one in \cite{B}.

\begin{lemma}[\cite{B}] \label{lemma:VdCclassic} Let $\mathcal{H}$ be a Hilbert space,
	$a\colon \Z^L\to \mathcal{H}$ be a sequence bounded by $1$, and $(I_{N})_{N\in\N}$ be a F\o lner sequence in $\Z^L$. Then 
	\begin{align*} \label{equation:VdC} \varlimsup_{N\to \infty} \left \| \mathbb{E}_{n\in I_N} a(n)\right \|^2 & \leq 4 \E^{\square}_{h\in\mathbb{Z}^{L}} \varlimsup_{N\to \infty}\vert \mathbb{E}_{n \in I_N} \langle a(n+h), a(n) \rangle  \vert. 
	\end{align*}
\end{lemma}
We also need the following variation of Lemma~\ref{lemma:VdCclassic}: 

\begin{lemma} \label{lemma:iteratedVDC}
	Let $\mathcal{H}$ be a Hilbert space,
	$(a(n;h_1,\ldots,h_s))_{(n;h_1,\ldots,h_s)\in (\Z^{L})^{s+1}}$\footnote{ We use this unorthodox notation to separate the variable $n$ from the $h_i$'s. The variable $n$ will play a different, comparing to the $h_i$'s, role later.} be a sequence bounded by $1$ in $\mathcal{H}$. 
	 Then for $\kappa\in \mathbb{N}$,
	\begin{align*}
	&\overline{\mathbb{E}}^{\square}_{h_1,\ldots,h_s\in \Z^{L}}\F\varlimsup_{N\to \infty}\left \| \mathbb{E}_{n\in I_{N}} a(n;h_1,\ldots,h_s) \right \|^{2\kappa}\\
	&\leq 4^{\kappa}\overline{\mathbb{E}}^{\square}_{h_1,\ldots,h_s,h_{s+1}\in \Z^{L}}\F\varlimsup_{N\to \infty} \left \vert   \mathbb{E}_{n\in I_{N}} \left \langle a(n+h_{s+1};h_1,\ldots,h_s) ,  a(n;h_1,\ldots,h_s)  \right \rangle \right \vert^{\kappa}.
	\end{align*}	
\end{lemma} 

\begin{proof}
	For fixed $h_1,\ldots,h_s$, we apply \cref{lemma:VdCclassic} for $a(n)=a(n;h_1,\ldots,h_s)$ and $h=h_{s+1}.$ By Jensen's inequality, we have
	\begin{equation}\nonumber
		\begin{split}
			&\F\varlimsup_{N\to \infty} \left \| \mathbb{E}_{n \in I_{N}} a(n;h_1,\ldots,h_s) \right \|^{2\kappa} 
			\\& \leq 4^{\kappa} \F\left(\overline{\mathbb{E}}^{\square}_{h_{s+1}\in \Z^{L}} \varlimsup_{N\to \infty} \left \vert  \mathbb{E}_{n\in I_{N}} \left \langle a(n+h_{s+1};h_1,\ldots,h_s) ,  a(n;h_1,\ldots,h_s)  \right \rangle \right \vert \right)^{\kappa} 
			\\& \leq  4^\kappa\F \overline{\mathbb{E}}^{\square}_{h_{s+1}\in \Z^{L}} \varlimsup_{N\to \infty} \left \vert   \mathbb{E}_{n\in I_{N}} \left \langle a(n+h_{s+1};h_1,\ldots,h_s),  a(n;h_1,\ldots,h_s)  \right \rangle \right \vert^{\kappa}
			\\& \leq  4^\kappa \overline{\mathbb{E}}^{\square}_{h_{s+1}\in \Z^{L}} \F\varlimsup_{N\to \infty} \left \vert   \mathbb{E}_{n\in I_{N}} \left \langle a(n+h_{s+1};h_1,\ldots,h_s),  a(n;h_1,\ldots,h_s)  \right \rangle \right \vert^{\kappa}.
		\end{split}	
	\end{equation}	
	The conclusion follows by taking the limsup of the averages over $h_{s},\ldots,h_{1}.$ 
\end{proof}

\subsection{Host-Kra characteristic factors}
The use of Host-Kra characteristic factors is a fundamental tool in studying problems related to multiple averages. They were first introduced in \cite{HK} for ergodic $\Z$-systems (see also \cite{Z2}) and later for $\Z^{d}$-systems in \cite{Ho}. In this paper, we need to use a slightly more general version of these characteristic factors, which is similar to the one used in \cite{Sun}.

For a $\Z^d$-measure preserving system ${\bf{X}}=(X,\mathcal{B},\mu,(T_{g})_{g\in \Z^d})$ and a subgroup $H$ of $\Z^d$, $\mathcal{I}(H)$ denotes the sub-$\sigma$-algebra of $(T_h)_{h\in H}$-invariant sets, \emph{i.e.}, sets $A\in \mathcal{B}$ such that $T_h A=A$ for all $h\in H$.  For an invariant sub-$\sigma$-algebra  $\mathcal{A}$ of $\mathcal{B}$, the measure $\mu\times_{\mathcal{A}} \mu$ denotes the \textit{relative independent product of $\mu$ with itself over $\mathcal{A}$}. That is, $\mu\times_{\mathcal{A}} \mu$ is the measure defined on the product space $X\times X$ as 
\[ \int_{X\times X} f\otimes g~ d(\mu\times_{\mathcal{A}} \mu)= \int_{X} \mathbb{E}(f\vert \mathcal{A})  \mathbb{E}(g \vert \mathcal{A})d\mu   \]
for all $f,g\in L^{\infty}(\mu)$. 

Let   $H_{1},\dots,H_{k}$ be subgroups of $\Z^d$. Define 
\[\mu_{H_{1}}=\mu\times_{\I(H_{1})}\mu\]
and for $k>1,$
\[\mu_{H_{1},\dots,H_{k}}=\mu_{H_{1},\dots,H_{k-1}}\times_{\mathcal{I}(H_{k}^{[k-1]})}\mu_{H_{1},\dots,H_{k-1}},\]
where $H^{[k-1]}_{k}$ denotes 
the subgroup of $(\Z^{d})^{2^{k-1}}$ consisting of all the elements of the form 
  $(h_k,\ldots,h_k)$ ($2^{k-1}$ copies of $h_{k}$) for some $h_{k}\in H_{k}$. The characteristic factor $Z_{H_{1},\dots,H_{k}}(\X)$ is defined to be the sub-$\sigma$-algebra of $\mathcal{B}$ such that \[\mathbb{E}(f\vert Z_{H_{1},\dots,H_{k}}(\X))=0 \text{ if and only if } \Vert f\Vert_{H_{1},\dots,H_{k}}^{2^{k}}\coloneqq \int_{X^{[k]}}f^{\otimes 2^{k}}\,d\mu_{H_{1},\dots,H_{k}}=0,\]
where $f^{\otimes 2^k}=f\otimes \cdots \otimes f$ and $X^{[k]}=X\times\cdots\times X$ ($2^k$ copies of $f$ and $X$ respectively).
When there is no confusion, we simply write $Z_{H_{1},\dots,H_{k}}:=Z_{H_{1},\dots,H_{k}}(\X)$ (with $\Vert \cdot\Vert_{H_{1},\dots,H_{k}}$ being the corresponding seminorm). 
Similarly to the proof of Lemma 4 of \cite{Ho} (or Lemma 4.3 of \cite{HK}), one can show that $Z_{H_{1},\dots,H_{k}}$ is well defined. Note that when $k=1$, $Z_{H_{1}}=\I(H_{1})$.
When we have $k$ copies of $H,$ we write $Z_{H^{\times k}}\coloneqq Z_{H,\dots,H},$ and $Z_{H^{\times \infty}}\coloneqq \bigvee_{k=1}^{\infty} Z_{H^{\times k}}$.

\begin{convention}
	For convenience, we adopt a flexible way to write the Host-Kra characteristic factors combining the aforementioned notation. For example, if $A=\{H_{1},H_{2}\}$, then the notation $Z_{A,H_{3},H^{\times 2}_{4},(H_{i})_{i=5,6}}$ refers to  $Z_{H_{1},H_{2},H_{3},H_{4},H_{4},H_{5},H_{6}}$, and $Z_{H_{1},H^{\times \infty}_{2},H^{\times \infty}_{3}}$ refers to $\bigvee_{k=1}^{\infty}Z_{H_{1},H^{\times k}_{2},H^{\times k}_{3}}$.\footnote{ Or, equivalently $\bigvee_{k_1=1}^{\infty}\bigvee_{k_2=1}^{\infty}Z_{H_{1},H^{\times k_1}_{2},H^{\times k_2}_{3}}$.  By the Lemma~\ref{replacement0} (i), the factors are independent of the order in which we take the subgroups.} We adopt a similar flexibility for the subscripts of the seminorms.  
	
	When each $H_{i}$ is generated by a single element $g_{i}$, we write $\Vert \cdot\Vert_{g_{1},\dots,g_{d}} \coloneqq \Vert \cdot\Vert_{H_{1},\dots,H_{d}}$ and  $Z_{g_{1},\dots,g_{d}}\coloneqq Z_{H_{1},\dots,H_{d}}$ in short.
\end{convention}

For the rest of the section, $\X=(X,\mathcal{B},\mu,(T_g)_{g\in\mathbb{Z}^d})$ will denote, as usual, a $\mathbb{Z}^d$-system.

\medskip

Let $H$ be a subgroup of $\Z^{d}$ and $(a(g))_{g\in H}$ be a sequence on a Hilbert space. If for all F\o lner sequences $(I_{N})_{N\in\N}$ in $H$, the limit $\lim_{N\to\infty}\mathbb{E}_{g\in I_{N}}a(g)$ exists, we then use $\mathbb{E}_{g\in H}a(g)$ to denote this limit.
The following theorem is classical (see for example \cite[Theorem~8.13]{ET}).

\begin{theorem}[Mean ergodic theorem for $\mathbb{Z}^{d}$-actions]\label{erg}
	For every $f\in L^{2}(\mu)$ and every subgroup $H$ of $\Z^{d}$, the limit $\mathbb{E}_{g\in H}T_{g}f$ exists in $L^{2}(\mu)$ and equals to $\mathbb{E}(f\vert \I (H))$ (or  $\mathbb{E}(f\vert Z_{H})$).
\end{theorem}

The following are some basic properties of the Host-Kra seminorms.

\begin{lemma}\label{replacement0}
	Let $H_{1},\dots,H_{k},H'$ be subgroups of $\Z^{d}$ and $f\in L^{\infty}(\mu)$. 
	\begin{itemize}
		\item[(i)] For every permutation $\sigma\colon\{1,\dots,k\}\to\{1,\dots,k\}$, we have that \[Z_{H_{1},\dots,H_{k}}(\X)=Z_{H_{\sigma(1)},\dots,H_{\sigma(k)}}(\X).\]
		\item[(ii)] If $\I(H_{j})=\I(H')$, then $Z_{H_{1},\dots,H_{j},\dots,H_{k}}(\X)=Z_{H_{1},\dots,H_{j-1},H',H_{j+1},\dots,H_{k}}(\X)$.
		\item[(iii)] For $k\geq 2$ we have that
		\[\Vert f\Vert^{2^{k}}_{H_{1},\dots,H_{k}}
		=\mathbb{E}_{g\in H_{k}}\Bigl\Vert f\cdot T_{g}f\Bigr\Vert^{2^{k-1}}_{H_{1},\dots,H_{k-1}},\] while for $k=1,$	
		\[\Vert f\Vert^{2}_{H_{1}}
		=\mathbb{E}_{g\in H_{1}}\int_{X} f\cdot T_{g}f\,d\mu.\]
		\item[(iv)] Let $k\geq 2$. If $H'\leq H_{j}$ is of finite index, then \[Z_{H_{1},\dots,H_{j},\dots,H_{k}}(\X)=Z_{H_{1},\dots,H_{j-1},H',H_{j+1},\dots,H_{k}}(\X).\]
		\item[(v)] If $H'\leq H_{j}$, then $Z_{H_{1},\dots,H_{j},\dots,H_{k}}(\X)\subseteq Z_{H_{1},\dots,H_{j-1},H',H_{j+1},\dots,H_{k}}(\X)$.
		
	\item[(vi)] For $k\geq 2$,  $\|f\|_{H_1,\ldots,H_{k-1}}\leq \| f \|_{H_1,\ldots,H_{k-1},H_k}$ and thus $Z_{H_1,\ldots,H_{k-1}}(\X)\subseteq  Z_{H_1,\ldots,H_{k-1},H_k}(\X).$
	
	 	\item[(vii)] 	For $k\geq 1$, if $H_1',\ldots, H_k'$ are subgroups of $
	 		\Z^d$, then $Z_{H_1,\ldots,H_k}(\X) \vee Z_{H_1',\ldots,H_k'}(\X) \subseteq Z_{H_1',\ldots,H_k',H_1,\ldots,H_k}(\X).$ 

	\end{itemize}	
\end{lemma}

\begin{proof}
	(i) and (ii) follow from \cite[Lemma~2.2]{Sun} (for (i), see also \cite{Ho}).
	
	To show (iii), if $k\geq 2$, then 
	\begin{eqnarray*}
		\Vert f\Vert^{2^{k}}_{H_{1},\dots,H_{k}}
		& = &\int_{X^{[k]}} f^{\otimes 2^{k}}\,d\mu_{H_{1},\dots,H_{k}}
		\\& = &\int_{X^{[k-1]}} f^{\otimes 2^{k-1}}\cdot \mathbb{E}( f^{\otimes 2^{k-1}}\vert \I(H^{[k-1]}_{k})) \,d\mu_{H_{1},\dots,H_{k-1}}
		\\& = &\mathbb{E}_{g\in H_{k}}\int_{X^{[d-1]}} f^{\otimes 2^{k-1}} \cdot  (T_{g}f)^{\otimes 2^{k-1}}\,d\mu_{H_{1},\dots,H_{k-1}} 
		\\& = &\mathbb{E}_{g\in H_{k}}\Bigl\Vert f\cdot T_{g}f\Bigr\Vert^{2^{k-1}}_{\X,H_{1},\dots,H_{k-1}},
	\end{eqnarray*}
	where we invoked the mean ergodic theorem (Theorem~\ref{erg}) in the penultimate equality. Similarly, for $k=1$,
	\begin{equation}\nonumber
	\begin{split} 
	\Vert f\Vert^{2}_{H_{1}}
	=\int_{X^{2}} f\otimes f\,d\mu_{H_{1}}
	=\int_{X}f\cdot  \mathbb{E}(f\vert \I(H_{1})) \,d\mu
	=\mathbb{E}_{g\in H_{1}}\int_{X}f\cdot  T_{g}f \,d\mu.
	\end{split}
	\end{equation}
	
	We now prove (iv). For convenience, we use multiplicative notation. By (i), we may assume without loss of generality that $j=k$.
	Suppose that $H_{k}=\sqcup_{i=1}^{l}g_{i}H'$ for some $l>0$ and $g_{i}\in \Z^{d}, 1\leq i\leq l$. We may assume that $g_{1}$ is the identity element in $\Z^d$. Let $(I_{N})_{N\in\mathbb{N}}$ be any F\o lner sequence in $H'$. We claim that  $(I_{N}\cdot\{g_{1},\dots,g_{l}\})_{N\in\mathbb{N}}$ is a F\o lner sequence in $H_{k}$. Indeed, by the elementary  inclusion $(A\cup B)\triangle C \subseteq (A\triangle C) \cup (B\triangle C)$ it follows that
	\begin{align*}(I_{N}\cdot\{g_{1},\dots,g_{l}\}) \triangle g (I_{N}\cdot\{g_{1},\dots,g_{l}\})  \subseteq \bigcup_{1\leq i,j\leq l} I_{N}g_i \triangle g I_N  g_j 
	 = \bigcup_{1\leq i,j\leq l}  g_i I_{N} \triangle g_j g I_{N},
	\end{align*}
	and since $\vert I_N\vert^{-1}\cdot\vert g_i I_{N} \triangle g_j g I_{N}\vert=\vert I_N\vert^{-1}\cdot\vert I_{N} \triangle (g_i^{-1}g_j g)I_{N} \vert\to 0$ as $N\to\infty,$  the claim follows. 
	
	By (iii), we have that
	\begin{equation}
	\begin{split}
	\Vert f\Vert^{2^{k}}_{H_{1},\dots,H_{k-1},H_{k}}
	& =\mathbb{E}_{g\in H_{k}}\Bigl\Vert f\cdot T_{g}f\Bigr\Vert^{2^{k-1}}_{H_{1},\dots,H_{k-1}}
	\\ & =\lim_{N\to\infty}\frac{1}{l\vert I_{N}\vert}\sum_{i=1}^{l}\sum_{g\in I_{N}}\Bigl\Vert f\cdot T_{g_{i}g}f\Bigr\Vert^{2^{k-1}}_{H_{1},\dots,H_{k-1}}
	\\& \geq \lim_{N\to\infty}\frac{1}{l\vert I_{N}\vert}\sum_{g\in I_{N}}\Bigl\Vert f\cdot T_{g}f\Bigr\Vert^{2^{k-1}}_{H_{1},\dots,H_{k-1}}
	\\& =\frac{1}{l}\Vert f\Vert^{2^{k}}_{H_{1},\dots,H_{k-1},H'}.
	\end{split}
	\label{50}
	\end{equation}
	On the other hand, since $\I({H}^{[k-1]}_{k})$ is a sub-$\sigma$-algebra of $\I({H'}^{[k-1]})$, by the Cauchy-Schwarz inequality,
	\begin{equation}
	\begin{split}
	\Vert f\Vert^{2^{k}}_{H_{1},\dots,H_{k-1},H'}
	& =\int_{X^{[k]}}f^{\otimes 2^{k}} \,d\mu_{H_{1},\dots,H_{k-1},H'}
	\\&=\int_{X^{[k-1]}}f^{\otimes 2^{k-1}}\cdot \mathbb{E}(f^{\otimes 2^{k-1}}\vert \I({H'}^{[k-1]})) \,d\mu_{H_{1},\dots,H_{k-1}} \\
	&=\int_{X^{[k-1]}} \Bigl  \vert \mathbb{E}(f^{\otimes 2^{k-1}} \vert \I({H'}^{[k-1]})) \Bigr \vert^2\,d\mu_{H_{1},\dots,H_{k-1}}	
	\\	&\geq \int_{X^{[k-1]}} \Bigl  \vert \mathbb{E}(f^{\otimes 2^{k-1}} \vert \I({H}^{[k-1]}_{k})) \Bigr \vert^2\,d\mu_{H_{1},\dots,H_{k-1}}	
	\\&=\int_{X^{[k-1]}}f^{\otimes 2^{k-1}}\cdot \mathbb{E}(f^{\otimes 2^{k-1}}\vert \I({H}^{[k-1]}_{k})) \,d\mu_{H_{1},\dots,H_{k-1}}
	\\&=\int_{X^{[k]}}f^{\otimes 2^{k}}\,d\mu_{H_{1},\dots,H_{k-1},H_{k}}=\Vert f\Vert^{2^{k}}_{H_{1},\dots,H_{k-1},H_{k}}.
	\end{split}
	\label{8}
	\end{equation}
	Therefore, $\Vert f\Vert_{H_{1},\dots,H_{k-1},H_{k}}=0\Leftrightarrow\Vert f\Vert_{H_{1},\dots,H_{k-1},H'}=0$, and the conclusion follows.
	
	(v) Since $\Vert f\Vert^{2^{k}}_{H_{1},\dots,H_{k-1},H_{k}}\leq \Vert f\Vert^{2^{k}}_{H_{1},\dots,H_{k-1},H'}$ by (\ref{8}) whenever $H'$ is a subgroup of $H_{k}$, we have that  $Z_{H_{1},\dots,H_{k-1},H_{k}}(\X)\subseteq Z_{H_{1},\dots,H_{k-1},H'}(\X)$. So (v) follows from (i).

 (vi) Similarly to (iii), and by Jensen inequality we have \begin{eqnarray*}
			\Vert f\Vert^{2^{k}}_{H_{1},\dots,H_{k-1},H_{k}}
			& = &\int_{X^{[k]}} f^{\otimes 2^{k}}\,d\mu_{H_{1},\dots,H_{k}}
			\\& = &\int_{X^{[k-1]}} \mathbb{E}(f^{\otimes 2^{k-1}}\vert \I(H^{[k-1]}_{k}))^2 \,d\mu_{H_{1},\dots,H_{k-1}}
			\\ & \geq & \Bigl(\int_{X^{[k-1]}} \mathbb{E}(f^{\otimes 2^{k-1}}\vert \I(H^{[k-1]}_{k})) \,d\mu_{H_{1},\dots,H_{k-1}} \Bigr)^2
			\\ & = & \Bigl(\int_{X^{[k-1]}} f^{\otimes 2^{k-1}}\, d\mu_{H_{1},\dots,H_{k-1}}  \Bigr)^2
			\\ & = & \| f  \|_{H_1,\ldots,H_{k-1}}^{2^{k}}
	\end{eqnarray*} 
(note that the penultimate equality holds because the function and its conditional expectation have the same integral), from where the conclusion follows. 

(vii) Applying (vi) several times, we get that both $Z_{H_1,\ldots,H_k}(\X)$ and$ Z_{H_1',\ldots,H_k'}(\X)$ are sub-$\sigma$-algebras of $Z_{H_1',\ldots,H_k',H_1,\ldots,H_k}(\X)$, hence so is their joining. \end{proof}

\begin{remark*}
	We caution the reader that Lemma~\ref{replacement0} (iv) is not valid for $k=1$. 
	In fact, for an ergodic $\Z$-system $\X=(X,\mathcal{B},\mu, T)$ where $T^2$ is not ergodic, we have $Z_{\mathbb{Z}}(\X)=I(\mathbb{Z})\neq I(2\mathbb{Z})=Z_{2\mathbb{Z}}(\X)$. The reason why this fails is that for $k=1$
 the inequality in (\ref{50}) is no longer valid since the term $\Bigl\Vert f\cdot T_{g_{i}g}f\Bigr\Vert^{2^{k-1}}_{\X,H_{1},\dots,H_{k-1}}$ is replaced by $\int_{X} f\cdot T_{g_{i}g}f\,d\mu$, which might  be negative.
\end{remark*}

As an immediate corollary of Lemma \ref{replacement0} (ii), we have:

\begin{corollary}\label{ppp}
	Let $H_{1},\dots,H_{k}$ be subgroups of $\Z^{d}$. If  the $H_{i}$-action $(T_{g})_{g\in H_{i}}$ is ergodic on $\X$ for all $1\leq i\leq k$, then $Z_{H_{1},\dots,H_{k}}(\X)=Z_{({\Z^{d}})^{\times k}}(\X)$.
\end{corollary}

\subsection{Structure theorem and nilsystems}
Let $X=N/\Gamma$, where $N$ is a ($k$-step) nilpotent Lie group and $\Gamma$ is a discrete cocompact subgroup of $N$. Let $\mathcal{B}$ be the Borel $\sigma$-algebra of $X,$ $\mu$ the Haar measure on $X,$ and for $g\in \Z^{d},$ let $T_{g}\colon X\to X$ with $T_{g}x=b_{g}\cdot x$ for some group homomorphism $g\mapsto b_{g}$ from $\Z^{d}$ to $N$. We say that $\X=(X,\mathcal{B},\mu,(T_{g})_{g\in \Z^{d}})$ is a \emph{($k$-step) $\Z^{d}$-nilsystem}.

An important reason which makes the Host-Kra characteristic factors powerful is their connection with nilsystems.  
The following is a slight generalization of \cite[Theorem~3.7]{Z} (see \cite[Lemma~4.4.3 and Theorem~4.10.1]{JG}, or Lemma~\ref{replacement0} (ii) and \cite[Theorem~3.7]{Sun}), which is a higher dimensional version of Host-Kra structure theorem (\cite{HK}).

\begin{theorem}[Structure theorem]\label{ppq}
	Let $\X$ be an ergodic $\mathbb{Z}^{d}$-system. Then $Z_{(\mathbb{Z}^{d})^{\times k}}(\X)$ is an inverse limit of $(k-1)$-step $\mathbb{Z}^{d}$-nilsystems. 
\end{theorem} 

The 1-step Host-Kra nilfactor is the \emph{Kronecker factor}, which is intimately related to the spectrum of the system (\cite{HK}). We say that a non-$\mu$-a.e. constant function $f\in L^{\infty}(\mu)$ is an \emph{eigenfunction} of the $\mathbb{Z}^{d}$-system $\X=(X,\mathcal{B},\mu, (T_{g})_{g\in\mathbb{Z}^{d}})$ if $T_{g}f=\lambda_{g} f$ for all $g\in \mathbb{Z}^{d},$ where  $g\mapsto \lambda_{g}$ is a group homomorphism from $\mathbb{Z}^{d}$ to $\mathbb{S}^{1}$. For each $g\in\mathbb{Z}^{d}$, we say that $\lambda_{g}$ is an \emph{eigenvalue} of $\X$. 
If $(X,\mathcal{B},\mu, T)$ is a $\mathbb{Z}$-system,  we say that a non-$\mu$-a.e. constant function $f\in L^{\infty}(\mu)$ is an \emph{eigenfunction} of $T$ if $Tf=\lambda f$ for some $\lambda\in\mathbb{S}^{1},$ and  we say that $\lambda$ is an \emph{eigenvalue} of $T$.

The \emph{Kronecker factor} $\mathcal{K}(\X)$ of the $\mathbb{Z}^{d}$-system $\X=(X,\mathcal{B},\mu, (T_{g})_{g\in\mathbb{Z}^{d}})$ is the sub-$\sigma$-algebra of $\mathcal{B}$ that corresponds to the algebra of functions spanned by the eigenfunctions of $\X$ in $L^{2}(\mu)$. 
As a special case of Theorem~\ref{ppq}, we have:

\begin{lemma}\label{k} For an ergodic $\Z^{d}$-system  $\X$, we have that $\mathcal{K}(\X)=Z_{\mathbb{Z}^{d},\mathbb{Z}^{d}}(\X)$.
\end{lemma}

An application of the Kronecker factor is to characterize single averages along polynomials.

\begin{proposition}\label{poly} 
	Let $L\in\mathbb{N}^\ast,$ $p\colon\mathbb{Z}^{L}\to\mathbb{Z}$ be a non-constant polynomial,
	\footnote{ We caution the reader that this result is only true for $p\colon\mathbb{Z}^{L}\to\mathbb{Z}^{d}$ with $d=1$. Indeed, for $d=2$, by taking $p\colon\mathbb{Z}\to\mathbb{Z}^{2},$ with $p(n)=(n,-n)$, $T_{(1,0)}=T_{(0,1)}=T$ for some $T$, and $f$ which is not constant to $0$ with $\mathbb{E}(f\vert Z_{\mathbb{Z}^{2},\mathbb{Z}^{2}}(\X))=0$, we have that $\mathbb{E}_{n\in\mathbb{Z}}T_{p(n)}f(x)=f(x)\not\equiv 0$.}
	$\X=(X,\mathcal{B},\mu, (T_{g})_{g\in\mathbb{Z}})$ be a $\mathbb{Z}$-system, and $f\in L^\infty(\mu)$. If $\mathbb{E}(f\vert Z_{\mathbb{Z},\mathbb{Z}}(\X))=0$, then 
	\[\mathbb{E}_{n\in\mathbb{Z}^{L}}T_{p(n)}f=0.\]
\end{proposition}

 Proposition \ref{poly} was proved implicitly in \cite[Section 2]{B0};
we also provide an alternative proof of it in Section~\ref{s:3} using the language of this paper.

\subsection{Concatenation theorem}
An essential ingredient in our approach is the following concatenation theorem established by Tao and Ziegler (in \cite{TZ}), which studies the properties of intersections of different characteristic factors.

\begin{theorem}[Concatenation theorem, \mbox{\cite[Theorem~1.15]{TZ}}]\label{ct0}
	Let  $\X$ be a $\Z^{d}$-system, $k,k'\in\mathbb{N}^\ast$ and $H_{1},\dots,H_{k},H'_{1},\dots,H'_{k'}$ subgroups of $\mathbb{Z}^d$.  Then
	\[Z_{H_{1},\dots,H_{k}}\cap Z_{H'_{1},\dots,H'_{k'}}\subseteq Z_{(H_{i}+H'_{i'})_{1\leq i\leq k, 1\leq i'\leq k'}}.\]
\end{theorem}

As an immediate corollary, we have: 
\begin{corollary}\label{ct}
	Let $\X$ be a $\Z^{d}$-system, $s,d_{1},\dots,d_{s}\in\mathbb{N}^\ast$ and $H_{i,j}, 1\leq i\leq s, 1\leq j\leq d_{i},$ be subgroups of $\Z^{d}$. Then
	\[ \bigcap_{i=1}^{s} Z_{H_{i,1},H_{i,2},\dots,H_{i,d_{i}}}\subseteq Z_{(H_{1,n_{1}}+H_{2,n_{2}}+\dots+H_{s,n_{s}})_{1\leq n_{i}\leq d_{i}, 1\leq i \leq s}} .\] 
\end{corollary}

\subsection{Range of polynomials}\label{Sub:2.5}

In this subsection we state and prove two elementary lemmas regarding the range of polynomials.

\begin{definition*}
	For $\textbf{b}=(b_{1},\dots,b_{L})\in(\mathbb{Q}^{d})^{L}, b_{i}\in\mathbb{Q}^{d}$, we define
	\begin{equation}\label{G}
		G(\textbf{b})\coloneqq \text{span}_{\mathbb{Q}}\{b_{1},\dots,b_{L}\}\cap\mathbb{Z}^{d},
	\end{equation}	
	 and
		\begin{equation}\nonumber
		G'(\textbf{b})\coloneqq \text{span}_{\mathbb{Z}}\{b_{1},\dots,b_{L}\}.
	\end{equation}	
\end{definition*}	

 Note that $G'(\textbf{b})$ is a subgroup of $G(\textbf{b})$ of finite index. ($G(\textbf{b})$ can be seen either as a subgroup or a subspace (over $\mathbb{Z}$) of $\mathbb{Z}^d$; we freely use both.)

\begin{lemma}\label{ag}
	Let $\c\colon(\mathbb{Z}^{L})^{s}\to(\mathbb{Q}^{d})^{L}$ be a polynomial and let $V$ be a subspace of $\mathbb{Z}^{d}$ over $\Z$. Then the set
	\[\{(h_{1},\dots,h_{s})\in(\mathbb{Z}^{L})^{s}\colon G(\c(h_{1},\dots,h_{s}))\subseteq V\}\] 
	is either $(\mathbb{Z}^{L})^{s}$ or of (upper) Banach density $0$.\footnote{ For a set $E\subseteq \Z^d,$ we define its \emph{upper Banach density} (or just \emph{upper density} when there is no confusion) with $d^\ast(E):=\vl \max_{t\in\mathbb{Z}^d}\frac{
|(E-t)\cap  \{1,\ldots, N \}^d|}{N^d}.$ If the limit exists, we say that its value is the \emph{Banach density} (or just \emph{density}) of $E$.}
\end{lemma}

\begin{proof}
	For convenience, denote 
		\[W\coloneqq \{(h_{1},\dots,h_{s})\in(\mathbb{Z}^{L})^{s}\colon G(\c(h_{1},\dots,h_{s}))\subseteq V\},\] 
	where one views $\c$ as the matrix:
	\[\c(h_{1},\dots,h_{s})=\begin{pmatrix}
	c_{1,1}(h_{1},\dots,h_{s}) & \dots & c_{1,L}(h_{1},\dots,h_{s})\\
	\vdots & \vdots & \vdots\\
	c_{d,1}(h_{1},\dots,h_{s}) & \dots & c_{d,L}(h_{1},\dots,h_{s})
	\end{pmatrix}
	\]
	for some polynomials $c_{i,j}\colon(\mathbb{Z}^{L})^{s}\to \mathbb{Q}$, $1\leq i\leq d, 1\leq j\leq L$. 
	
	We start with the case $V=\{{\bf 0}\}$. Let $W_{i,j}$ be the set of $(h_{1},\dots,h_{s})\in(\mathbb{Z}^{L})^{s}$ such that $c_{i,j}(h_{1},\dots,h_{s})=0$. Then $W=\bigcap_{i=1}^{d}\bigcap_{j=1}^{L} W_{i,j}$ and so it suffices to show that either each $W_{i,j}$ is $(\Z^L)^s$ or that some $W_{i,j}$ is of density $0.$ 
	By relabelling the variables, we may assume that $L=1$ (and change $s$ to $Ls$).
	Hence, it suffices to show that for a polynomial $c\colon\mathbb{Z}^{s}\to\mathbb{Z}$, the set 
	\[W=\{(h_{1},\dots,h_{s})\in(\mathbb{Z})^{s}\colon c(h_{1},\dots,h_{s})=0\}\]
	is either $\mathbb{Z}^{s}$ or of density $0.$

	If $s=1$, then either $c\equiv 0$ or $c(x)=0$ has finitely many roots. So $W$ is either $\mathbb{Z}$ or of upper Banach density $0.$
	Suppose now that the conclusion holds for some $s\geq 1$, and assume that
	$c(h_{1},\dots,h_{s+1})=\sum_{i=0}^{K}q_{i}(h_{2},\dots,h_{s+1})h^{i}_{1}$
	for some $K\in\mathbb{N}$  and polynomials $q_{i}\colon\mathbb{Z}^{s}\to\mathbb{Q}$ for all $0\leq i\leq K$. Let \[W'=\{(h_{2},\dots,h_{s+1})\in\mathbb{Z}^{s}\colon q_{i}(h_{2},\dots,h_{s+1})=0, 0\leq i\leq K\}.\]
	By induction hypothesis, either $W'=\mathbb{Z}^{s}$ or
	$W'$ is of upper Banach density $0.$
	If $W'=\mathbb{Z}^{s}$, then $c\equiv 0$ and so $W=\mathbb{Z}^{s+1}$. If  $W'$ is of upper Banach density 0, then $W\subseteq W_{1}\cup W_{2}$, where $W_{1}=\mathbb{Z}\times W'$ and $W_{2}=\{(h_{1},\dots,h_{s+1})\in\mathbb{Z}^{s+1}\colon (h_{2},\dots,h_{s+1})\notin W', c(h_{1},\dots,h_{s+1})=0\}.$ Since $W'$ is of upper Banach density 0, so is $W_{1}$. On the other hand, for any $(h_{2},\dots,h_{s+1})\notin W'$, $c(\cdot,h_{2},\dots,h_{s+1})$ is not constant 0 and so it has at most $K$ roots. This implies that $W_{2}$ is of upper Banach density 0, so $W$ is of density $0$, completing the induction.

	Now assume that $V\neq \{{\bf 0}\}$. Since $V$ is a subspace of $\mathbb{Z}^{d}$ over $\Z$, under a change of coordinates, we may assume that   $V=\{0\}^{\ell}\times \mathbb{Z}^{d-\ell}$ for some $0\leq \ell\leq d$. If $\ell=0$, then $V=\Z^{d}$ and there is nothing to prove.  If $\ell>0$, then by restricting to the first polynomials $c_{i,j}, 1\leq i\leq d, 1\leq j\leq \ell$, we are reduced to the case $V=\{{\bf 0}\}$, finishing the proof.  
\end{proof}

\begin{lemma}\label{ag2}
	Let $\c\colon(\mathbb{Z}^{L})^{s}\to(\mathbb{Q}^{d})^{L}$ be a polynomial given by\footnote{ Recall that for $n=(n_1,\ldots,n_L)\in \mathbb{Z}^L$ and $v=(v_1,\ldots,v_L)\in\mathbb{N}^L,$ $n^v$ denotes the quantity $n_1^{v_1}\dots n_L^{v_L}.$ We also use the convention $0^0=1.$}
	\[\c(h_{1},\dots,h_{s})=\sum_{a_{1},\dots,a_{s}\in\mathbb{N}^{L}}h^{a_{1}}_{1}\dots h^{a_{s}}_{s}\cdot \u(a_{1},\dots,a_{s})  \]
	for some ${\bf u}(a_1,\dots,a_s)\in (\mathbb{Q}^d)^L$ which all but finitely many are equal to $0.$ Then
	\[\text{span}_{\Q}\{G({\bf c}(h_1,\dots,h_s)):\;h_1,\dots,h_s\in\mathbb{Z}^L\}=\text{span}_{\Q}\{G({\bf u}(a_1,\dots,a_s)):\;a_1,\dots,a_s\in\mathbb{N}^L\}.\footnote{ Here, when $H_{i},i\in\N$ are subsets of $\mathbb{Q}^{d}$, we use the notation $\text{span}_{\mathbb{Q}}\{H_{i}\colon i\in\N\}$ to denote the set $\text{span}_{\mathbb{Q}}\{x\in \mathbb{Q}^{d}\colon x\in \cup_{i\in\N} H_{i}\}$.}\]
\end{lemma}

For the reader's convenience we first make the statement clear with an example, with $L=2,$ $s=1$, $d=4,$ and then present the proof. 
Let ${\bf c}\colon \Z^2 \to (\Z^{4})^2 $ be given by 
\[ {\bf c}(h_1,h_2)=\begin{pmatrix} h_1 & 0 \\ -3h_1h_2 & h_1 \\ h_1^2 & -h_2 -2h_2^2\\ 7h_1h_2 & h_1^2 \end{pmatrix}. \]

Denoting $h=(h_1,h_2),$ we have 
{\small \begin{align*}{\bf c}(h_1,h_2)&=h_1 \begin{pmatrix} 1 & 0 \\ 0 & 1 \\ 0 & 0 \\ 0 & 0 \end{pmatrix} + h_2\begin{pmatrix} 0 & 0 \\ 0 & 0 \\ 0 & -1 \\ 0 & 0 \end{pmatrix} +h_1h_2  \begin{pmatrix} 0 & 0 \\ -3 & 0 \\ 0 & 0 \\ 7 & 0 \end{pmatrix}+h_1^2\begin{pmatrix} 0 & 0 \\ 0 & 0 \\ 1 & 0 \\ 0 & 1 \end{pmatrix}   + h_2^2 \begin{pmatrix} 0 & 0 \\ 0 & 0 \\ 0 & -2 \\ 0 & 0 \end{pmatrix} \\
&= h^{(1,0)} \begin{pmatrix} 1 & 0 \\ 0 & 1 \\ 0 & 0 \\ 0 & 0 \end{pmatrix} + h^{(0,1)}\begin{pmatrix} 0 & 0 \\ 0 & 0 \\ 0 & -1 \\ 0 & 0 \end{pmatrix} +h^{(1,1)}  \begin{pmatrix} 0 & 0 \\ -3 & 0 \\ 0 & 0 \\ 7 & 0 \end{pmatrix}+h^{(2,0)}\begin{pmatrix} 0 & 0 \\ 0 & 0 \\ 1 & 0 \\ 0 & 1 \end{pmatrix}   + h^{(0,2)} \begin{pmatrix} 0 & 0 \\ 0 & 0 \\ 0 & -2 \\ 0 & 0 \end{pmatrix} \\ 
&= h^{(1,0)}  {\bf u}(1,0) + h^{(0,1)}{\bf u}(0,1) +h^{(1,1)} {\bf u}(1,1) +h^{(2,0)}{\bf u}(2,0) + h^{(0,2)} {\bf u}(0,2),
 \end{align*} }
where the ${\bf u}(i,j)$ denote the corresponding matrices from the previous step.

\cref{ag2} establishes that the span of the columns of ${\bf c}(h_1,h_2)$ (for all $h_1,h_2\in \mathbb{Z}$) equals to the span of the columns of the $\u(a_1,a_2)$ (for all $a_1,a_2\in \mathbb{N}$). More explicitly, it states that

\[  \text{span}_{\Q} \left \{\begin{pmatrix} h_1 \\ -3h_1h_2 \\ h_1^2 \\7h_1h_2 \end{pmatrix},\begin{pmatrix} 0 \\ h_1 \\ -h_2-2h_2^2 \\h_1^2 \end{pmatrix} : h_1,h_2\in \Z  \right \}   \]
equals to
\[ \text{span}_{\Q} \left \{ \begin{pmatrix} 1 \\ 0 \\ 0 \\0 \end{pmatrix}, \begin{pmatrix} 0 \\ 1 \\ 0 \\0 \end{pmatrix}, \begin{pmatrix} 0 \\ 0 \\ -1 \\0 \end{pmatrix}, \begin{pmatrix} 0 \\ -3 \\ 0 \\7 \end{pmatrix}, \begin{pmatrix} 0 \\ 0 \\ 1 \\0 \end{pmatrix} , \begin{pmatrix} 0 \\ 0 \\ 0 \\1 \end{pmatrix}, \begin{pmatrix} 0 \\ 0 \\ -2 \\0 \end{pmatrix} \right\}.   \]

\begin{proof}[Proof of \cref{ag2}]
	We first assume that $L=1$. In this case, we have that 
	\[c(h_{1},\dots,h_{s})=\sum_{a_{1},\dots,a_{s}\in\mathbb{N}}h^{a_{1}}_{1}\dots h^{a_{s}}_{s}\cdot u(a_{1},\dots,a_{s})\]
	for $h_{1},\dots,h_{s}\in\Z$ and some $u(a_{1},\dots,a_{s})\in\mathbb{Q}^{d}$. It suffices to show that
	\[\text{span}_{\mathbb{Q}}\{c(h_{1},\dots,h_{s})\colon h_{1},\dots,h_{s}\in\mathbb{Z}\}=\text{span}_{\mathbb{Q}}\{u(a_{1},\dots,a_{s})\colon a_{1},\dots,a_{s}\in\mathbb{N}\}.\]
	Since $c(h_1,\ldots,h_s)$ belongs to the $\Q$-span of $\{u(a_1,\ldots,a_s)\}_{a_1,\ldots,a_s \in \N},$ the inclusion ``$\subseteq$'' is straightforward. We will show the other inclusion.	
	When $s=1,$ we have that $c(h_{1})=\sum_{i=0}^{K}h^{i}_{1}u(i)$ for some $K\in\mathbb{N}$.
	Since the matrix $(j^i)_{0\leq i,j\leq K},$\footnote{ Recall that we have set $0^0\coloneqq 1.$} is (the transpose of) a Vandermonde matrix, its determinant is non-zero, so each $u(i)$ is a  linear combination of $c(0),\dots, c(K)$. Therefore,
	the conclusion holds for $s=1$.
	
	We now assume that the conclusion holds for some $s\geq 1$ and we prove it for $s+1$. Write \[c(h_{1},\dots,h_{s+1})=\sum_{a_{1},\dots,a_{s+1}\in\mathbb{N}}h^{a_{1}}_{1}\dots h^{a_{s+1}}_{s+1}\cdot u(a_{1},\dots,a_{s+1})=\sum_{i\in\mathbb{N}} h^{i}_{s+1} v_{i}(h_{1},\dots,h_{s})\]
	for some polynomials $v_{i}\colon\mathbb{Z}^{s}\to\mathbb{Q}^{d}$ given by
	$$v_{i}(h_{1},\dots,h_{s})=\sum_{a_{1},\dots,a_{s}\in\mathbb{N}}h^{a_{1}}_{1}\dots h^{a_{s}}_{s}\cdot u(a_{1},\dots,a_{s},i).$$
	Since the conclusion holds for $s=1,$ we have that for all $h_{1},\dots,h_{s}\in\mathbb{Z}$ and $i\in\mathbb{N}$, $v_{i}(h_{1},\dots,h_{s})\in \text{span}_\Q\{c(h_1,\ldots,h_s,h_{s+1}):\;h_{s+1}\in\mathbb{Z}\}$.  Applying the induction hypothesis for $s$, we have that 
	\[u(a_{1},\dots,a_{s},i)\in \text{span}_{\mathbb{Q}}\{v_{i}(h_{1},\dots,h_{s})\colon h_{1},\dots,h_{s}\in\mathbb{Z}\}\]
	for all $a_{1},\dots,a_{s},i\in\mathbb{N},$ hence the conclusion holds for $s+1$. By induction, the $L=1$ case is complete.

	For the general case,
	suppose that $\c(h_{1},\dots,h_{s})=(c_{1}(h_{1},\dots,h_{s}),\dots,c_{L}(h_{1},\dots,h_{s}))$ and $\u(a_{1},\dots,a_{s})$ $=(u_{1}(a_{1},\dots,a_{s}),\dots,u_{L}(a_{1},\dots,a_{s}))$, where $c_{i}\colon(\mathbb{Z}^{L})^{s}\to \mathbb{Q}^{d}$,  $u_{i}\colon(\mathbb{N}^{L})^{s}\to\mathbb{Q}^{d}$, $1\leq i\leq L$. Then
	\begin{equation}\label{23}
	\begin{split}
	c_{i}(h_{1},\dots,h_{s})=\sum_{a_{1},\dots,a_{s}\in\mathbb{N}^{L}}h^{a_{1}}_{1}\dots h^{a_{s}}_{s}\cdot u_{i}(a_{1},\dots,a_{s})
	\end{split}	
	\end{equation}
	for all $1\leq i\leq L$. By definition, one easily checks that
	\[\text{span}_{\mathbb{Q}}\{G(\c(h_{1},\dots,h_{s}))\colon h_{1},\dots,h_{s}\in\mathbb{Z}^{L}\}=\text{span}_{\mathbb{Q}}\{c_{i}(h_{1},\dots,h_{s})\colon h_{1},\dots,h_{s}\in\mathbb{Z}^{L}, 1\leq i\leq L\},\] and
	\[\text{span}_{\mathbb{Q}}\{G(\u(a_{1},\dots,a_{s}))\colon a_{1},\dots,a_{s}\in\mathbb{N}^{L}\}=\text{span}_{\mathbb{Q}}\{u_{i}(a_{1},\dots,a_{s})\colon a_{1},\dots,a_{s}\in\mathbb{N}^{L}, 1\leq i\leq L\}.\;\;\;\;\;\;\;\;\;\]
	So, it suffices to show that for every $1\leq i\leq L$,
	\begin{equation}\nonumber
	\begin{split}
	\text{span}_{\mathbb{Q}}\{c_{i}(h_{1},\dots,h_{s})\colon h_{1},\dots,h_{s}\in\mathbb{Z}^{L}\}=\text{span}_{\mathbb{Q}}\{u_{i}(a_{1},\dots,a_{s})\colon a_{1},\dots,a_{s}\in\mathbb{N}^{L}\},\;\;\text{or}
	\end{split}	
	\end{equation}	
	\begin{equation}\label{24}
	\begin{split}
	\text{span}_{\mathbb{Q}}\{c_{i}(h)\colon h\in\mathbb{Z}^{Ls}\}=\text{span}_{\mathbb{Q}}\{u_{i}(a)\colon a\in\mathbb{N}^{Ls}\},\;\;\;\;\;\;
	\end{split}	
	\end{equation}
	by viewing $(h_{1},\dots,h_{s})$ and $(a_{1},\dots,a_{s})$ as the $Ls$-dimensional vectors $h$ and $a$. Rewriting (\ref{23}) as 
	\begin{equation}\nonumber
	\begin{split}
	c_{i}(h)=\sum_{a\in\mathbb{N}^{Ls}}h^{a}\cdot u_{i}(a),
	\end{split}	
	\end{equation}
	we can apply the conclusion of the case $L'=1$, $s'=Ls$, $d'=d$ and $c_{i}\colon (\Z^{L'})^{s'}=(\Z^{L})^{s}\to (\mathbb{Z}^{d'})^{L'}=\Z^{d}$ to get (\ref{24}). This finishes the proof.
\end{proof}

\section{Equivalent conditions for $((T_{1}\times\dots\times T_{d})^{p(n)})_{n\in\mathbb{Z}^{L}}$ being ergodic} \label{s233}

In this short section, we provide  equivalent conditions to Property (ii) in \cref{t1}, \emph{i.e.}, we characterize when $((T_{1}\times\dots\times T_{d})^{p(n)})_{n\in\mathbb{Z}^{L}}$ is ergodic for $\mu^{\otimes d}$.

\medskip

The following lemma is an implication of \cite[Lemma~4.18]{Fu}.

\begin{lemma}\label{prod}
	Let $\X_{i}=(X_{i},\mathcal{B}_{i},\mu_{i},T_{i}),$ $1\leq i\leq d$ be $\mathbb{Z}$-systems. The set of eigenvalues of $T_{1}\times \dots\times T_{d}$ consists of all numbers of the form $\prod_{i=1}^{d}\lambda_{i}$, where $\lambda_{i}$ is either $1$ or an eigenvalue of $T_{i}$, where at least one $\lambda_{i}$ is an eigenvalue.
\end{lemma}

\begin{proof}
	Suppose first that $\lambda_{i}$ is either 1 or an eigenvalue of $T_{i}$ and that at least one $\lambda_{i}$ is an eigenvalue. Then, for all $1\leq i\leq d$, $T_{i}f_{i}=\lambda_{i}f_{i}$ for some $f_{i}\in L^{\infty}(\mu_{i})$, where not all $f_{i}$'s are $\mu_{i}$-a.e. constant. Then $(T_{1}\times \dots\times T_{d})(f_{1}\otimes\dots\otimes f_{d})=\left( \prod_{i=1}^{d}\lambda_{i}\right)(f_{1}\otimes\dots\otimes f_{d}).$ Since $f_{1}\otimes\dots\otimes f_{d}$ is not $(\mu_{1}\times\dots\times\mu_{d})$-a.e. constant, $\prod_{i=1}^{d}\lambda_{i}$ is an eigenvalue of $T_{1}\times \dots\times T_{d}$.
	
	Conversely, let $\lambda$ be an eigenvalue of $T_{1}\times \dots\times T_{d}$ with a corresponding eigenfunction $f$. 
	By \cite[Lemma~4.18]{Fu}, $f=\sum_{n}c_{n}f_{1,n}\otimes\dots\otimes f_{d,n}$, where $c_{n}\in\mathbb{C}$, $T_{i}f_{i,n}=\lambda_{i,n}f_{i,n}$ for some $\lambda_{i,n}\in\mathbb{S}^{1}$ with $\prod_{i=1}^{d}\lambda_{i,n}=\lambda$. Each $\lambda_{i,n}$ is either 1 or an eigenvalue of $T_{i}$.
	Since $f$ is not $(\mu_{1}\times\dots\times\mu_{d})$-a.e. constant, some $f_{1,n}\otimes\dots\otimes f_{d,n}$ is also not $(\mu_{1}\times\dots\times\mu_{d})$-a.e. constant. For such $n$, at least one of $\lambda_{1,n},\dots, \lambda_{d,n}$ is an eigenvalue of $T_{i}$.
	Note that if $f_{i,n}$ is $\mu_{i}$-a.e. constant, then $\lambda_{i,n}=1$. Otherwise $\lambda_{i,n}$ is an eigenvalue of $T_{i}$, which finishes the proof.
\end{proof}

Let $p\colon\mathbb{Z}^{L}\to\mathbb{Z}$ be a polynomial and $\lambda\in\mathbb{S}^{1}$. We say that $\lambda$ is \emph{uniform} for $p$ if $\mathbb{E}_{n\in\mathbb{Z}^{L}}\lambda^{p(n)}=0$. So, $\lambda=1$ is not uniform for any integer-valued polynomial, while by Weyl's equidistribution theorem, every $\lambda=e^{2\pi i a}$ for some $a\notin\mathbb{Q}$ is uniform for all integer-valued polynomials.

The following proposition, which lists conditions equivalent to Property (ii) of Theorem \ref{t1}, is the main result of the section.

\begin{proposition}[Conditions equivalent to (ii) of Theorem \ref{t1}]\label{234}
	Let $(X,\mathcal{B},\mu,T_{1},\dots,T_{d})$ be a system with commuting transformations and $p\colon\mathbb{Z}^{L}\to\mathbb{Z}$ be a polynomial. The following statements are equivalent:
	\begin{itemize}
		\item[(i)] $((T_{1}\times \dots\times T_{d})^{p(n)})_{n\in\mathbb{Z}^{L}}$ is ergodic for $\mu^{\otimes d}$.
		\item[(ii)]  Every eigenvalue of $T_{1}\times \dots\times T_{d}$ is uniform for $p$.
		\item[(iii)] For every $1\leq i\leq d$, if $\lambda_{i}$ is either 1 or an eigenvalue of $T_{i}$, where at least one $\lambda_{i}$ is an eigenvalue, then
		$\prod_{i=1}^{d}\lambda_{i}$ is uniform for $p$.
	\end{itemize}
\end{proposition}

\begin{proof}
	For convenience denote ${\bf Y}=(Y,\mathcal{D},\nu,T)=(X^{d},\mathcal{B}^{\otimes d},\mu^{\otimes d},T_{1}\times \dots\times T_{d})$.
	
	(i) $\Rightarrow$ (ii): Suppose that $\lambda$ is an eigenvalue of $T$. Let $f\in L^{\infty}(\nu)$ be a non-$\nu$-a.e. constant function such that $Tf=\lambda f$. By (i),
	\[0=\mathbb{E}_{n\in\mathbb{Z}^{L}} T^{p(n)}f=\mathbb{E}_{n\in\mathbb{Z}^{L}}\lambda^{p(n)} f.\]
	Since $f$ is not $\nu$-a.e. constant, $\mathbb{E}_{n\in\mathbb{Z}^{L}}\lambda^{p(n)}=0$ and so $\lambda$ is uniform for $p$.
	
	(ii) $\Rightarrow$ (i): It suffices to show that for all $f\in L^{\infty}(\nu)$ with $\int_{Y}f\,d\nu=0$, we have that 
	\[\mathbb{E}_{n\in\mathbb{Z}^{L}}T^{p(n)}f=0.\]
	By \cref{poly}, it follows that
	\[\mathbb{E}_{n\in\mathbb{Z}^{L}}T^{p(n)}f=\mathbb{E}_{n\in\mathbb{Z}^{L}}T^{p(n)}\mathbb{E}(f\vert Z_{T,T}({\bf{Y}})).\]
	By \cref{k}, we can approximate $\mathbb{E}(f\vert Z_{T,T}({\bf{Y}}))$ in $L^{2}(\nu)$ by finite linear combinations of eigenfunctions of $T$. So, we may assume without loss of generality that $\mathbb{E}(f\vert Z_{T,T}({\bf{Y}}))$ itself is an eigenfunction of $T$  and $T\mathbb{E}(f\vert Z_{T,T}({\bf{Y}}))=\lambda\mathbb{E}(f\vert Z_{T,T}({\bf{Y}}))$. 
	Since $\lambda$ is uniform for $p$,
	\[\mathbb{E}_{n\in\mathbb{Z}^{L}}T^{p(n)}\mathbb{E}(f\vert Z_{T,T}({\bf{Y}}))=\mathbb{E}_{n\in\mathbb{Z}^{L}}\lambda^{p(n)}\mathbb{E}(f\vert Z_{T,T}({\bf{Y}}))=0\]
	and we are done.
	
	(ii) $\Leftrightarrow$ (iii): This is a direct corollary of Lemma \ref{prod}. 	
\end{proof}

\section{PET induction}\label{s:3}

This section deals and explains the PET induction scheme, which is one of the main tools that we use in order to study expressions of the form \eqref{i1}, \eqref{i1*} and, more generally, \eqref{i1**}.\footnote{ For us, PET is an abbreviation for ``Polynomial Exhaustion Technique'' (PET also stands for ``Polynomial Ergodic Theorem'').} This technique was introduced by Bergelson (in the now classical \cite{B}) to study multiple averages for essentially distinct polynomials in weakly mixing systems and show the joint ergodicity property in that setting. His method used an inductive argument via  van der Corput lemma, reformulated in his setting, to reduce the ``complexity'' of the family of polynomials.   

Following this pivotal work of Bergelson, variations of the initial PET induction scheme were used to tackle more general cases, as the one in \cite{CFH} to deal with multiple, commuting $T_i$'s and ``nice'' families of polynomials, and in \cite{Jo} to deal with multiple, commuting, $T_i$'s and ``standard'' families of multi-variable polynomials, which we actually follow here too.  

The idea is the following: one runs the van der Corput lemma (vdC-operation) in some family of integer valued functions-sequences satisfying some special property and gets a family also satisfying the special property  but of lower ``complexity''. This allows one to run an inductive argument and arrive at a base case.  In our case the base case is when all the iterates are linear.

Of course, in all the different aforementioned cases, one has to do several technical variations in the method. In this paper for example, an essential detail is that whenever we talk about a polynomial with multiple variables, we always treat the first variable as a special one (see below for more details). Also, to the best of our knowledge, it is the first time that via the vdC-operations, while running (the variation of) the PET induction, we track down the coefficients of the polynomials (see Section \ref{s:pet}),  which is crucial for our arguments. 

\begin{definition*}
	For a polynomial $p(n;h_{1},\dots,h_{s})\colon(\mathbb{Z}^{L})^{s+1}\to\mathbb{Z}$, we denote with $\deg(p)$ \emph{the degree of $p$ with respect to $n$} (for example, for $s=1, L=2$, the degree of $p(n_{1},n_{2};h_{1,1},h_{1,2})=h_{1,1}h_{1,2}n_{1}^{2}+h_{1,1}^{5}n_{2}$ is 2).
	
	For a polynomial $p(n;h_{1},\dots,h_{s})=(p_{1}(n;h_{1},\dots,h_{s}),\dots,p_{d}(n;h_{1},\dots,h_{s}))\colon(\mathbb{Z}^{L})^{s+1}\to\mathbb{Z}^{d},$ we let $\deg(p)=\max_{1\leq i\leq d}\deg(p_{i})$ and we say that $p$ is \emph{essentially constant} if  $p(n;h_{1},\dots,h_{s})$ is independent of the variable $n$. We say that the polynomials $p, q \colon(\mathbb{Z}^{L})^{s+1}\to\mathbb{Z}^{d}$ are \emph{essentially distinct} if $p-q$ is not essentially constant,  and  \emph{essentially equal} otherwise.
	
	Actually, for a tuple $\q=(q_{1},\dots,q_{\ell})$ with polynomials
	$q_{1},\dots,q_{\ell}\colon(\mathbb{Z}^{L})^{s+1}\to\mathbb{Z}^{d}$, we let $\deg(\q)=\max_{1\leq i\leq \ell}\deg(q_{i})$. We say that $\q$ is \emph{non-degenerate} if $q_{1},\dots,q_{\ell}$ are all not essentially constant, and are pairwise  essentially distinct.\footnote{ The separation between using or not bold characters might look confusing in the beginning, it makes it clearer though when we use both vectors and vectors of vectors of polynomials.}
\end{definition*}

Fix a $\mathbb{Z}^{d}$-system $(X,\mathcal{B},\mu,(T_{g})_{g\in\mathbb{Z}^{d}})$.  
Let $q_{1},\dots,q_{\ell}\colon(\mathbb{Z}^{L})^{s+1}\to\mathbb{Z}^{d}$ be polynomials and $g_{1},\dots,g_{\ell}\colon$ $ X\times (\mathbb{Z}^{L})^{s}\to\mathbb{R}$ be functions such that each $g_{m}(\cdot;h_{1},\dots,h_{s})$ is an $L^{\infty}(\mu)$ function bounded by $1$  for all $h_{1},\dots,h_{s}\in\mathbb{Z}, 1\leq m\leq \ell$. For convenience, let  $\q=(q_{1},\dots,q_{\ell})$ and $\g=(g_{1},\dots,g_{\ell})$. We call $A=(L,s,\ell,\g,\q)$
a \emph{PET-tuple}, and for $\kappa\in \N$ we set

\begin{equation}\nonumber
\begin{split}
&S(A,\kappa)\coloneqq \overline{\mathbb{E}}^{\square}_{h_{1},\dots,h_{s}\in\mathbb{Z}^{L}}\sup_{\substack{  (I_{N})_{N\in\mathbb{N}} \\ \text{ F\o lner seq.} }}\varlimsup_{N\to \infty}\Bigl\Vert\mathbb{E}_{n\in I_{N}}\prod_{m=1}^{\ell}T_{q_{m}(n;h_{1},\dots,h_{s})}g_{m}(x;h_{1},\dots,h_{s})\Bigr\Vert^{\kappa}_{L^{2}(\mu)}.
\end{split}
\end{equation}

We define $\deg(A)=\deg(\q)$, and we say that $A$ is \emph{non-degenerate} if $\q$ is non-degenerate.
For any $f\in L^{\infty}(\mu)$, we say that $A=(L,s,\ell,\g,\q)$ is \emph{standard} for $f$ if there exists $1\leq m\leq \ell$ such that $\deg(A)=\deg(q_{m})$ and $g_{m}(x;h_{1},\dots,h_{s})=f(x)$ for every $x,h_1,\ldots,h_s$. That is, $f$ appears as one of the functions in $\g$, only depending on the first variable, and that the polynomial acting on $f$ is of the highest degree.  We say $A=(L,s,\ell,\g,\q)$ is \emph{semi-standard} for $f$ if there exists $1\leq m\leq \ell$ such that $g_{m}(x;h_{1},\dots,h_{s})=f(x)$ for every $x,h_1,\ldots,h_s$, which is similar to being standard, but we do not require the polynomial acting on $f$ to be of the highest degree.

For each PET-tuple $A=(L,s,\ell,\g,\q)$ and polynomial $q\colon(\mathbb{Z}^{L})^{s+1}\to\mathbb{Z}^{d}$, we define the \emph{vdC-operation}, $\partial_{q}A$, according to the following three steps:

\medskip

{\bf Step 1}:  For all $1\leq m\leq \ell$, let $g'_{m}=g'_{m+\ell}=g_{m},$ and  $q_1',\ldots,q_{2\ell}' \colon(\mathbb{Z}^{L})^{s+2}\to\mathbb{Z}^{d}$ be polynomials defined as   $$\displaystyle q'_m(n;h_1,\ldots,h_{s+1})=\left\{ \begin{array}{ll} q_m(n;h_1,\ldots,h_{s})-q(n;h_1,\ldots,h_{s}) & \; ,1\leq m\leq \ell\\ q_{m-\ell}(n+h_{s+1};h_1,\ldots,h_{s})-q(n;h_1,\ldots,h_{s})  & \; ,\ell+1\leq m\leq 2\ell\end{array} \right.,$$ \emph{i.e.}, we subtract the polynomial $q$ from the first $\ell$ polynomials and for the second $\ell$ ones we first shift by $h_{s+1}$ about the first variable and then we subtract $q$.

\medskip

{\bf Step 2}: 
We remove from $q'_{1}(n;h_{1},\dots,h_{s+1}),\dots,q'_{2\ell}(n;h_{1},\dots,h_{s+1})$ the polynomials which are essentially constant and the corresponding terms with those as iterates (this will be justified via the use of the Cauchy-Schwarz inequality and the fact that the functions $g_m$ are bounded), and then put the non-essentially constant ones in groups $J_{i}=\{q''_{i,1},\dots,q''_{i,t_{i}}\},$ $ 1\leq i\leq r$ for some $r,$ $t_{i}\in\mathbb{N}^\ast$ such that two polynomials are essentially distinct if and only if they belong to different groups. We now write $q''_{i,j}(n;h_{1},\dots,h_{s+1})=q''_{i,1}(n;h_{1},\dots,h_{s+1})+p''_{i,j}(h_{1},\dots,h_{s+1})$ for some polynomial $p''_{i,j}$ for all $1\leq j\leq t_{i},$ $1\leq i\leq r$. For convenience, we also relabel $g'_{1},\dots, g'_{2\ell}$ accordingly as $g''_{i,j}$ for all $1\leq j\leq t_{i},$ $1\leq i\leq r$. 

\medskip

{\bf Step 3}: 
For all $1\leq i\leq r$, 
let $q^{\ast}_{i}=q''_{i,1}$ and \[g^{\ast}_{i}(x;h_{1},\dots,h_{s+1})=g''_{i,1}(x;h_{1},\dots,h_{s+1})\prod^{t_{i}}_{j=2}T_{p''_{i,j}(h_{1},\dots,h_{s+1})}g''_{i,j}(x;h_{1},\dots,h_{s+1}).\]
Set $\q^{\ast}=(q^{\ast}_{1},\dots,q^{\ast}_{r})$, $\g^{\ast}=(g^{\ast}_{1},\dots,g^{\ast}_{r})$ and let this new PET-tuple be $\partial_{q}A=(L,s+1,r,\g^{\ast},\q^{\ast})$.
\footnote{ Here we  abuse the notation by writing  $\partial_{q}A$ to denote any of such operations obtained from Step 1 to 3. Strictly speaking, $\partial_{q}A$ is not uniquely defined as the order of grouping of $q'_{1},\dots,q'_{2\ell}$ in Step 2 is ambiguous. However, this is done without loss of generality, since  the order does not affect the value of $S(\partial_{q}A, \cdot)$. } 

In practice, the polynomial $q$ is some of the initial polynomials $q_1,\ldots, q_\ell.$ Therefore, if $q=q_{t}$ for some $1\leq t\leq\ell$, we write $\partial_{t}A$ instead of $\partial_{q_{t}}A$ to lighten the notation.

\medskip

We will use the previous notation and quantifiers for the vdC-operation from now on.

\medskip

The following important proposition informs us that, modulo some power and some constant which are unimportant for our purpose, the value of $S(\cdot, \cdot)$ grows by using the vdC-operation described above.

\begin{proposition}\label{induction}
	Let	$(X,\mathcal{B},\mu,(T_{g})_{g\in\mathbb{Z}^{d}})$ be a $\mathbb{Z}^{d}$-system, $A=(L,s,\ell,\g,\q)$ a PET-tuple, and $q\colon(\mathbb{Z}^{L})^{s+1}\to\mathbb{Z}^{d}$ a polynomial. Then $\partial_{q}A$ is  non-degenerate and $S(A,2\kappa)\leq 4^{\kappa} S(\partial_{q}A,\kappa)$ for every $\kappa \in \N$. 
\end{proposition}

\begin{proof}
	Since in Step 2 of the vdC-operation, essentially constant polynomials are removed and polynomials which are essentially the same are grouped together, we have that $\partial_{q}A$ is non-degenerate.

	On the other hand,
	we have that
	$S(A,2\kappa)$ equals to 
	
	\begin{flalign*}
	& \overline{\mathbb{E}}^{\square}_{h_{1},\dots,h_{s}\in\mathbb{Z}^{L}}\F\varlimsup_{N\to \infty}\Bigl\Vert\mathbb{E}_{n\in I_{N}}\prod_{m=1}^{\ell}T_{q_{m}(n;h_{1},\dots,h_{s})}g_{m}(x;h_{1},\dots,h_{s})\Bigr\Vert^{2\kappa}_{L^{2}(\mu)} & \\
	\leq	&   4^{\kappa} \overline{\mathbb{E}}^{\square}_{h_{1},\dots,h_{s+1}\in\mathbb{Z}^{L}}\F\varlimsup_{N\to \infty} \left\vert\mathbb{E}_{n\in I_{N}}\left \langle \prod_{m=1}^{\ell}T_{q_{m}(n;h_{1},\dots,h_{s})}g_m(x;h_{1},\dots,h_{s}),  \right. \right. \;\;\; \text{(by Lemma~\ref{lemma:iteratedVDC})}\\
	&\left. \left.  \hspace{4.3cm} \prod_{m=1}^{\ell}T_{q_{m}(n+h_{s+1};h_{1},\dots,h_{s})}g_m(x;h_{1},\dots,h_{s}) \right \rangle \right\vert^{\kappa}\\
	=	&   4^{\kappa} \overline{\mathbb{E}}^{\square}_{h_{1},\dots,h_{s+1}\in\mathbb{Z}^{L}}\F  \varlimsup_{N\to \infty}\left\vert \mathbb{E}_{n\in I_{N}}\left \langle \prod_{m=1}^{\ell}T_{{q}'_{m}(n;h_{1},\dots,h_{s},h_{s+1})}g'_m(x;h_{1},\dots,h_{s}),   \right. \right. \; \text{  \stackanchor{(invariance} {of $\mu$)}}\\
	&\left. \left.  \hspace{4.3cm} \prod_{m=1}^{\ell}T_{{q}'_{m+\ell}(n;h_{1},\dots,h_{s},h_{s+1})}g'_m(x;h_{1},\dots,h_{s}) \right \rangle \right\vert^{\kappa}  \;\;\;\;\;\;\;\;\;\;\;\; \text{(end of Step 1)}\\
	\leq &   4^{\kappa} \overline{\mathbb{E}}^{\square}_{h_{1},\dots,h_{s+1}\in\mathbb{Z}^{L}}\F\varlimsup_{N\to \infty}\Bigl\Vert  \mathbb{E}_{n\in I_{N}}\prod_{i=1}^{r}T_{q''_{i,1}(n;h_{1},\dots,h_{s+1})}\big( g''_{i,1}(x;h_{1},\dots,h_{s+1}) \cdot  \\
	&  \hspace{4.3cm} \prod_{j=2}^{t_{i}}T_{p''_{i,j}(h_{1},\dots,h_{s+1})}g''_{i,j}(x;h_{1},\dots,h_{s+1}) \big) \Bigr\Vert_{L^{2}(\mu)}^{\kappa} \;\;\;\;\;\;\;\;\; \text{ \stackanchor{(Cauchy-Schwarz}{and Step 2)} }\\
	=	& 4^{\kappa} \overline{\mathbb{E}}^{\square}_{h_{1},\dots,h_{s+1}\in\mathbb{Z}^{L}}\F\varlimsup_{N\to \infty}\Bigl\Vert\mathbb{E}_{n\in I_{N}}\prod_{i=1}^{r}T_{q^\ast_i(n;h_{1},\dots,h_{s+1})}g^{\ast}_{i}(x;h_{1},\dots,h_{s+1})\Bigr\Vert_{L^{2}(\mu)}^{\kappa} \;\;\;\;\; \text{ (Step 3)}, 
	\end{flalign*}		
	which is $4^{\kappa}S(\partial_{q}A,\kappa),$ completing the proof.\footnote{ Note that the last inequality is exactly the point where the removal of the terms with bounded iterates happens. All these terms are grouped together, while the rest are grouped into sets of non-essentially distinct polynomials according to, and following the notation of, Step 3. By applying the Cauchy-Schwarz inequality to those two terms, as the functions are assumed to be bounded by 1, we get the stated inequality.} 
\end{proof}

	The following theorem shows that when we start with a PET-tuple which is standard for a function, then, after finitely many vdC-operations, we arrive at a new PET-tuple of degree 1 which is still standard for the same function. This is useful because by \cite[Proposition~3.1]{Jo}, whenever we have an average with linear iterates, we can bound the limsup of the norm of the average by some Host-Kra seminorm of the functions. We caution the reader that in our method, we alternate this standard procedure and instead of deriving to linear iterates for ``some functions,'' we run the PET induction multiple times to arrive at linear iterates isolating ``each function'' separately.

\begin{theorem}\label{PET}
	Let	$(X,\mathcal{B},\mu,(T_{g})_{g\in\mathbb{Z}^{d}})$ be a $\mathbb{Z}^{d}$-system and $f\in L^{\infty}(\mu)$. If $A$ is a non-degenerate PET-tuple which is standard for $f$, then there exist $\rho_{1},\dots,\rho_{t}\in\mathbb{N}^\ast,$ for some $t\in \mathbb{N},$ such that $\partial_{\rho_{t}}\dots\partial_{\rho_{1}}A$ is a  non-degenerate PET-tuple which is standard for $f$ with $\deg(\partial_{\rho_{t}}\dots\partial_{\rho_{1}}A)=1$.
\end{theorem}

As an example to demonstrate how the method works, we present some computations for our Example~\ref{Eg:1}.

\medskip

\noindent {\bf First part of computations for Example~\ref{Eg:1}:} 
For a $\mathbb{Z}^{2}$-system $(X,\mathcal{B},\mu,(T_{g})_{g\in\mathbb{Z}^{2}})$ and $f_{1}, f_{2}\in L^{\infty}(\mu),$ the PET-tuple of Example~\ref{Eg:1}  is \[A=(1,0,2,(f_{1},f_{2}),(p_{1}, p_{2})),\] where $p_{1}(n)=(n^{2}+n,0)=(n^{2}+n)e_{1},$ $ p_{2}(n)=(0,n^{2})=n^{2}e_{2},$ for $e_{1}=(1,0)$ and $e_{2}=(0,1).$ 
 For $i=1$ and 2, we explain how to find a sequence of vdC-operations to reduce $A$ into a non-degenerate PET-tuple of degree 1 which is standard for $f_{i}$.

We first isolate the function $f_1.$ Setting  $e=(1,-1),$   we have 
$\partial_{2} A=(1,1,3,(f_{1},f_{1},f_{2}),\p_{1}),$
where the tuple $\p_{1}$ essentially equals  $$(n^{2}e+ne_{1}, n^{2}e+(2h_{1}+1)ne_{1}, 2h_{1}ne_{2})$$ (one term is removed because it is essentially constant and so $\ell=3$). Then
$\partial_{3}\partial_{2} A=(1,2,4, (f_{1},f_{1},$ $f_{1},f_{1}), \p_{2}),$
where the tuple $\p_{2}$ essentially equals

\noindent $((n^{2}+2h_{1}n)e+(1-2h_{1})ne_{1}, (n^{2}+2h_{1}n)e+ne_{1}, (n^{2}+2(h_{1}+h_{2})n)e+(1-2h_{1})ne_{1}, (n^{2}+2(h_{1}+h_{2})n)e+ne_{1})$
(two terms are removed because they are essentially constant and so $\ell=4$). Finally 
$\partial_{2}\partial_{3}\partial_{2} A=(1,3,7, (f_{1},\dots,f_{1}), \p_{3}),$
where the tuple $\p_{3}$ essentially equals  

\noindent $(-2h_{1}ne_{1},2h_{2}ne-2h_{1}e_{1},2h_{2}ne,2h_{3}ne-2h_{1}ne_{1},2h_{3}ne,2(h_{2}+h_{3})ne-2h_{1}ne_{1},2(h_{2}+h_{3})ne)$
(one term is removed because it is essentially constant and so $\ell=7$). We have that $\partial_{2}\partial_{3}\partial_{2} A$ is non-degenerate and standard for $f_{1}$ with $\deg(\partial_{2}\partial_{3}\partial_{2} A)=1$.

\

We continue by isolating $f_2$. Note that 
$\partial_{1} A=(1,1,3,(f_{2},f_{1},f_{2}),\p_{1}),$ 	where the tuple $\p_{1}$ essentially equals $$(-n^{2}e-ne_{1},2h_{1}ne_{1},-n^{2}e-ne_{1}+2h_{1}ne_{2})$$ (one term is removed for it is essentially constant and so $\ell=3$). 
Then
$\partial_{2}\partial_{1} A=(1,2,4, (f_{2},f_{2},f_{2},$ $f_{2}), \p_{2}),$
where the tuple $\p_{2}$ essentially equals 
$$(-n^{2}e-(2h_{1}+1)ne_{1}, -(n^{2}+2h_{1}n)e-ne_{1}, -(n^{2}+2h_{2}n)e-(2h_{1}+1)ne_{1}, -(n^{2}+2(h_{1}+h_{2})n)e-ne_{1})$$
(two terms are removed because they are essentially constant and so $\ell=4$).
Finally 
$\partial_{1}\partial_{2}\partial_{1} A=(1,3,7, (f_{2},\dots,f_{2}), \p_{3}),$
where the tuple $\p_{3}$ essentially equals 

\noindent $(2h_{1}ne_{2},-2h_{2}ne,-2h_{2}ne+2h_{1}ne_{2}, -2h_{3}ne,
-2h_{3}ne+2h_{1}ne_{2},-2(h_{2}+h_{3})ne,-2(h_{2}+h_{3})ne+2h_{1}ne_{2})$
(one term is removed because it is essentially constant and so $\ell=7$). We have that $\partial_{1}\partial_{2}\partial_{1} A$ is non-degenerate and standard for $f_{2}$ with $\deg(\partial_{1}\partial_{2}\partial_{1} A)=1$.

\begin{proof}[Proof of Theorem~\ref{PET}]
	We follow the ideas of the PET induction in \cite{Jo} and \cite{L9}.
	
	If $\deg(A)= 1,$ there is nothing to prove. So, we assume that $\deg(A)\geq 2,$ $A=(L,s,\ell,\g=(g_{1},\dots,g_{\ell}),\q=(q_{1},\dots,q_{\ell}))$, with $q_{i}=(q_{i,1},\dots,q_{i,d}),$ $1\leq i\leq \ell$, where each $q_{i,j}$ is a polynomial from $\mathbb{Z}^{s+1}$ to $\mathbb{Z}$. Recall that $\deg (q_{i})=\max_{1\leq j\leq d}\deg(q_{i,j})$. In this proof, we are thinking of $\q$ as an $\ell\times d$ matrix $(q_{i,j})_{1\leq i\leq \ell, 1\leq j\leq d}$ with polynomial entries.
	
	We say that $p,q\colon(\mathbb{Z}^{L})^{s+1}\to\mathbb{Z}$ are \emph{equivalent}, and we write that $p\sim q$, if $\deg(p)=\deg(q)$ and $\deg(p-q)<\deg(p)$; otherwise, we write $p \nsim q$. It is not hard to see that ``$\sim$'' defines an equivalence relation. Suppose that $\deg(\q)\leq D$. We define the \emph{column weight} of the column $j$ to be the vector $w_{j}(\q)=(w_{1,j}(\q),\dots,w_{D,j}(\q))$, where each $w_{k,j}(\q)$ is equal  to the number of equivalent classes in $\q$ of degree $k$ in the column $j$ (\emph{i.e.}, among $q_{1,j},\dots,q_{\ell,j}$). For two column weights $\v=(v_{1},\dots,v_{D})$ and $\v'=(v'_{1},\dots,v'_{D})$, we say that $\v<\v'$ if there exists $1\leq k\leq D$ such that $v_{k}<v'_{k}$ and $v_{k'}=v'_{k'}$ for all $k'>k$ (notice that we start comparing them from the last coordinate because this is the one associated to the highest degree).  Then, the set of weights and the set of column degrees are well-ordered sets.
	Putting this information about $\q$ in rows, we get the $D\times d$ matrix $w_{\q}=[w_{1}(\q),\dots,w_{d}(\q)]$ which we call the \textit{subweigth} of $\q$.
	
		Given a matrix $M$ (with polynomial entries), we define its \textit{k-reduction}, denoted by $R_k(M),$ to be the submatrix of $M$ obtained by only considering the rows whose first $k$ elements are $0,$ after discarding these $0$'s. 
		For instance, for the matrix \[M=\begin{pmatrix} 0  & p_1 & p_2 & p_3 \\ 
		p_4 & 0 & 0 & p_5 \\ 0 & 0 & p_6 & p_7 \\ 0 & 0 & 0 & p_8\end{pmatrix}\]
		where $p_1,\ldots,p_8$ are non-zero polynomials,  its $i$-reduction for $i=1,2,3,4$ is $\begin{pmatrix} p_1& p_2 & p_3 \\ 0 & p_6 & p_7 \\ 0 & 0 & p_8\end{pmatrix}$, $\begin{pmatrix} p_6 & p_7 \\ 0 & p_8 \end{pmatrix}$ , $\begin{pmatrix} p_8 \end{pmatrix}$ and $\emptyset$ respectively. By convention, the 0-reduction $R_0(M)$ is $M$ itself and the $k$-reduction for $k\geq \ell$ is $\emptyset$.
		
		We now define an order associated to matrices. The \textit{weight} of a  matrix $\q$  with polynomial entries, denoted by $W(\q)$, is the vector of the matrices $(w(R_0(\q)),w(R_1(\q)),\ldots,w(R_{\ell-1}(\q)))$, where $\ell$ is the number of columns of $\q$. 
		Given two polynomial matrices $\q$ and $\q'$, $\deg(\q ),\deg(\q' )\leq D$, we say that $W(\q')<W(\q )$ if there exist $1\leq J, K\leq \ell$ such that 
		\[ w_j(R_k(\q))=w_j(R_k(\q')) \text{ for all } j<J \text{  and all } k=0,\ldots,\ell-1;  \] 
		and 
		\[w_{J}(R_{k}(\q))=w_{J}(R_{k}(\q')) \text{ for all }   k=0,\ldots,K-1 \text{ and } w_{J}(R_{K}(\q))< w_{J}(R_{K}(\q')). \]	Under this order, the set of weights of matrices is well-ordered. 
For a PET-tuple  $A=(L,s,\ell,\g,\q)$, we define $W(A)=W(\q)$ to be the \emph{weight} of $A$.  
	
	\textbf{Claim:} Let	$A$  be a  non-degenerate PET-tuple which is standard for $f$ with $\deg(A)\geq 2$. There exists $1\leq \rho \leq \ell$ such that $\partial_{\rho} A$ is  non-degenerate and standard for $f$ with $W(\partial_{\rho} A)<W(A)$. 
	
	We first finish the proof of the theorem assuming that the claim holds. Let $A$ be a  non-degenerate PET-tuple which is standard for $f$ and $\deg(A)\geq 2$. 
	After using the claim finitely many steps, the decreasing chain $W(A)>W(\partial_{\rho_{1}}A)>W(\partial_{\rho_{2}}\partial_{\rho_{1}}A)>\dots$ will eventually terminate, so we will end up with a non-degenerate PET-tuple $\partial_{\rho_{t}}\dots\partial_{\rho_{1}}A$ which is standard for $f$, with $\deg(\partial_{\rho_{t}}\dots\partial_{\rho_{1}}A)=1$. This finishes the proof. 
	
	So it suffices to prove the claim.
	 Relabeling if necessary, we may assume without loss of generality that $g_{1}=f$ and $\deg(q_{1,1})=\deg(A)\geq 2$. Let $j_{0}\in \{0,\ldots,\ell\}$ be the smallest integer such that $R_{j_0+1}(\q)=\emptyset$. We choose $1\leq \rho\leq \ell$ in the following way:
	\begin{enumerate}
		\item Case that $j_0=0$. This case has three sub-cases. 
		\begin{enumerate} 
			
			\item If some $q_{i,1}\nsim q_{1,1}$, then let $\rho$ be the smallest 
			integer such that $q_{\rho,1}\nsim q_{1,1}$. 
			
			In this case, since  $q_{\rho,1}\nsim q_{1,1}$ and $A$ is standard for $f$, $\partial_{\rho}A$ is standard for $f$. Moreover, $w_{D,1}(\partial_{\rho}A)=w_{D,1}(A)-1$ and so $W(\partial_{\rho} A)<W(A)$. 
			
			\item If all $q_{1,1},\dots,q_{\ell,1}$ are equivalent and there exist $2\leq i\leq \ell,$ $1\leq j\leq d$ such that $q_{i,j}\nsim q_{1,j}$, and either $\deg(q_{i,j})$ or $\deg(q_{1,j})$ equals $\deg(\q)$, then let $\rho$ be the smallest integer such that there exists $1\leq j\leq d$ with $q_{\rho,j}\nsim q_{1,j}$, and either $\deg(q_{\rho,j})$ or $\deg(q_{1,j})$ equals  $\deg(\q)$. In this case, since  $q_{\rho,j}$ is not equivalent to $q_{1,j}$, and either $\deg(q_{\rho,j})$ or $\deg(q_{1,j})$ equals  $\deg(\q)$, $\partial_{\rho}A$ is standard for $f$.  Moreover, $w_{D,1}(\partial_{\rho}A)=0<w_{D,1}(A)$ and so $W(\partial_{\rho} A)<W(A)$. 
			
			\item If all $q_{1,1},\dots,q_{\ell,1}$ are equivalent, and for all $1\leq j\leq d$, either $\deg(q_{i,j})$ is $\deg(q_{1,j})$ for all $1\leq i\leq \ell$ or $\deg(q_{i,j})<\deg(\q)$ for all $1\leq i\leq \ell$, then let $\rho=\ell+1$.\footnote{ We leave it to the interested reader to check that (a), (b) and (c) cover all the possibilities in Case (i).}
			
			In this case, $\deg(\partial_{\rho}A)<\deg(A)$. Since $\deg(q_{1,1})\geq 2$, we have that 
			
			$\deg(q_{1,1}(n,h_{1},\dots,h_{s})-q_{1,1}(n+h_{s+1},h_{1},\dots,h_{s}))=\deg(q_{1,1})-1=\deg(\partial_{\rho}A)\geq 1.$
			
			 So $\partial_{\rho}A$ is standard for $f$. Moreover, $w_{D,1}(\partial_{\rho}A)=0<w_{D,1}(A)$ and so $W(\partial_{\rho} A)<W(A)$. 
	\end{enumerate}
		
		\item Case that $j_0>0$. Consider the reduction $R_{j_0}(\q)$ of the matrix $\q$.
		\begin{enumerate} 
			\item  Suppose that an entry of the first column of $R_{j_0}(\q)$ (which is of course an entry of the $j_0+1$ column of $\q$) is not equivalent to any other entry of the first column of $R_{j_0}(\q)$. Among such entries, let $\rho$ be the smallest index such that $q_{\rho,j_{0}+1}$ has minimal degree. 
			
			In this case, we have that $\partial_{\rho}A$ is standard for $f$. Moreover, \[w_{\deg(q_{\rho,j_{0}+1}),1}(\partial_{\rho}^{j_0}A)>w_{\deg(q_{\rho,j_{0}+1}),j_{0}}(R_{j_0}(\q)),\] where $\partial_{\rho}^{k}=\partial_\rho\ldots \partial_\rho$ ($k$ times). One can check that this implies that $W(\partial_{\rho} A)<W(A)$.
			
			\item Suppose all entries in the first column of $R_{j_0}(\q)$ are equivalent. Then let $\rho$ be such  that $q_{\rho,j_0+1}$ corresponds to the first entry of the first column of $R_{j_0}(\q)$. 
			
			In this case, $\partial_{\rho}A$ is standard for $f$. Moreover, $$w_{\deg(q_{\rho,j_{0}+1}),1}(\partial_{\rho}^{j_0}A )>w_{\deg(q_{\rho,j_{0}+1}),j_{0}}(R_{j_0}(\q)).$$  One can check that this fact implies that $W(\partial_{\rho} A)<W(A)$.
		\end{enumerate} 	
		
	\end{enumerate}

	This proves the claim and completes the proof.
\end{proof}

We now provide a  proof of Proposition \ref{poly}.

\begin{proof}[Proof of Proposition \ref{poly}]
	Let $A=(L,0,1,\{f\},\{p\})$. It suffices to show that $S(A,\kappa)=0$  for some $\kappa \in \N$, assuming that $\mathbb{E}(f\vert Z_{\Z,\Z}(\X))=0$. For any $s\in\mathbb{N}^\ast$ and function $u\colon(\mathbb{Z}^{L})^{s}\to\mathbb{Z}$, let $\Delta u\colon(\mathbb{Z}^{L})^{s+1}\to\mathbb{Z}$ be the function $\Delta u(x_{1},\dots,x_{s+1})=u(x_{1}+x_{s+1},\dots,x_{s})-u(x_{1},\dots,x_{s})$ and $\Delta^{k}u=(\Delta\circ \dots\circ \Delta) u$ ($k$ times).
	
	If $\deg(p)>1,$ then it is easy to verify that $\partial_{1} A=(L,1,1,\{f\},\{\Delta p\})$. By induction, $\partial^{k}_{1} A=(L,k,1,\{f\},\{\Delta^{k} p\})$ for all $k<\deg (p)$. By Proposition \ref{induction}, we have that 
	$S(A,2^{K})\leq 4^{2^K-1} S(\partial_{1}^{K}A,1)$, where $K=\deg(p)-1$. It is easy to see that $\deg(\Delta p)=\deg(p)-1,$\footnote{ Recall that ``$\deg$'' only ``sees'' the first variable.} and so $\deg(\Delta^{K} p)=1$. We may then assume that $\Delta^{K} p(n,h_{1},\dots,h_{K})=c(h_{1},\dots, h_{K})\cdot n+c'(h_{1},\dots,h_{K})$ for some polynomials $c(h_{1},\dots, h_{K})\in \mathbb{Z}^{L}, c'(h_{1},\dots, h_{K})\in\mathbb{Z}$ of $h_1,\ldots,h_K$ with  $c$ not being the constant zero vector.
	By \cref{erg}, 
	\begin{equation}  \label{equation:PolyCharFactor} \E_{n\in \mathbb{Z}^{L}}T_{\Delta^{K} p(n,h_{1},\dots,h_{K})}f=T_{c'(h_1,\ldots,h_K)}\mathbb{E}(f\vert \I(G(c(h_{1},\dots,h_{K})))).\end{equation}
	If $c(h_{1},\dots, h_{K})\neq 0$, then \[ \I(G(c(h_{1},\dots,h_{K})))=Z_{G(c(h_{1},\dots,h_{K}))}\subseteq Z_{\mathbb{Z},G(c(h_{1},\dots,h_{K}))}=Z_{\mathbb{Z},\mathbb{Z}},\footnote{ Note that one cannot conclude that $\I(G(c(h_{1},\dots,h_{K})))=Z_{G(c(h_{1},\dots,h_{K}))}=Z_{\mathbb{Z}}$ because Lemma \ref{replacement0} (iv) is invalid for $d=1$.}\]
	where in the last equality we used Lemma~\ref{replacement0} (iv), since $G(c(h_{1},\dots,h_{K}))$ is a finite index subgroup of $\Z$. By Lemma~\ref{ag}, the set of $(h_{1},\dots,h_{K})\in(\mathbb{Z}^{L})^{K}$ such that $c(h_{1},\dots, h_{K})= 0$ is of upper Banach density $0,$ so
	\begin{align*}
	S(\partial_{1}^{K}A,1)= &\E^{\square}_{h_{1},\dots,h_{K}\in\Z^{L}}\F\vl\Bigl\Vert \mathbb{E}_{n\in I_{N}}T_{\Delta^{K} p(n,h_{1},\dots,h_{K})}f\Bigr\Vert_{L^{2}(\mu)} \\
	=& \E^{\square}_{h_{1},\dots,h_{K}\in\Z^{L}}\F\vl\Bigl\Vert \mathbb{E}_{n\in I_{N}}T_{\Delta^{K} p(n,h_{1},\dots,h_{K})}\mathbb{E}(f\vert Z_{\mathbb{Z},\mathbb{Z}})\Bigr\Vert_{L^{2}(\mu)}=0.
	\end{align*}
	
	This implies that  $S(A,2^{K})=0$, which finishes the proof.
\end{proof}	

\section{Characterizing multiple averages along polynomials}\label{s:4}

In this section we state Theorem~\ref{T:2}, the stronger form of Theorem \ref{T:3}, which is the main contribution of this work. Its validity implies (see below) both Theorems~\ref{t2} and ~\ref{t1}, our main joint ergodicity results.

\subsection{Characteristic factors for multiple averages}
Recall that a family of (integer valued) polynomials  $p_{1},\dots,p_{k}\colon\mathbb{Z}^{L}$ $\to\mathbb{Z}^{d}$ is non-degenerate if $p_{i}, p_{i}-p_{j}$ are not essentially constant for all $1\leq i, j\leq k$, $i\neq j$.  The following theorem states that in order to study multiple averages along polynomials, it suffices to assume that all the functions $f_{i}$ are measurable with respect to certain Host-Kra characteristic factors.

\begin{theorem}[Characteristic factors for multiple averages along polynomials]\label{T:2}
	Let $d,k,L\in\mathbb{N}^\ast$
	and $p_{1},\dots,p_{k}\colon\mathbb{Z}^{L}\to\mathbb{Z}^{d}$ be a non-degenerate family of polynomials of degree at most $K.$ Suppose that $p_{i}(n)=\sum_{v\in\mathbb{N}^{L},\vert v\vert\leq K}b_{i,v}n^{v}$ for some $b_{i,v}\in \mathbb{Q}^{d}$. Let $R\subseteq\mathbb{Q}^{d}$ be the set
	\[R \coloneqq \bigcup_{v\in\mathbb{N}^{L}, 0<\vert v\vert\leq K}\{b_{i,v}, b_{i,v}-b_{i',v}\colon 1\leq i, i'\leq k\}\backslash\{{\bf 0}\}.\]
	Let  $(X,\mathcal{B},\mu, (T_{g})_{g\in\mathbb{Z}^{d}})$ be a $\mathbb{Z}^{d}$-system.
	For every $f_{1},\dots,f_{k}\in L^{\infty}(\mu)$, 
	 we have that
	\begin{equation}\label{12}
		\mathbb{E}_{n\in\mathbb{Z}^{L}} T_{p_{1}(n)}f_1\cdot\ldots\cdot T_{p_{k}(n)}f_k=0
		\text{ if $\mathbb{E}(f_i\vert Z_{\{G(r)^{\times \infty}\}_{r\in R}})=0$ for some $1\leq i\leq k$.}
	\end{equation}
	In particular, if $(T_{g})_{g\in G(r)}$ is ergodic for $\mu$ for all $r\in R$, then for every $f_{1},\dots,f_{k}\in L^{\infty}(\mu)$, 
	\begin{equation}\label{413}
		\mathbb{E}_{n\in\mathbb{Z}^{L}} T_{p_{1}(n)}f_1\cdot\ldots\cdot T_{p_{k}(n)}f_k=0 \text{ if $\mathbb{E}(f_i\vert Z_{(\mathbb{Z}^{d})^{\times \infty}})=0$ for some $1\leq i\leq k$.}
	\end{equation}
\end{theorem}

\begin{remark*}
	The following weaker form of (\ref{12}) in Theorem \ref{T:2} can be derived by the results of \cite{Jo}: 
	\begin{equation}\nonumber	
			\mathbb{E}_{n\in\mathbb{Z}^{L}} T_{p_{1}(n)}f_1\cdot\ldots\cdot T_{p_{k}(n)}f_k=0 \text{ if  $\mathbb{E}(f_i\vert Z_{\{G(r)^{\times \infty}\}_{r\in \Z^{d}\setminus \{\bf 0 \}   }})=0$ for some $1\leq i\leq k$.}
	\end{equation}
	Hence, (\ref{413}) holds if $T_{g}$ is assumed to be ergodic for $\mu$ for all $g\in\mathbb{Z}^{d}\backslash\{{\bf{0}}\}$. Theorem~\ref{T:2} improves the result of \cite{Jo} since one only needs to require finitely many $T_{g}$'s to be ergodic (\emph{i.e.}, the generators of $G(r), r\in R$) in order to deduce (\ref{413}).
	On the other hand, it is worth noting that (\ref{12}) has room for improvement (meaning that it is possible for one to replace the factor $Z_{\{G(r)^{\times \infty}\}_{r\in R}}$  of (\ref{12}) with smaller ones), as we shall see in the examples below. Actually, we do have a stronger version of (\ref{12})  (see the proof of Theorem~\ref{T:2}), but (\ref{12}) already captures the essence of our result as it is stated here.
\end{remark*}

Another important example of polynomial averages is the following, for which we actually characterize its convergence to the ``expected'' limit (in Theorem~\ref{t1}), where all the transformations have the same polynomial iterate.  

\begin{example}\label{ex1}
	Let $(X,\mathcal{B},\mu,T_{1},\dots,T_{d})$ be a system with commuting transformations. One should think of $T_{1},\dots,T_{d}$ as a $\mathbb{Z}^{d}$-action $(S_{g})_{g\in\mathbb{Z}^{d}}$ with $T_{i}=S_{e_{i}}$, where we recall that $e_{i}\in\mathbb{Z}^{d}$ denotes the vector whose $i$th entry is 1 and all other entries are $0$. Let $p_{1},\dots,p_{d}\colon\mathbb{Z}\to\mathbb{Z}^{d}$ be polynomials given by $p_{i}(n)=p(n)e_{i}$ for some  polynomial $p\colon\mathbb{Z}\to\mathbb{Z}$. By Theorem \ref{T:2}, we have that
	\begin{equation}\label{43}
		\mathbb{E}_{n\in\Z} T_{1}^{p(n)}f_1\cdot\ldots\cdot T_{d}^{p(n)}f_{d}=0 \text{ if $\mathbb{E}(f_i\vert Z_{\{G(r)^{\times \infty}\}_{r\in R}})=0$ for some $1\leq i\leq d$,}
	\end{equation}
	where $R=\{T_{i}, T_{i}T^{-1}_{j}\colon 1\leq i,j\leq d, i\neq j\}$.
	We remark that $Z_{\{G(r)^{\times \infty}\}_{r\in R}}$ is not necessarily the smallest factor with this property. For example, if $p(n)=n$, then (\ref{43}) is a weaker form of \cref{P:1}  (or \cite[Proposition~1]{Ho}).
\end{example}

\noindent {\bf Continuation of Example~\ref{Eg:1}.}
Recall the $\mathbb{Z}^{2}$-system $\X$ with two commuting transformations $T_{1},T_{2}$ and $p_{1},p_{2}\colon\mathbb{Z}\to\mathbb{Z}^{2}$ polynomials given by $p_{1}(n)=(n^{2}+n,0)$ and $p_{2}(n)=(0,n^{2})$. By Theorem~\ref{T:2}, we have that
\begin{equation}\label{40}
	\begin{split}
		\quad\mathbb{E}_{n\in\Z} T_{1}^{n^{2}+n}f_1\cdot T_{2}^{n^{2}}f_2
		=0 \text{ if $\mathbb{E}(f_i\vert Z_{\{G(r)^{\times \infty}\}_{r\in R}})=0$ for $i=1$ or 2,}
	\end{split}
\end{equation}
where $R=\{T_{1},T_{2},T_{1}T^{-1}_{2}\}$.
Again $Z_{\{G(r)^{\times \infty}\}_{r\in R}}$ is not the smallest factor with this property (later, in  equality (\ref{400}), we will obtain an improvement of (\ref{40})).

It is an interesting, in general open (and definitely hard), question to ask what are the smallest factors $Z_{1},\dots, Z_{k}$ of $\X$ such that for every $f_{1},\dots,f_{k}\in L^{\infty}(\mu)$,
\begin{equation}\nonumber
	\mathbb{E}_{n\in\Z} T_{p_{1}(n)}f_1\cdot\ldots\cdot T_{p_{k}(n)}f_k=0 \text{ if $\mathbb{E}(f_i\vert Z_{i})=0$ for some $1\leq i\leq k$.}
\end{equation}

\subsection{Proofs of the joint ergodicity results assuming Theorem \ref{T:2}}\label{s5.2}

In this subsection we explain how to derive our main joint ergodicity results, Theorems~\ref{t2} and \ref{t1}, assuming the validity of Theorem~\ref{T:2}. To this end, we recall an adapted from \cite{F2021} definition from \cite{AA}.

\begin{definition*}[\cite{AA}]
We say that a collection of mappings $a_1,\ldots,a_k\colon \Z^d \to \Z^d$ is:  
\begin{enumerate}
    \item[(i)] \emph{good for seminorm estimates for the system $(X,\mathcal{B},\mu, (T_g)_{g\in \Z^d})$ along a F\o lner sequence $(I_N)_{N\in\N}$ in $\Z^d$}, if there exists $M\in \N^\ast$ such that if $f_1,\ldots,f_k\in L^\infty(\mu)$ and $\norm{f_\ell}_{(\Z^d)^{\times M}}=0$ for some $\ell\in\{1,\ldots,k\},$ then 
    \begin{equation*}
        \lim_{N\to\infty}\frac{1}{|I_N|}\sum_{n\in I_{N}}\prod_{i=1}^\ell T_{a_{i}(n)}f_i=0,
    \end{equation*} where the convergence takes place in $L^2(\mu).$ 
    \item[(ii)] \emph{good for equidistribution for the system $(X,\mathcal{B},\mu, (T_g)_{g\in \Z^{d}})$ along a F\o lner sequence $(I_N)_{N\in\N}$ in $\Z^{d}$}, if for all $\alpha_1,\ldots,\alpha_k\in \text{Spec}\left((T_{g})_{g\in \Z^{d}}\right),$ not all of them trivial, we have 
    \begin{equation*}
\lim_{N\to\infty}\frac{1}{|I_N|}\sum_{n\in I_{N}}\exp(\alpha_{1}(a_{1}(n))+\dots+\alpha_{k}(a_{k}(n)))=0,       
    \end{equation*}
    where $$\text{Spec}\left((T_{g})_{g\in\Z^{d}}\right)\coloneqq \{\alpha\in \text{Hom}(\Z^{d},\R/\Z)\colon T_{g}f=\exp(\alpha(g))f, \text{ $g\in \Z^{d},$ for some non-zero } f\in L^{2}(\mu)\},$$ and $\exp(x)\coloneqq e^{2\pi i x}$ for all  $x\in\R.$ 
\end{enumerate}
\end{definition*}

\begin{proof}[Proof of Theorem~\ref{t2} assuming  Theorem~\ref{T:2}]
	Let $R$ be defined as in Theorem \ref{T:2}. As $(T_{g})_{g\in G(r)}$ is ergodic for all $r\in R$, by Theorem \ref{T:2}, we may assume without loss of generality that all $f_{1},\dots, f_{k}$ are measurable with respect to $Z_{(\mathbb{Z}^{d})^{\times\infty}}(\X)$ (note that conditions (i) and (ii) remain valid when passing to a factor system). By $L^{1}(\mu)$-approximation, we may assume without loss of generality that all $f_{1},\dots, f_{k}$ are measurable with respect to $Z_{(\mathbb{Z}^{d})^{\times M}}(\X)$ for some $M\in\mathbb{N}$. By Theorem~\ref{ppq} and again by $L^{1}(\mu)$-approximation, we may further assume without loss of generality that all $f_{1},\dots, f_{k}$ are measurable with respect to a factor of $\X$ which is isomorphic to an $(M-1)$-step $\mathbb{Z}^{d}$-nilsystem.

Fix a F\o lner sequence $(I_{N})_{N\in\N}$ in $\Z^{L}$. 
We wish to show that 
$$\lim_{N\to\infty}\frac{1}{|I_N|}\sum_{n\in I_{N}}\prod_{i=1}^k T_{p_{i}(n)}f_i=\prod_{i=1}^{k}\int_{X}f_{i}\,d\mu$$
for all $f_{1},\dots,f_{k}\in L^{\infty}(\mu)$.
 
We first consider the case $L=d$ in order to use \cite[Theorem~3.9]{AA}. To this end, it suffices to show that $(p_{1},\dots,p_{k})$ is good for seminorm estimates and good for equidistribution. 
 
 If one of $f_{i}$ is constant zero, then there is nothing to prove, so we assume that no $f_{i}$ is identically equal to zero.
Since $\X$ is isomorphic to an $(M-1)$-step $\mathbb{Z}^{d}$-nilsystem, we have that $\nnorm{f_{i}}_{(\Z^{d})^{\times M}}\neq 0,$ which implies that 
 $(p_{1},\dots,p_{k})$ is good for seminorm estimates. 
 
Suppose, for the sake of contradiction, that $(p_{1},\dots,p_{k})$ is not good for equidistribution. Then there exist $\alpha_{1},\dots,\alpha_{k}\in \text{Spec}\left((T_{g})_{g\in\Z^{d}}\right)$, not all of them trivial, such that 
 \begin{equation}\label{E:d=L}
     \lim_{N\to\infty}\frac{1}{|I_N|}\sum_{n\in I_{N}}\exp(\alpha_{1}(p_{1}(n))+\dots+\alpha_{k}(p_{k}(n)))=c
 \end{equation}
 for some $c\neq 0$.\footnote{We remark that the limit on the left hand side of (\ref{E:d=L}) exists by \cite[Theorem~A]{L05}. %To see this, denote $q(n):=\alpha_{1}(p_{1}(n))+\dots+\alpha_{k}(p_{k}(n))$. If $q\notin \mathbb{Q}[n]+\mathbb{C}$, then the limit is zero by \cite{Weyl}. If $q\in \mathbb{Q}[n]+\mathbb{C}$, then $(\exp(q(n)))_{n\in\Z^{L}}$ is a periodic sequence and so the limit also exists.
 }
 For $1\leq i\leq k$,
 since $\alpha_{i} \in \text{Spec}\left((T_{g})_{g\in\Z^{d}}\right)$, there exists some nonzero $f_{i}\in L^{2}(\mu)$ such that $T_{g}f_{i}=\exp(\alpha_{i}(g))f_{i}$ for all $g\in\Z^{d}$. Since $\X$ is ergodic, we have that $\vert f_{i}\vert$ is a non-zero constant $\mu$-a.e. Using \eqref{E:d=L}, we have
 \[
 \lim_{N\to\infty}\frac{1}{|I_N|}\sum_{n\in I_{N}}\bigotimes_{i=1}^k T_{p_{i}(n)}f_i=\lim_{N\to\infty}\frac{1}{|I_N|}\sum_{n\in I_{N}}\bigotimes_{i=1}^k \exp(\alpha_{i}(p_{i}(n)))f_i=c\bigotimes_{i=1}^k f_{i}\not\equiv 0.
 \]
 On the other hand, since at least one of $\alpha_{1},\dots,\alpha_{k}$ is non-trivial, we have that $\int_{X^{k}}\bigotimes_{i=1}^{k}f_{i}\,d\mu^{\otimes k}$ $=\prod_{i=1}^{k}\int_{X}f_{i}\,d\mu=0$, which contradicts condition (ii). Therefore, $(p_{1},\dots,p_{k})$ is good for equidistribution.
 
 \
 
    Assume now that $L<d$. Let $(I'_{N})_{N\in\N}$ be the  F\o lner sequence in $\Z^{d}$ given by $I'_{N}:=I_{N}\times [-N,N]^{d-L}$.
    Let $p'_{1},\dots,p'_{k}\colon\Z^{d}\to\Z^{d}$ be polynomials given by
    $p'_{i}(n,m):=p_{i}(n)$ for all $n\in\Z^{L}$ and $m\in \Z^{d-L}$. 
    Let $R'$ be the set defined in (\ref{R}) but associated with the polynomials $p'_{1},\dots,p'_{k}$.
    It is hot hard to see that $R=R'$. Moreover, since $(T_{p_{1}(n)}\times\dots\times T_{p_{k}(n)})_{n\in\Z^{L}}$ is ergodic for $\mu^{\otimes k}$, so is $(T_{p'_{1}(n)}\times\dots\times T_{p'_{k}(n)})_{n\in\Z^{d}}$. 
    
    By the $d=L$ case, we have that
   $$\lim_{N\to\infty}\frac{1}{|I_N|}\sum_{n\in I_{N}}\prod_{i=1}^k T_{p_{i}(n)}f_i=\lim_{N\to\infty}\frac{1}{|I'_N|}\sum_{n\in I'_{N}}\prod_{i=1}^k T_{p'_{i}(n)}f_i=\prod_{i=1}^{k}\int_{X}f_{i}\,d\mu.$$

   Finally, we assume that $L>d$. Let $(S_{(n,m)})_{n\in\Z^{d},m\in\Z^{L-d}}$ be a $\Z^{L}$-action on $(X,\mathcal{B},\mu)$ such that $S_{(n,m)}=T_{n}$. Let $p'_{1},\dots,p'_{k}\colon\Z^{L}\to\Z^{L}$ be polynomials given by
    $p'_{i}(n):=(p_{i}(n),0,\dots,0)$ for all $n\in\Z^{L}$, where the last $L-d$ entries are zero.  
    Let $R'$ be the set defined in (\ref{R}) but associated with the polynomials $p'_{1},\dots,p'_{k}$. By definition, the set $R'$ consists of elements of the form $(r,0,\dots,0)\in \Z^{L}, r\in R$ (with respect to the $\Z^{L}$-system $(X,\mathcal{B},\mu,(S_{g})_{g\in\Z^{L}})$). By construction of $(S_{g})_{g\in\Z^{L}}$, the 
    ergodicity of $(T_{g})_{g\in G(r)}$ with respect to the $\Z^{d}$-system $(X,\mathcal{B},\mu,(T_{g})_{g\in\Z^{d}})$ implies ergodicity of the 
    ergodicity of $(S_{g})_{g\in G(r,0,\dots,0)}$ with respect to the $\Z^{L}$-system $(X,\mathcal{B},\mu,(S_{g})_{g\in\Z^{L}})$ for all $r\in R$. Moreover, since $S_{p'_{i}(n)}=S_{(p_{i}(n),0,\dots,0)}=T_{p_{i}(n)}$ for all $n\in \Z^{L}$, we have that $(S_{p'_{1}(n)}\times\dots\times S_{p'_{k}(n)})_{n\in\Z^{L}}=(T_{p_{1}(n)}\times\dots\times T_{p_{k}(n)})_{n\in\Z^{L}}$ is ergodic for $\mu^{\otimes k}$. 
    
    By the $d=L$ case, we have that
   \[\lim_{N\to\infty}\frac{1}{|I_N|}\sum_{n\in I_{N}}\prod_{i=1}^k T_{p_{i}(n)}f_i=\lim_{N\to\infty}\frac{1}{|I_N|}\sum_{n\in I_{N}}\prod_{i=1}^k S_{p'_{i}(n)}f_i=\prod_{i=1}^{k}\int_{X}f_{i}\,d\mu,
   \]
   which completes the proof.    
\end{proof}

Before we proceed with the proof of Theorem~\ref{t1}, we need the following lemma and proposition:

\begin{lemma} \label{lemma:product kroneckers}
Let	$(X,\mathcal{B},\mu, T_{1},\dots,T_{d})$ be a system with commuting transformations. Then in the product space $(X^{d},\mathcal{B}^d,\mu^{\otimes d})$, the $\sigma$-algebra of $T_1\times \cdots\times T_d$-invariant sets is measurable with respect to $\bigotimes_{i=1}^{d} Z_{T_i,T_i}$.
\end{lemma}

\begin{proof}
It suffices to show that $\mathbb{E}(f_1\otimes \cdots \otimes f_d \vert \I(T_1\times\cdots \times T_d))=0$ whenever $\mathbb{E}(f_i \vert Z_{T_i,T_i})=0$ for some $1\leq i \leq d$.  As usual, we assume without loss that all functions are bounded by $1$  in $L^{\infty}(\mu)$ and that 
$\mathbb{E}(f_1 \vert Z_{T_1,T_1})=0$ (or equivalently $\|f_1 \|_{T_1,T_1}=0$). By \cref{lemma:VdCclassic} and Jensen's inequality, setting $a(n)=T_1^n f_1 \otimes \cdots \otimes T_d^n f_d,$ we have that 
\begin{align*}
\| \mathbb{E}(f_1\otimes \cdots \otimes f_d \vert \I(T_1\times\cdots \times T_d)) \|^{4}_{L^{2}(\mu^{\otimes d})} &= \| \mathbb{E}_{n\in \Z} (T_1\times \cdots \times T_d)^{n} f_1\otimes \cdots \otimes f_d \|^{4}_{L^{2}(\mu^{\otimes d})} \\ 
& \leq \left ( 4  \mathbb{E}^{\square}_{h\in \Z}  | \mathbb{E}_{n\in \Z} \langle a(n) ,  a(n+h)  \rangle  \right |)^{2} \\
& \leq 16  \mathbb{E}^{\square}_{h\in \Z}   \mathbb{E}_{n\in \Z} \langle  f_1 \otimes \cdots \otimes  f_d ,   T_1^{h} f_1 \otimes \cdots \otimes T_d^{h} f_d  \rangle^2 \\
& \leq  16  \mathbb{E}^{\square}_{h\in \Z} \left \vert \int  f_1 \cdot  T_1^{h} f_1  d\mu  \right \vert^{2} \\
& =  16  \mathbb{E}^{\square}_{h\in \Z} \left \vert \int  \mathbb{E}(f_1 \cdot  T_1^{h} f_1 \vert  \I(T_1)) d\mu  \right  \vert^2 \\
& \leq  16  \mathbb{E}^{\square}_{h\in \Z} \left \| \mathbb{E}(f_1 \cdot  T_1^{h} f_1 \vert  \I(T_1) ) \right \|_{L^2(\mu)}^{2} \\
& =16 \|f_1  \|_{T_1,T_1}^{4},
\end{align*} 
where the last line follows, for instance, from \cref{replacement0} (iii). This finishes the proof. 
\end{proof}

\begin{proposition}\label{inverse}
	Let $d,L\in\mathbb{N}^\ast,$ $p\colon\mathbb{Z}^{L}\to\mathbb{Z}$  a polynomial and $(X,\mathcal{B},\mu, T_{1},\dots,T_{d})$ a system with commuting transformations such that $(T_{1}^{p(n)},\dots,T_{d}^{p(n)})_{n\in\mathbb{Z}^{L}}$ is jointly ergodic for $\mu$. Then
	\begin{itemize}
		\item[(i)]  $((T_{i}T_{j}^{-1})^{p(n)})_{n\in\mathbb{Z}^{L}}$ is ergodic for $\mu$ for all $1\leq i,j\leq d, i\neq j$; and
		\item[(ii)] $(T_{1}^{p(n)}\times\dots\times T_{d}^{p(n)})_{n\in\mathbb{Z}^{L}}$ is ergodic for $\mu^{\otimes d}$. 
	\end{itemize}
\end{proposition}

\begin{proof}
	The idea of the proof for Part~(i) is similar to \cite[Proposition~2.1]{BD}. Since the language we use is different, we present the proof for completeness.
	
	By assumption,
	\begin{equation}\label{t02}
		\mathbb{E}_{n\in\mathbb{Z}^{L}}T^{p(n)}_{1}f_{1}\cdot\ldots\cdot T^{p(n)}_{d}f_{d}=\int_{X}f_{1}\,d\mu\cdot\ldots\cdot \int_{X}f_{d}\,d\mu
	\end{equation}
	for all $f_{1},\dots,f_{d}\in L^{\infty}(\mu)$. Suppose first that (i) fails. We may assume without loss of generality that $((T_{1}T^{-1}_{2})^{p(n)})_{n\in\mathbb{Z}^{L}}$ is not ergodic for $\mu$. So there exist $g\in L^{\infty}(\mu)$ not $\mu$-a.e. equal to a constant function and a function $g'\in L^{2}(\mu)$  such that
	\[g'\coloneqq \mathbb{E}_{n\in\mathbb{Z}^{L}}(T_{1}T^{-1}_{2})^{p(n)}g\neq \int_{X}g\,d\mu.
	 \footnote{ By Proposition \ref{poly} and Lemma \ref{k}, to show the existence of this limit it suffices to do so for the case where $\X=\mathcal{K}(\X)$, which is a classical result (see for example \cite{L05}).}\]

	Then $\int_{X}g'\,d\mu=\int_{X}g\,d\mu$.
	Note that $g'$ cannot be $\mu$-a.e. equal to a constant. Letting $f_{1}=g, f_{2}=g'$ and $f_{0}=f_{3}=\dots=f_{d}\equiv 1$, we have that
	\begin{eqnarray*}
		\int_{X}f_{0}\cdot\mathbb{E}_{n\in\mathbb{Z}^{L}}T^{p(n)}_{1}f_{1}\cdot\ldots\cdot T^{p(n)}_{d}f_{d}\,d\mu
		& = &\mathbb{E}_{n\in\mathbb{Z}^{L}}\int_{X}T^{p(n)}_{1}g\cdot T^{p(n)}_{2}g'\,d\mu
		\\&= &\mathbb{E}_{n\in\mathbb{Z}^{L}}\int_{X}(T_{1}T^{-1}_{2})^{p(n)}g\cdot g'\,d\mu
		\\ & = &\int_{X}\mathbb{E}_{n\in\mathbb{Z}^{L}}(T_{1}T^{-1}_{2})^{p(n)}g\cdot g'\,d\mu
		\\&= &\int_{X}g'^{2}\,d\mu>\left(\int_{X}g'\,d\mu\right)^{2}=\left(\int_{X}g\,d\mu\right)^{2}
		\\	&= &\int_{X}f_{1}\,d\mu\cdot\ldots \cdot\int_{X}f_{d}\,d\mu,
	\end{eqnarray*}
	a contradiction to (\ref{t02}), proving  (i).

	To show (ii), it suffices to show that for all $f_{1},\dots,f_{d}\in L^{\infty}(\mu)$ with $\prod_{i=1}^{d} \int_{X}f_{i}\,d\mu=0$,  we have that
	\begin{eqnarray}\label{73}
	\mathbb{E}_{n\in\mathbb{Z}^{L}}\bigotimes_{i=1}^{d}T_{i}^{p(n)}f_{i}=\mathbb{E}_{n\in\mathbb{Z}^{L}}(T_{1}\times\dots\times T_{d})^{p(n)}\bigotimes_{i=1}^{d}f_{i}=0.
\end{eqnarray}
	We first claim that
		\begin{eqnarray}\nonumber
			\mathbb{E}_{n\in\mathbb{Z}^{L}}\bigotimes_{i=1}^{d}T_{i}^{p(n)}f_{i}=0 \text{ if $\mathbb{E}(f_{i}\vert Z_{\mathbb{Z}^{d},\mathbb{Z}^{d}}(\X))=0$ for some $1\leq i\leq d$.}
		\end{eqnarray}
	
	We apply the proof of Proposition~\ref{poly} to the $\Z$-system $(X^{d},\mathcal{B}^{d},\mu^{\otimes d},T_{1}\times\dots\times T_{d})$. Suppose that $\mathbb{E}(f_{i}\vert Z_{\mathbb{Z}^{d},\mathbb{Z}^{d}}(\X))=0$ for some $1\leq i\leq d$. 
	By  \cref{erg} and  (\ref{equation:PolyCharFactor}) in the proof of Proposition~\ref{poly},
	it suffices to show that the set of $(h_{1},\dots,h_{K})\in (\Z^{L})^{K}$ such that
	\begin{equation}\label{72}
		\mathbb{E}\Bigl(\bigotimes_{i=1}^{d}f_{i}\vert \I(G(c(h_1,\ldots,h_K)))\Bigr)=0
	\end{equation}		
	is of density 1, where $c\colon(\Z^{L})^{K}\to \Z$ is a non-constant polynomial. 
	If $c(h_{1},\dots,h_{K})\neq 0$, then  $\I(G(c(h_1,\ldots,h_K)))$ is the sub-$\sigma$-algebra of $\mathcal{B}^{d}$ consisting of the $(T_{1}\times\dots\times T_{d})^{c(h_{1},\dots,h_{K})}$-invariant sets. By Lemma \ref{lemma:product kroneckers}, $$\I(G(c(h_1,\ldots,h_K)))\subseteq\bigotimes_{i=1}^{d}Z_{T^{c(h_1,\ldots,h_K)}_{i},T^{c(h_1,\ldots,h_K)}_{i}}=\bigotimes_{i=1}^{d}Z_{T_{i},T_{i}},$$
	where we used \cref{replacement0} (iv) in the last equality. On the other hand,  by (\ref{t02}), we have that $(T^{p(n)}_{i})_{n\in\mathbb{Z}^{L}}$ is ergodic for $\mu$ for all $1\leq i\leq d$, which implies that $T_{i}$ is ergodic for $\mu$. By Lemma~\ref{replacement0}~(ii), we have that
	$$\I(G(c(h_1,\ldots,h_K)))\subseteq\bigotimes_{i=1}^{d}Z_{T_{i},T_{i}}=\bigotimes_{i=1}^{d}Z_{\Z^{d},\Z^{d}}.$$
	Since  $\mathbb{E}(f_{i}\vert Z_{\mathbb{Z}^{d},\mathbb{Z}^{d}}(\X))=0$, we have that $\mathbb{E}\Bigl(\bigotimes_{i=1}^{d}f_{i}\Big\vert \bigotimes_{i=1}^{d}Z_{\Z^{d},\Z^{d}}\Bigr)=\bigotimes_{i=1}^{d}\mathbb{E}(f_{i}\vert Z_{\Z^{d},\Z^{d}})=0$, and so (\ref{72}) holds whenever $c(h_1,\ldots,h_K)\neq 0$. By  Proposition~\ref{poly}, such tuples $(h_1,\ldots,h_K)$ are of density 1. This proves the claim.

	By the claim, it now suffices to prove (\ref{73}) under the assumption that all $f_{i}$'s are measurable with respect to $Z_{\Z^{d},\Z^{d}}$.
	By Lemma \ref{k}, we can approximate each $f_{i}$ in $L^{2}(\mu)$ by linear combinations of eigenfunctions of $\X$. By multi-linearity, we may assume without loss of generality that each $f_{i}$ is a non-constant eigenfunction of $\X$ satisfying $T_{g}f_{i}=\lambda_{i}(g)f_{i}$ for all $g\in \mathbb{Z}^{d}$ for some group homomorphism $\lambda_{i}\colon\mathbb{Z}^{d}\to\mathbb{S}^{1}$ and that $f_i(x)\neq 0$ $\mu$-a.e $x\in X$.  Then by (\ref{t02}), \[ 0= \prod_{i=1}^{d}\int_{X} f_i d\mu=  \mathbb{E}_{n\in\mathbb{Z}^{L}}\prod_{i=1}^{d}T_{i}^{p(n)}f_{i}=\Bigl(\mathbb{E}_{n\in\mathbb{Z}^{L}}\prod_{i=1}^{d}\lambda_{i}(p(n)e_{i})\Bigr)\prod_{i=1}^{d}f_{i}.\]
	 This implies that $\mathbb{E}_{n\in\mathbb{Z}^{L}}\prod_{i=1}^{d}\lambda_{i}(p(n)e_{i})=0$. So,
	\[\mathbb{E}_{n\in\mathbb{Z}^{L}}\bigotimes_{i=1}^{d}T_{i}^{p(n)}f_{i}=\Bigl(\mathbb{E}_{n\in\mathbb{Z}^{L}}\prod_{i=1}^{d}\lambda_{i}(p(n)e_{i})\Bigr)\bigotimes_{i=1}^{d}f_{i}=0.\] 
	This proves (ii) and finishes the proof.
\end{proof}	

\begin{proof}[Proof of Theorem \ref{t1} assuming  Theorem \ref{T:2}]
	We first prove the ``if'' part. We want to show that
	\begin{equation}\label{t01}
		\mathbb{E}_{n\in\mathbb{Z}^{L}}T^{p(n)}_{1}f_{1}\cdot\ldots\cdot T^{p(n)}_{d}f_{d}=\int_{X}f_{1}\,d\mu\cdot\ldots\cdot \int_{X}f_{d}\,d\mu
	\end{equation}
	for all $f_{1},\dots,f_{d}\in L^{\infty}(\mu)$.
	
Regard $T_{i}$ as $T_{e_{i}}$ and let $p_{1},\dots,p_{d}\colon\mathbb{Z}^{L}\to\mathbb{Z}^{d}$ be polynomials given by $p_{i}(n)=p(n)e_{i}$. 
Suppose that $p(n)=\sum_{0\leq m\leq K}q_{m}n^{m}$ for some $q_{m}\in\mathbb{Q}$.
If $p$ is a constant polynomial, then there is nothing to prove. So we assume that $p$ is not constant, and so $p_{1},\dots,p_{d}$ is a non-degenerate family of polynomials. 
The set $R$ defined in Theorem~\ref{T:2} is $R=\{q_{m}e_{i}, q_{m}(e_{i}-e_{j})\colon 1\leq i,j\leq d, i\neq j,$ $1\leq m\leq K, q_{m}\neq 0\}$. By  assumption (i), all the $T_{i}T^{-1}_{j}$'s (or $T_{e_{i}-e_{j}}$'s), $i\neq j$ are ergodic for $\mu$, and so $(T_{g})_{g\in G(q(e_{i}-e_{j}))}=(T_{g})_{g\in G(e_{i}-e_{j})}$ is ergodic for $\mu$ for all $q\neq 0$. By  assumption (ii), $(T_{i}^{p(n)})_{n\in\mathbb{Z}^{L}}$ is ergodic for $\mu$ for all $1\leq i\leq d$, which implies that $T_{i}$ (or $T_{e_{i}}$) is ergodic for $\mu$.
So, for all $q\neq 0,$  $(T_{g})_{g\in G(qe_{i})}=(T_{g})_{g\in G(e_{i})}$ is ergodic for $\mu$.
Thus the assumptions of  Theorem \ref{T:2} are fulfilled.

 By Theorem \ref{T:2}, we may assume without loss of generality that $\X=Z_{(\mathbb{Z}^{d})^{\times \infty}}(\X)$ (note that conditions (i) and (ii) remain valid when passing to a factor system). 
	Since $(T_{1}^{p(n)}\times\dots\times T_{d}^{p(n)})_{n\in\mathbb{Z}^{L}}$ is ergodic for $\mu^{\otimes d}$, 
	  an argument similar to the one in the proof of \cref{t2} yields the ``if'' part of this theorem.

	To prove the ``only if'' part, assume that (\ref{t01}) holds for all $f_{1},\dots,f_{d}\in L^{\infty}(\mu)$. If $T_{i}T^{-1}_{j}$ is not ergodic for some $1\leq i,j\leq d, i\neq j$, then there exists $g\in L^{\infty}(\mu)$  which is not $\mu$-a.e. equal to a constant such that $T_{i}g=T_{j}g$.
	So 
	\[\mathbb{E}_{n\in\mathbb{Z}^{L}}(T_{i}T^{-1}_{j})^{p(n)}g=\mathbb{E}_{n\in\mathbb{Z}^{L}}g=g\neq \int_{X} gd\mu,\] 
	which implies that  $((T_{i}T^{-1}_{j})^{p(n)})_{n\in\mathbb{Z}^{L}}$ is not ergodic for $\mu$, a contradiction to
	(i) in Proposition~\ref{inverse}. This proves (i).

	Since  (\ref{t01}) holds, (ii) follows directly from the statement (ii) of Proposition~\ref{inverse} and the proof is complete.	
\end{proof}

\subsection{Ingredients to proving Theorem \ref{T:2}}

The rest of the paper is devoted to the proof of Theorem~\ref{T:2}. 
In order to keep track of the coefficients of the polynomials after the iterated van der Corput operations, we introduce the following definition: 
\begin{definition*}
	Let $d\in\mathbb{N}^\ast$ and $V$ denote the collection of all finite subsets $\{u_{1},\dots,u_{k}\} \subseteq \mathbb{Q}^{d}$ containing the zero vector ${\bf 0}$.
	For $R_{1}=\{u_{1},\dots,u_{k}\}\in V$ and $R_{2}\subseteq\mathbb{Q}^{d}$, we say that $R_{1}$ is \emph{equivalent} to $R_{2}$ (denoted as $R_{1}\sim R_{2}$) if  there exists $1\leq i\leq k$
	such that $R_{2}=\{-ru_{i},r(u_{j}-u_{i})\colon 1\leq j\leq k\}$ for some $r\in\mathbb{Q}\backslash\{0\}.$
	Note that $R_1 \sim R_2$ implies that $R_1$ and $R_2$ have the same cardinality.\footnote{ Note that ${\bf 0}\in R_{2}$ as $r(u_{j}-u_{i})={\bf 0}$ for $j=i.$}
\end{definition*}

\begin{lemma} The relation 
	$\sim$ is an equivalence relation on $V$.
\end{lemma}

\begin{proof}
	If $R_{1}=\{u_{1},\dots,u_{k}\}$ and $u_{i}=0$, then  $R_{1}=\{-ru_{i},r(u_{j}-u_{i})\colon 1\leq j\leq k\}$ for $r=1,$ and so
	$R_{1}\sim R_{1}$. Suppose that $R_{1}\sim R_{2}$. We may write $R_{1}=\{u_{1},\dots,u_{k}\}$ and $R_{2}=\{v_{1},\dots,v_{k}\}$, where $v_{i}=-ru_{i}$ and $v_{j}=r(u_{j}-u_{i})$ for all $1\leq j\leq k, j\neq i$ for some $1\leq i\leq k$. It follows that $u_i=-(1/r)v_i$ and $u_j=(1/r)(v_j-v_i)$ which means $R_2\sim R_1$.  
	
	Assume now that $R_1\sim R_2$ and  $R_{2}\sim R_{3}$. We may write $R_2$ as above and $R_{3}=\{w_{1},\dots,w_{k}\}$, where $w_{i'}=-r'v_{i'}$ and $w_{j}=r'(v_{j}-v_{i'})$ for all $j\neq i'$ for some $1\leq i'\leq k$. 
	
	If $i=i'$, then $w_{i}=-r'v_{i}=-r'(-ru_{i})=rr'u_{i}$, and $w_{j}=r'(v_{j}-v_{i})=r'r(u_{j}-u_{i})-r'(-ru_{i})=rr'u_{j}$ for all $j\neq i$. So $R_{3}=rr'R_{1}$. This implies that $R_{1}\sim R_3$. 
	
	If $i\neq i'$, then $w_{i}=r'(v_{i}-v_{i'})=r'(-ru_{i})-r'r(u_{i'}-u_{i})=-rr'u_{i'}$, $w_{i'}=-r'v_{i'}=r'r(u_{i}-u_{i'})$, and $w_{j}=r'(v_{j}-v_{i'})=r'r(u_{j}-u_{i})-r'r(u_{i'}-u_{i})=rr'(u_{j}-u_{i'})$ for all $j\neq i, i'$. This implies that $R_{1}\sim R_{3}$ and the result follows.
\end{proof}	

We write $R_{1}\lesssim R_{2}$ for some $R_{1},R_{2}\in V$ if there exists $R_{3}\in V$ such that $R_{2}\sim R_{3}$ and $R_{1}\subseteq R_{3}$. 

Recall that for $\textbf{b}=(b_{1},\dots,b_{L})\in(\mathbb{Q}^{d})^{L},$ $ b_{i}\in\mathbb{Q}^{d}$, we denote $G(\textbf{b})=\text{span}_{\mathbb{Q}}\{b_{1},\dots,b_{L}\}\cap\mathbb{Z}^{d}$  and $G'(\textbf{b})\coloneqq \text{span}_{\mathbb{Z}}\{b_{1},\dots,b_{L}\}$.
The first ingredient we need to prove Theorem \ref{T:2} is an upper bound for the multiple averages in terms of Host-Kra seminorms. The following proposition shows that we can somehow control the coefficients we get in the end of the PET-induction by the initial ones. 

\begin{proposition}[Bounding multiple averages by averaged Host-Kra seminorms]\label{pet}
	Let $d,k,L\in\mathbb{N}^{\ast},$ $p_{1},\dots$, $p_{k}\colon\mathbb{Z}^{L}\to \mathbb{Z}^{d}$ a non-degenerate family of polynomials of degrees at most $K,$ with $p_{i}(n)=\sum_{v\in\mathbb{N}^{L}, \vert v\vert \leq K}b_{i,v}n^{v}$ for some $b_{i,v}\in \mathbb{Q}^{d},$ and $R_{v}\coloneqq\{b_{i,v}\colon 1\leq i\leq k\}\cup\{{\bf 0}\}$.  
	Then
	there exist $t_{1},\dots,t_{k}\in\mathbb{N}^\ast$, $s\in\N$ and polynomials $\c_{i,m}\colon(\mathbb{Z}^{L})^{s}\to(\mathbb{Z}^{d})^{L}, 1\leq i\leq k, 1\leq m\leq t_{i}$ with $\c_{i,m}\not\equiv {\bf 0},$ such that the following hold:
	\begin{itemize}
		\item[(i)] (Control of the coefficients) Each $\c_{i,m}$ is of the form
		\[\c_{i,m}(h_{1},\dots,h_{s})=\sum_{a_{1},\dots,a_{s}\in\mathbb{N}^{L}}h^{a_{1}}_{1}\dots h^{a_{s}}_{s}\cdot \u_{i,m}(a_{1},\dots,a_{s})\]
		for some  $$\u_{i,m}(a_{1},\dots,a_{s})=(u_{i,m,1}(a_{1},\dots,a_{s}),\dots,u_{i,m,L}(a_{1},\dots,a_{s}))\in(\mathbb{Q}^{d})^{L}$$ with all but finitely many terms being zero for each $(i,m).$ In addition, for all $a_{1},\dots,a_{s}\in\mathbb{N}^{L}$ not all equal to {\bf 0} and every $1\leq i\leq k, 1\leq r\leq L$, denoting
		\[U_{i,r}(a_{1},\dots,a_{s})\coloneqq \{u_{i,m,r}(a_{1},\dots,a_{s})\in\mathbb{Q}^{d}\colon 1\leq m\leq t_{i}\}\cup\{{\bf 0}\},\]
		we have that there exists $v\in\mathbb{N}^{L}, \vert v\vert>0$ such that $U_{i,r}(a_{1},\dots,a_{s})\lesssim R_{v}$.
		\item[(ii)] (Control of the average) For every $\mathbb{Z}^{d}$-system $\X=(X,\mathcal{B},\mu, (T_{g})_{g\in\mathbb{Z}^{d}})$ and every $f_{1},\dots,$ $f_{k}\in L^{\infty}(\mu)$ bounded by $1,$ we have that 
		\begin{equation}\label{30}
			\F\vl\Bigl\Vert\mathbb{E}_{n\in I_{N}}\prod_{i=1}^k T_{p_{i}(n)}f_i\Bigr\Vert^{2^{t_0}}_{L^{2}(\mu)}
			\leq C \cdot \min_{1\leq i\leq k}\E^{\square}_{h_{1},\dots,h_{s}\in\mathbb{Z}^{L}}\Vert f_{i}\Vert_{( G'(\c_{i,m}(h_{1},\dots,h_{s})))_{1\leq m\leq t_{i}}},
		\end{equation}
	\end{itemize}
	where $t_0$ and $C>0$ are constants depending only on $p_{1},\dots,p_{k}$.\footnote{ One can in fact show that $t_0,t_1,\ldots,t_k$ depend only on $d, k, L$ and the highest degree of $p_{1},\dots,p_{k}$. More specifically, $t_0$ can be chosen to be the $\max\{t_1,\ldots,t_k\},$ where $t_i$ is the number of vdC-operations that we have to perform in order our PET tuple to be non-degenerate, standard for $f_i$ and with degree equal to $1.$ }
\end{proposition}

\begin{remark}
	Note that we allow $s=0$ in Proposition \ref{pet}. In this case, we write 
	$$\c_{i,m}(\emptyset)=\u_{i,m}(\emptyset)=(u_{i,m,1}(\emptyset),\dots,u_{i,m,L}(\emptyset))\in(\mathbb{Z}^{d})^{L}$$ for some $u_{i,m,r}(\emptyset)\in\mathbb{Z}^{d}$, and 
	$$U_{i,r}(\emptyset):=\{u_{i,m,r}(\emptyset)\in\mathbb{Z}^{d}\colon 1\leq m\leq t_{k}\}\cup\{\bold{0}\}.$$
	Moreover, the right hand side of \eqref{30} becomes $\min_{1\leq i\leq k}\Vert f_{i}\Vert_{(G'(\c_{i,m}(\emptyset)))_{1\leq m\leq t_{i}}}$.\footnote{ In this paper, when $s=0$, averages of the form $\E^{\square}_{h_{1},\dots,h_{s}\in\mathbb{Z}^{L}}a(h_{1},\dots,h_{s})$ are understood as the single term $a(\emptyset)$.}
\end{remark}

The second ingredient we need in order to show Theorem~\ref{T:2}  (which is the main novelty of this paper) is to estimate the right hand side of (\ref{30}) using the concatenation theorem.

\begin{proposition}[Bounding averaged Host-Kra seminorms by a single one]\label{pet3}
	Let $p_{1},\dots,p_{k}\colon\mathbb{Z}^{L}\to \mathbb{Z}^{d}$ be a family of polynomials. Suppose that
	there exist $t_{1},\dots,t_{k}\in\mathbb{N}^\ast$,  $s\in\N$ and polynomials $\c_{i,m}\colon(\mathbb{Z}^{L})^{s}\to(\mathbb{Z}^{d})^{L}, 1\leq i\leq k, 1\leq m\leq t_{i},$ with $\c_{i,m}\not\equiv {\bf 0}$  given by
	\begin{equation}\label{34}
		\begin{split}
			\c_{i,m}(h_{1},\dots,h_{s})=\sum_{a_{1},\dots,a_{s}\in\mathbb{N}^{L}}h^{a_{1}}_{1}\dots h^{a_{s}}_{s}\cdot \u_{i,m}(a_{1},\dots,a_{s})
		\end{split}
	\end{equation}
	for some $\u_{i,m}(a_{1},\dots,a_{s})\in(\mathbb{Q}^{d})^{L}$ with all but finitely many terms equal to ${\bf 0}$ for each $(i,m)$ such that the following holds:
	if for every $\mathbb{Z}^{d}$-system $(X,\mathcal{B},\mu, (T_{g})_{g\in\mathbb{Z}^{d}})$ and every $f_{1},\dots,f_{k}\in L^{\infty}(\mu)$ bounded by $1,$ we have that 
	\begin{equation}\label{35}
		\begin{split}
			\F\vl\Bigl\Vert\mathbb{E}_{n\in I_{N}} \prod_{i=1}^k T_{p_{i}(n)}f_i\Bigr\Vert_{L^{2}(\mu)}
			\leq C \cdot \min_{1\leq i\leq k}\E^{\square}_{h_{1},\dots,h_{s}\in\mathbb{Z}^{L}}\Vert f_{i}\Vert_{(G'(\c_{i,m}(h_{1},\dots,h_{s})))_{1\leq m\leq t_{i}}},
		\end{split}
	\end{equation}
	where $C>0$ is a constant depending only on $p_{1},\dots,p_{k}$,
	then letting
	\begin{equation}\nonumber
		\begin{split}
			H_{i,m}=
			\text{span}_{\mathbb{Q}}\{G(\u_{i,m}(a_{1},\dots,a_{s}))\colon a_{1},\dots,a_{s}\in\mathbb{N}^{L}\}\cap\mathbb{Z}^{d},
		\end{split}
	\end{equation}
	we have that
	\begin{equation}\label{36}
		\begin{split}
			\F\vl\Bigl\Vert\mathbb{E}_{n\in I_{N}} T_{p_{1}(n)}f_1\cdot\ldots\cdot T_{p_{k}(n)}f_k\Bigr\Vert_{L^{2}(\mu)}=0 \;\text{ if $\;\min_{1\leq i\leq k}\Vert f_{i}\Vert_{H^{\times\infty}_{i,1},\dots,H^{\times\infty}_{i,t_{i}}}=0$}.
		\end{split}
	\end{equation}
\end{proposition}

We now use Propositions~\ref{pet} and \ref{pet3} to show Theorem~\ref{T:2},
and leave the proofs of Propositions~\ref{pet} and \ref{pet3} to Sections~\ref{s:pet} and \ref{s:pet3} respectively. 

\begin{proof}[Proof of Theorem \ref{T:2} assuming Propositions \ref{pet} and \ref{pet3}]
Let $R$ be the set defined in Theorem \ref{T:2}.
	We can assume without loss of generality that $\mathbb{E}(f_1\vert Z_{\{G(r)^{\times \infty}\}_{r\in R}})=0$.
	Suppose that $p_{i}(n)=\sum_{v\in\mathbb{N}^{L},\vert v\vert\leq K}b_{i,v}n^{v}$ for some $b_{i,v}\in \mathbb{Q}^{d}$ and denote $R_{v}=\{b_{i,v}\colon 1\leq i\leq k\}\cup\{{\bf 0}\}$ as in Proposition~\ref{pet}. By the same proposition,
	there exist $t_{1},\dots,t_{k}\in\mathbb{N}^\ast,$ $s\in\mathbb{N}$\footnote{ We actually address the $s\in\mathbb{N}^\ast$ case for aesthetic reasons here, as the $s=0$ case follows analogously.} and  polynomials $\c_{i,m}\colon(\mathbb{Z}^{L})^{s}\to(\mathbb{Z}^{d})^{L}, 1\leq i\leq k, 1\leq m\leq t_{i},$ with $\c_{i,m}\not\equiv {\bf 0}$
	given by
	$$\c_{i,m}(h_{1},\dots,h_{s})=\sum_{a_{1},\dots,a_{s}\in\mathbb{N}^{L}}h^{a_{1}}_{1}\dots h^{a_{s}}_{s}\cdot \u_{i,m}(a_{1},\dots,a_{s})$$
	for some  $\u_{i,m}(a_{1},\dots,a_{s})\in(\mathbb{Q}^{d})^{L}$ with all but finitely many terms equal to ${\bf 0}$ for each $(i,m)$ (and satisfying the additional assumptions given by Proposition~\ref{pet}), such that (\ref{30}) holds. Let \begin{equation}\nonumber
		\begin{split}
			H_{i,m}=\text{span}_{\mathbb{Q}}\{G(\u_{i,m}(a_{1},\dots,a_{s}))\colon a_{1},\dots,a_{s}\in\mathbb{N}^{L}\}\cap\mathbb{Z}^{d}.
		\end{split}
	\end{equation}
	By  Proposition~\ref{pet3}, 
	\begin{equation}\label{366}
		\F\vl\Bigl\Vert\mathbb{E}_{n\in I_{N}} T_{p_{1}(n)}f_1\cdot\ldots\cdot T_{p_{k}(n)}f_k\Bigr\Vert_{L^{2}(\mu)}=0\;\text{ if $\;\min_{1\leq i\leq k}\Vert f_{i}\Vert_{H^{\times\infty}_{i,1},\dots,H^{\times\infty}_{i,t_{i}}}=0$.} \footnote{ Note that we have in fact proved the following stronger version of (\ref{12}):
			\[\mathbb{E}_{n\in\mathbb{Z}^{L}}  T_{p_{1}(n)}f_1\cdot\ldots\cdot T_{p_{k}(n)}f_k=0\; \text{ if $\;\mathbb{E}(f_i\vert Z_{\{H_{i,j}\}_{1\leq j\leq t_{i}}^{\times \infty}})=0$ for some $1\leq i\leq k$.}
			\]}
	\end{equation}
	On the other hand, by the description of $\c_{i,m}$, writing $$\u_{i,m}(a_{1},\dots,a_{s})=(u_{i,m,1}(a_{1},\dots,a_{s}),\dots,u_{i,m,L}(a_{1},\dots,a_{s})),$$ each $u_{i,m,j}(a_{1},\dots,a_{s})$ belongs to the set $U_{i,r}$, which is contained in a set equivalent to one of $R_{v}, v\in\mathbb{N}^{L}, 0< \vert v\vert\leq k$. By the definition of $R$, $u_{i,m,j}(a_{1},\dots,a_{s})=qr$ for some $q\in\mathbb{Q}$ and $r\in R$. Since $\c_{1,m}\not\equiv {\bf 0},$ there exists $q_{m}r_{m}\in H_{1,m}\backslash\{\bold{0}\}$ for some $q_{m}\in\mathbb{Q}$ and $r_{m}\in R$ for all $1\leq m\leq t_{1}$. So $G(r_{m})$ is a subgroup of $H_{1,m}$. 	By Lemma~\ref{replacement0}, we have that 
	\[Z_{H^{\times\infty}_{1,1},\dots,H^{\times\infty}_{1,t_{1}}}
	\subseteq Z_{G(r_{1})^{\times\infty},\dots,G(r_{t_{1}})^{\times\infty}}
	\subseteq Z_{\{G(r)^{\times \infty}\}_{r\in R}}.\]
	Since $\mathbb{E}(f_1\vert Z_{\{G(r)^{\times \infty}\}_{r\in R}})=0$, we have that $\mathbb{E}(f_1\vert Z_{H^{\times\infty}_{1,1},\dots,H^{\times\infty}_{1,t_{1}}})=0,$
	meaning that the second term of (\ref{366}) is 0, which implies that (\ref{12}) equals 0.

	If in addition, $(T_g)_{g\in G(r)}$ is assumed to be ergodic for all $r\in R$, then by Corollary~\ref{ppp}, we have that
	$Z_{\{G(r)^{\times \infty}\}_{r\in R}}=Z_{(\mathbb{Z}^{d})^{\times\infty}}$
	and the proof is complete.
\end{proof}

\section{Proof of Proposition~\ref{pet}}\label{s:pet}

Our strategy to show (\ref{30}) in \cref{pet} is the following: We first fix the functions $f_{i}$ on the right hand side of (\ref{30}). By a ``dimension-increment'' argument (see Proposition~\ref{ext} below), for a fixed $i,$ we may assume that $p_{i}$ has the highest degree among $p_{1},\ldots,p_{k}$, making the PET-tuple standard for $f_{i}$. Then, Theorem~\ref{PET} allows us to control the left hand side of (\ref{30}) by a PET-tuple of degree 1 which is also standard for $f_{i}$. Finally, a Host-Kra-type inequality for linear polynomials (see Proposition~\ref{P:1}) implies that (\ref{30}) holds for some polynomials $\c_{i,m}$. Up to this point, the method we use is similar to the one used in \cite{Jo} and \cite{L9} (the main difference is that we have a more explicit upper bound for $\vl\Bigl\Vert\mathbb{E}_{n\in I_{N}} T_{p_{1}(n)} f_1\cdot\ldots\cdot T_{p_{k}(n)} f_k\Bigr\Vert_{L^2(\mu)}$ in Proposition~\ref{pet}). Our innovation is that in order for the equation (\ref{30}) to be useful for our purposes, we need a better description of the functions $\c_{i,m}$, which is the content of part (i) of \cref{pet}.

We start with the linear case of Proposition \ref{pet} (the special case $L=1$ was first proved in \cite[Proposition~1]{Ho}).

\begin{proposition}[Host-Kra inequality for linear $\mathbb{Z}^{L}$-averages]\label{P:1} Let $d,k,L\in\mathbb{N}^{\ast}$ with $k\geq 2$, $(X,\mathcal{B},\mu,(T_{g})_{g\in\mathbb{Z}^{d}})$ a $\mathbb{Z}^{d}$-system and $p_{1},\dots,p_{k}\colon\mathbb{Z}^{L}\to\mathbb{Z}^{d}$ essentially distinct and essentially non-constant polynomials of degree $1.$ Suppose that 
	$p_{i}(n)=\u_{i}\cdot n+v_{i}$ for some $\u_{i}\in(\mathbb{Z}^{d})^{L}, v_{i}\in\mathbb{Z}^{d}$ for all $1\leq i\leq k$.\footnote{ Here for $\u=(u_{1},\dots,u_{L})\in(\Z^{d})^{L}$ and $n=(n_{1},\dots,n_{L})\in\N^{L}$, $\u\cdot n$ denotes $n_{1}u_{1}+\dots+n_{L}u_{L}\in\mathbb{Z}^{d}$.} 	
	Then for every $f_1,\ldots,f_k\in L^\infty(\mu)$ bounded by $1,$  we have that 
	\[\F\vl\Bigl\Vert\ei T_{p_{1}(n)} f_1\cdot\ldots\cdot T_{p_{k}(n)} f_k\Bigr\Vert_{L^2(\mu)}\leq C\cdot \min_{1\leq i\leq k}\Vert f_i\Vert_{G'(-\u_{i}), \{G'(\u_{j}-\u_{i})\}_{1\leq j\leq k, j\neq i}},\]
	where $C$ is a constant only depending on $k$. 	
	Moreover, writing $\u_{i}=(u_{i,1},\dots,u_{i,L}), u_{i,j}\in\mathbb{Z}^{d}$ and $R_{e_{r}}=\{u_{i,r}\colon 1\leq i\leq k\}\cup\{{\bf{0}}\}$, the set
	\[U_{i,r}(\emptyset)=\{-u_{i,r}, u_{j,r}-u_{i,r}\colon 1\leq j\leq k\}\]
	is equivalent to $R_{e_{r}}$ for all $1\leq i\leq k$.\footnote{
	It is not hard to verify that the set $R_{e_{r}}$ coincides with the sets $R_{v}$ defined in  Proposition~\ref{pet}. Setting $\bold{c_{i,i}}(\emptyset)=-\bold{u}_{i}$, $\bold{c_{i,j}}(\emptyset)=\bold{u}_{j}-\bold{u}_{i}$ for $1\leq j\leq k, j\neq i$, we see that Proposition \ref{P:1} is indeed a special case of Proposition \ref{pet} when $K=1$ and $s=0$.}
\end{proposition}

\begin{proof}
We first assume that $k=2$.
	Then, by Lemma~\ref{lemma:iteratedVDC}, the Cauchy-Schwarz inequality, and using the fact that $p_{1},p_{2}$ are of degree $1,$ we have that
	\begin{equation}\nonumber
	\begin{split}
	& \F\vl\Bigl\Vert\ei T_{p_{1}(n)} f_1\cdot T_{p_{2}(n)} f_k\Bigr\Vert_{L^2(\mu)}^{4}
	\\ & \leq  16 \E^{\square}_{h\in\mathbb{Z}^{L}}\F\vl\Bigl\vert\ei\int_{X} \prod_{i=1}^2T_{p_{i}(n)}f_{i}\cdot \prod_{i=1}^2 T_{p_{i}(n+h)}f_{i}\,d\mu\Bigr\vert^{2}
	\\&=  16\E^{\square}_{h\in\mathbb{Z}^{L}}\F\vl\Bigl\vert\int_{X}f_\ast\cdot f_2\cdot T_{p_{2}(h)}f_{2}\cdot\ei T_{p_{1}(n)-p_{2}(n)}(f_{1}\cdot T_{p_{1}(h)}f_{1})\,d\mu\Bigr\vert^{2}
	\\&\leq  16 \E^{\square}_{h\in\mathbb{Z}^{L}}\F\vl\Bigl\Vert\ei T_{p_{1}(n)-p_{2}(n)}(f_{1}\cdot T_{p_{1}(h)}f_{1})\Bigr\Vert_{L^2(\mu)}^{2},
	\end{split}
	\end{equation}
	where $f_\ast=T_{-v_1}f_1\cdot T_{-v_2}f_2$ captures the constant terms which are removed with the use of Cauchy-Schwarz inequality in the next step, as $p_i(n+h)=p_i(n)+p_i(h)-v_i.$

		Writing $p_{1}(n)-p_{2}(n)=\u'_{1}\cdot n+(v_{1}-v_{2})$, where $\u'_{1}\coloneqq \u_{1}-\u_{2}$, using Lemma~\ref{replacement0} (iii) and (iv), we have
	\begin{equation}\nonumber
	\begin{split}
	&16\E^{\square}_{h\in\mathbb{Z}^{L}}\F\vl\Bigl\Vert\ei T_{p_{1}(n)-p_{2}(n)}(f_{1}\cdot T_{p_{1}(h)}f_{1})\Bigr\Vert_{L^2(\mu)}^{2}
	\\& =  16\cdot\E^{\square}_{h\in\mathbb{Z}^{L}}\Bigl\Vert f_{1}\cdot T_{p_{1}(h)}f_{1}\Bigr\Vert_{G'(\u'_{1})}^{2}
	\\ & =  16\cdot\Vert f_1\Vert_{G'(-\u_{1}), G'(\u'_{1})}^{4}
	\\ & =  16\cdot\Vert f_1\Vert_{G'(-\u_{1}), G'(\u_{2}-\u_{1})}^{4}.
	\end{split}
	\end{equation}

	Suppose now that the conclusion holds for $k-1$ for some $k\geq 3$. Then, by Lemma~\ref{lemma:iteratedVDC}, the Cauchy-Schwarz inequality, and using the fact that $p_{1},\dots,p_{k}$ are of degree $1,$ we have that
	\begin{equation}\nonumber
	\begin{split}
	& \F\vl\Bigl\Vert\ei T_{p_{1}(n)} f_1\cdot\ldots\cdot T_{p_{k}(n)} f_k\Bigr\Vert_{L^2(\mu)}^{2^{k}}
	\\ & \leq  4^{2^{k-1}} \E^{\square}_{h\in\mathbb{Z}^{L}}\F\vl\Bigl\vert\ei\int_{X} \prod_{i=1}^kT_{p_{i}(n)}f_{i}\cdot \prod_{i=1}^k T_{p_{i}(n+h)}f_{i}\,d\mu\Bigr\vert^{2^{k-1}}
	\\&=  4^{2^{k-1}}\E^{\square}_{h\in\mathbb{Z}^{L}}\F\vl\Bigl\vert\int_{X}f_\ast\cdot f_k\cdot T_{p_{k}(h)}f_{k}\cdot\ei\prod_{i=1}^{k-1}T_{p_{i}(n)-p_{k}(n)}(f_{i}\cdot T_{p_{i}(h)}f_{i})\,d\mu\Bigr\vert^{2^{k-1}}
	\\&\leq  4^{2^{k-1}} \E^{\square}_{h\in\mathbb{Z}^{L}}\F\vl\Bigl\Vert\ei\prod_{i=1}^{k-1}T_{p_{i}(n)-p_{k}(n)}(f_{i}\cdot T_{p_{i}(h)}f_{i})\Bigr\Vert_{L^2(\mu)}^{2^{k-1}},
	\end{split}
	\end{equation}
	where $f_\ast=T_{-v_1}f_1\cdot\ldots\cdot T_{-v_k}f_k$ captures the constant terms which are removed with the use of Cauchy-Schwarz inequality in the next step, as $p_i(n+h)=p_i(n)+p_i(h)-v_i.$

	Note that $p_{i}(n)-p_{k}(n)=\u'_{i}\cdot n+(v_{i}-v_{k})$, where $\u'_{i}\coloneqq \u_{i}-\u_{k}$. By the induction hypothesis, there is a constant $C'$ depending  only on $k$  such that
	\begin{equation}\nonumber
	\begin{split}
	&	4^{2^{k-1}}\E^{\square}_{h\in\mathbb{Z}^{L}}\F\vl\Bigl\Vert\ei\prod_{i=1}^{k-1}T_{p_{i}(n)-p_{k}(n)}(f_{i}\cdot T_{p_{i}(h)}f_{i})\Bigr\Vert_{L^2(\mu)}^{2^{k-1}}
	\\& \leq  4^{2^{k-1}}C'\cdot\E^{\square}_{h\in\mathbb{Z}^{L}}\Bigl\Vert f_{1}\cdot T_{p_{1}(h)}f_{1}\Bigr\Vert_{G'(-\u'_{1}),\{G'(\u'_{1}-\u'_{j})\}_{2\leq j\leq k-1}}^{2^{k-1}}
	\\& =  C\cdot\E^{\square}_{h\in\mathbb{Z}^{L}}\Bigl\Vert f_{1}\cdot T_{p_{1}(h)}f_{1}\Bigr\Vert_{G'(\u_{k}-\u_{1}),\{G'(\u_{1}-\u_{j})\}_{2\leq j\leq k-1}}^{2^{k-1}}
	\\ & =  C\cdot\E^{\square}_{h\in\mathbb{Z}^{L}}\Bigl\Vert f_{1}\cdot T_{p_{1}(h)}f_{1}\Bigr\Vert_{\{G'(\u_{1}-\u_{j})\}_{2\leq j\leq k}}^{2^{k-1}}\\
	&=  C\cdot\Vert f_1\Vert_{G'(-\u_{1}), \{G'(\u_{j}-\u_{1})\}_{2\leq j\leq k}}^{2^{k}},
	\end{split}
	\end{equation}	
	where $C=4^{2^{k-1}}C'$ and  we used Lemma~\ref{replacement0} (iii) in the last equality. It is clear that the constant $C$ depends only on $k$.
	 By symmetry, 
	\[\F\vl\Bigl\Vert\ei T_{p_{1}(n)} f_1\cdot\ldots\cdot T_{p_{k}(n)} f_k\Bigr\Vert_{L^2(\mu)}\leq  C\cdot \min_{1\leq i\leq k}\Vert f_i\Vert_{G'(-\u_{i}), \{G'(\u_{j}-\u_{i})\}_{1\leq j\leq k, j\neq i}}\] and the claim follows. 
\end{proof}

Before proving the general case of Proposition~\ref{pet}, we continue with some additional computations for our Example~\ref{Eg:1}.

\noindent {\bf Second part of computations for Example~\ref{Eg:1}}: Recall that we are dealing with the case $(T^{n^{2}+n}_{1},T^{n^{2}}_{2})$, with the PET-tuple
\[A=(1,0,2,(f_{1},f_{2}),(p_{1}, p_{2})),\] 
where $p_{1}(n)=(n^{2}+n,0)=(n^{2}+n)e_{1}, p_{2}(n)=(0,n^{2})=n^{2}e_{2},$ and  $e_{1}=(1,0), e_{2}=(0,1)$ and $e=e_1-e_2$.
In this case, $L=1$ and $d=2$, $R_{1}=\{e_{1},{\bf 0}\}$, $R_{2}=\{e_{1},e_{2},{\bf 0}\}$ and $R_{v}=\{{\bf 0}\}$ for all $v>2$. Take $s=3$.

By the first part of computations of Example~\ref{Eg:1}, isolating $f_1,$ we have that
$\partial_{2}\partial_{3}\partial_{2} A=(3,7, (f_{1},\dots,f_{1}),$ $ \p_{3}),$
where the tuple $\p_{3}=(q_{1},\dots,q_{7})$ essentially equals to 
$$(-2h_{1}ne_{1},2h_{2}ne-2h_{1}e_{1},2h_{2}ne,2h_{3}ne-2h_{1}ne_{1},2h_{3}ne,2(h_{2}+h_{3})ne-2h_{1}ne_{1},2(h_{2}+h_{3})ne).$$	By Propositions \ref{P:1}, \ref{induction} and Lemma \ref{replacement0} (iv) and the fact that Host-Kra seminorms are $T_{g}$-invariant, we have that
\begin{equation}
\begin{split}
& \F\vl\Bigl\Vert\mathbb{E}_{n\in I_{N}} T^{n^{2}+n}_{1}f_1\cdot T^{n^{2}}_{2}f_2 \Bigr\Vert^{8}_{L^{2}(\mu)}
=  S(A,2^3) 
\leq C \cdot  S(\partial_{2}\partial_{3}\partial_{2} A,1)
\\& \quad\quad\quad\quad\quad\quad\quad\quad\quad\quad\quad\quad\quad\quad \leq C \cdot  \overline{\mathbb{E}}^{\square}_{h_{1},h_{2},h_{3}\in\mathbb{Z}} \Vert f_{1}\Vert_{G'(\c_{1,1}(h_{1},h_{2},h_{3})),\dots, G'(\c_{1,7}(h_{1},h_{2},h_{3}))},
\end{split}
\label{101}
\end{equation}
where 
$\c_{1,1}(h_{1},h_{2},h_{3})=-2h_{1}e_{1},$ $\c_{1,2}(h_{1},h_{2},h_{3})=2h_{2}e,$ $\c_{1,3}(h_{1},h_{2},h_{3})=-2h_{1}e_{1}+2h_{2}e$, $\c_{1,4}(h_{1},h_{2},h_{3})$ $=2h_{3}e$, $\c_{1,5}(h_{1},$ $h_{2},h_{3})=-2h_{1}e_{1}+2h_{3}e$, $\c_{1,6}(h_{1},h_{2},h_{3})=2(h_{2}+h_{3})e$, $\c_{1,7}(h_{1},h_{2},$ $h_{3})=-2h_{1}e_{1}+2(h_{2}+h_{3})e$.
 This verifies part (ii) of Proposition~\ref{pet} for $i=1$.
Moreover, using the notation in Proposition~\ref{pet}, we have that
\begin{eqnarray*}
	U_{1,1}(1,0,0) & = &\{-2e_{1},{\bf 0},-2e_{1},{\bf 0},-2e_{1},{\bf 0},-2e_{1}\}=\{-2e_{1},{\bf 0}\}\sim R_{1},
	\\ U_{1,1}(0,1,0) &= &\{{\bf 0},2e,2e,{\bf 0},{\bf 0},2e,2e\}=\{2e,{\bf 0}\}\subseteq\{2e,-2e_{2},{\bf 0}\}\sim R_{2},
	\\U_{1,1}(0,0,1) & = &\{{\bf 0},{\bf 0},{\bf 0},2e,2e,2e,2e\}=\{2e,{\bf 0}\}\subseteq\{2e,-2e_{2},{\bf 0}\}\sim R_{2}.
\end{eqnarray*}
This verifies part (i) of Proposition~\ref{pet} for $i=1$.

Similarly, by
isolating $f_2,$ we have that 
$\partial_{1}\partial_{2}\partial_{1} A=(3,7, (f_{2},\dots,f_{2}), \p_{3}),$
where the tuple $\p_{3}$ essentially equals to 

\noindent $(2h_{1}ne_{2},-2h_{2}ne,-2h_{2}ne+2h_{1}ne_{2},-2h_{3}ne,
-2h_{3}ne+2h_{1}ne_{2},-2(h_{2}+h_{3})ne,-2(h_{2}+h_{3})ne+2h_{1}ne_{2}).$
Analogously to \eqref{101}, we have 
\begin{equation}
\begin{split}
&\F\vl\Bigl\Vert\mathbb{E}_{n\in I_{N}} T^{n^{2}+n}_{1}f_1\cdot T^{n^{2}}_{2}f_2 \Bigr\Vert^{8}_{L^{2}(\mu)}
 =S(A,2^{3})\leq C\cdot S(\partial_{1}\partial_{2}\partial_{1} A,1)
\\ & \quad\quad\quad\quad\quad\quad\quad\quad\quad\quad\quad\quad\quad\quad \leq C\cdot \overline{\mathbb{E}}^{\square}_{h_{1},h_{2},h_{3}\in\mathbb{Z}} \Vert f_{2}\Vert_{G'(\c_{2,1}(h_{1},h_{2},h_{3})),\dots, G'(\c_{2,7}(h_{1},h_{2},h_{3}))},
\end{split}
\label{102}
\end{equation}
where 
$\c_{2,1}(h_{1},h_{2},h_{3})=2h_{1}e_{2},$ 
$\c_{2,2}(h_{1},h_{2},h_{3})=-2h_{2}e+2h_{1}e_{2},$ 
$\c_{2,3}(h_{1},h_{2},h_{3})=-2h_{2}e$,
$\c_{2,4}(h_{1},h_{2},h_{3})=-2h_{3}e+2h_{1}e_{2}$,
$\c_{2,5}(h_{1},h_{2},h_{3})=-2h_{3}e$, 
$\c_{2,6}(h_{1}, h_{2},h_{3})=-2(h_{2}+h_{3})e+2h_{1}e_{2}$, 
$\c_{2,7}(h_{1},h_{2},h_{3})=-2(h_{2}+h_{3})e.$ 
This verifies part (ii) of Proposition \ref{pet} for $i=2$.
Using the notation in Proposition \ref{pet}, we have that
\begin{eqnarray*}
	U_{2,1}(1,0,0)& = &\{2e_{2},2e_{2},{\bf 0},2e_{2},{\bf 0},2e_{2},{\bf 0}\}=\{2e_{2},{\bf 0}\}\subseteq\{-2e,2e_{2},{\bf 0}\}\sim R_{2},
	\\ U_{2,1}(0,1,0) & = &\{{\bf 0},-2e,-2e,{\bf 0},{\bf 0},-2e,-2e\}=\{-2e,{\bf 0}\}\subseteq\{-2e,2e_{2},{\bf 0}\}\sim R_{2},
	\\ U_{2,1}(0,0,1) &= &\{{\bf 0},{\bf 0},{\bf 0},-2e,-2e,-2e,-2e\}=\{-2e,{\bf 0}\}\subseteq\{-2e,2e_{2},{\bf 0}\}\sim R_{2}.
\end{eqnarray*}
This verifies part (i) of Proposition \ref{pet} for $i=2$.
\medskip

We now introduce some additional notation that we will use in the general case. Let $d,\ell,L\in\mathbb{N}^{\ast}$,
$s\in\mathbb{N}$ and $q_{1},\dots,q_{\ell}\colon(\mathbb{Z}^{L})^{s+1}\to\mathbb{Z}^{d}$ be polynomials. Denote $\q=(q_{1},\dots,q_{\ell})$, where 
\[q_{i}(n;h_{1},\dots,h_{s})=\sum_{b,a_{1},\dots,a_{s}\in\mathbb{N}^{L}} h^{a_{1}}_{1}\dots h^{a_{s}}_{s}n^{b}\cdot u_{i}(b;a_{1},\dots,a_{s})\]
for some $u_{i}(b;a_{1},\dots,a_{s})\in\mathbb{Q}^{d}$ with all but finitely many being ${\bf 0}$ for each $1\leq i\leq \ell$. For all $b,a_{1},\dots,a_{s}\in\mathbb{N}^{L}$, denote
$$R_{\q}(b;a_{1},\dots,a_{s})\coloneqq \{u_{i}(b;a_{1},\dots,a_{s})\colon 1\leq i\leq \ell\}\cup\{{\bf 0}\}\subseteq\mathbb{Q}^{d}.$$
Roughly speaking, $R_{\q}(b;a_{1},\dots,a_{s})$ records the coefficients of $\q$ at ``level''-$(b;a_{1},\dots,a_{s})$ (together with the zero vector $\textbf{0}$).

The following proposition shows that, during the PET-induction process, after applying the vdC-operation to our expression, we can still keep track of the coefficients of the polynomials.

\begin{proposition}[vdC-operations treat the sets $R_{\q}(b;a_{1},\dots,a_{s})$ nicely]\label{1234} 
	Let $d,\ell,L\in\mathbb{N}^{\ast}$,
	$s\in\mathbb{N},$ $(X,\mathcal{B},\mu,(T_{g})_{g\in\mathbb{Z}^{d}})$ a $\mathbb{Z}^{d}$-system,
	$q_{1},\dots,q_{\ell}\colon(\mathbb{Z}^{L})^{s+1}\to\mathbb{Z}^{d}$ polynomials and  $\q=(q_{1},\dots,q_{\ell})$. If $A=(L,s,\ell,\q)$ with $\partial_{w}A=(L,s+1,\ell^{\ast},\q^{\ast})$ for some  $\ell^{\ast}\in\mathbb{N}^{\ast}, 1\leq w\leq \ell$, where $\q^{\ast}=(q^{\ast}_{1},\dots,q^{\ast}_{\ell^{\ast}})$ for some polynomials $q^{\ast}_{1},\dots,q^{\ast}_{\ell^{\ast}}\colon(\mathbb{Z}^{L})^{s+2}\to\mathbb{Z}^{d},$  then for all $b,a_{1},\dots,a_{s+1}\in\mathbb{N}^{L}$ not all equal to ${\bf 0},$ we have that
	\begin{equation}\label{222}
	R_{\q^{\ast}}(b;a_{1},\dots,a_{s+1})\lesssim R_{\q}(b+a_{s+1};a_{1},\dots,a_{s}).
	\end{equation}
\end{proposition}
\begin{proof}
	For convenience we write $\q^{\ast}\approx (p_{1},\dots, p_{\ell'})$ for some polynomials $p_{1},\dots, p_{\ell'}$ if $\q^{\ast}$ can be obtained by removing all the essential constant polynomials from $p_{1},\dots, p_{\ell'}$, ordering the rest into groups such that two polynomials are essentially distinct if and only if they are in different groups, and then picking one polynomial from each group. It is not hard to see that if $\q\approx \q'$, then $R_{\q}(b;a_{1},\dots,a_{s+1})=R_{\q'}(b;a_{1},\dots,a_{s+1})$ for all $b,a_{1},\dots,a_{s+1}\in\mathbb{N}^{L}$ not all equal to $\textbf{0}$.
	
	Denote $q'_{i}\colon(\mathbb{Z}^{L})^{s+2}\to\mathbb{Z}^{d}$, $q'_{i}(n;h_{1},\dots,h_{s+1})=q_{i}(n+h_{s+1};h_{1},\dots,h_{s})$ for all $1\leq i\leq\ell$. 
	It suffices to show the statement for $\q^{\ast}\approx(q'_{1}-q_{1},q_{i}-q_{1},q'_{i}-q_{1}\colon i\neq 1)$
	as the general case follows similarly.
	
	Suppose that 
	\[q_{i}(n;h_{1},\dots,h_{s})=\sum_{b,a_{1},\dots,a_{s}\in\mathbb{N}^{L}} h^{a_{1}}_{1}\dots h^{a_{s}}_{s}n^{b}\cdot u_{i}(b;a_{1},\dots,a_{s})\]
	for all
	$1\leq i\leq\ell$. Then, one can immediately check that
	\[q'_{i}(n;h_{1},\dots,h_{s+1})=\sum_{b,a_{1},\dots,a_{s+1}\in\mathbb{N}^{L}} h^{a_{1}}_{1}\dots h^{a_{s+1}}_{s+1}n^{b}\cdot \binom{b+a_{s+1}}{b}u_{i}(b+a_{s+1};a_{1},\dots,a_{s}).\footnote{ For $a=(a_{1},\dots,a_{L}), b=(b_{1},\dots,b_{L})\in\mathbb{N}^{L}$, $\binom{a}{b}$ denotes the quantity $\prod_{m=1}^{L}\binom{a_{m}}{b_{m}}$.}\]

	If $a_{s+1}=0$, then the coefficient of $h^{a_{1}}_{1}\dots h^{a_{s}}_{s}n^{b}$ for $q'_{1}-q_{1}$ is ${\bf 0},$ and for both $q_{i}-q_{1}$ and $q'_{i}-q_{1}$ are $u_{i}(b;a_{1},\dots,a_{s})-u_{1}(b;a_{1},\dots,a_{s})$.
	This implies that
	$R_{\q^{\ast}}(b;a_{1},\dots,a_{s},0)=R_{\q'}(b;a_{1},\dots,a_{s},0)\lesssim R_{\q}(b;a_{1},\dots,a_{s}),$\footnote{ Note that $R_{\q'}(b;a_{1},\dots,a_{s},0)\sim R_{\q}(b;a_{1},\dots,a_{s})$ if and only if one of $u_{i}(b;a_{1},\dots,a_{s})$ is ${\bf 0}$.}
	which proves (\ref{222}).
	
	\noindent If $a_{s+1}>0$, then the coefficient of $h^{a_{1}}_{1}\dots h^{a_{s+1}}_{s+1}n^{b}$ for $q'_{1}-q_{1},q_{i}-q_{1}$ and $q'_{i}-q_{1}$ are
	$\binom{b+a_{s+1}}{b}u_{1}(b+a_{s+1};a_{1},\dots,a_{s})$, $ {\bf 0}$ and $ \binom{b+a_{s+1}}{b}u_{i}(b+a_{s+1};a_{1},\dots,a_{s})$ respectively. 
	In this case 
	$R_{\q^{\ast}}(b;a_{1},$ $\dots,a_{s+1})$ $=R_{\q'}(b;a_{1},\dots,a_{s+1})\sim R_{\q}(b+a_{s+1};a_{1},\dots,a_{s})$,
	which finishes the proof.
\end{proof}

Let $A$ be a PET-tuple and $f\in L^{\infty}(\mu)$.
If $A$ is semi-standard but not standard for $f$, then the PET-induction does not work well enough to provide an upper bound for $S(A,\kappa)$ in terms of the Host-Kra seminorms of $f$. To overcome this difficulty, we use a ``dimension-increment'' argument to change $A$ into a new PET-tuple which is standard for $f$, but at the cost of increasing the dimension from $L$ to $2L$.\footnote{ In the papers \cite{Jo,L9}, where similar methods were used, the dimension was increased from $L$ to $3L$ instead.} In fact, this is the main reason that justifies the multi-variable nature of the results in this article.  

This ``dimension-increment'' argument is carried out in the following proposition. The idea essentially comes from \cite{Jo,L9} but, again, some additional work needs to be done in order to keep track of the set $R_{\q}(b;a_{1},\dots,a_{s})$. 

\begin{proposition}[Dimension-increasing property]\label{ext}
	Let $L,d,\ell\in\mathbb{N}^{\ast},$ $s\in\mathbb{N},$ $(X,\mathcal{B},\mu,(T_{g})_{g\in\mathbb{Z}^{d}})$ a $\mathbb{Z}^{d}$-system, $f\in L^{\infty}(\mu)$, $q_{1},\dots,q_{\ell}\colon(\mathbb{Z}^{L})^{s+1}\to\mathbb{Z}^{d}$ polynomials,
	$g_{1},\dots,g_{\ell}\colon X\times (\mathbb{Z}^{L})^{s}\to\mathbb{R}$ functions such that each $g_{i}(\cdot;h_{1},\dots,h_{s})$ is an $L^{\infty}(\mu)$ function bounded by $1$ for all $h_{1},\dots,h_{s}\in\mathbb{Z}^{L}, 1\leq i\leq \ell$, and let  $\q=(q_{1},\dots,q_{\ell})$ and $\g=(g_{1},\dots,g_{\ell})$. 
	
	If the PET-tuple $A=(L,s,\ell,\g,\q)$ is non-degenerate and semi-standard but not standard for $f$, then there exist polynomials $q'_{1},\dots,q'_{2\ell-1}\colon(\mathbb{Z}^{2L})^{s+1}\to\mathbb{Z}^{d}$, functions
	$g'_{1},\dots,g'_{2\ell-1}\colon X\times (\mathbb{Z}^{2L})^{s}\to\mathbb{R}$ such that each $g'_{i}(\cdot;h_{1},\dots,h_{s})$ is an $L^{\infty}(\mu)$ function bounded by $1$ for all $h_{1},\dots,h_{s}\in\mathbb{Z}^{2L}, 1\leq i\leq 2\ell-1,$  $\q'=(q'_{1},\dots,q'_{2\ell-1})$ and $\g'=(g'_{1},\dots,g'_{2\ell-1})$, such that
	the PET-tuple $A'=(2L,s,2\ell-1,\g',\q')$ is non-degenerate and standard for $f$ and $S(A,2\kappa)\leq S(A',\kappa)$ for all $\kappa>0$. Moreover, for all $b,b',a_{1},\dots,a_{s},$ $a'_{1},\dots,a'_{s}\in\mathbb{N}^{L}$ not all equal to ${\bf 0}$, there exist $b'',a''_{1},\dots,a''_{s}\in\mathbb{N}^{L}$ not all equal to ${\bf 0},$ such that
	\begin{equation}\label{321}
	R_{\q'}(b,b';a_{1},\dots,a_{s},a'_{1},\dots,a'_{s})\sim R_{\q}(b'';a''_{1},\dots,a''_{s}).
	\end{equation}		
\end{proposition}

\begin{proof}
	Since $A$ is semi-standard but not standard for $f$,
	we may assume without loss of generality that $g_{1}(x;h_{1},\dots,h_{s})=f(x)$,
	$\deg(q_{1})<\deg(A)$, and $\deg(q_{\ell})=\deg(A)$.
	For convenience denote ${\bf h}=(h_{1},\dots,h_{s})$ and ${\bf h}'=(h'_{1},\dots,h'_{s})$. For $1\leq m\leq \ell$ we set \[ q'_{m}((n,n');({\bf h},{\bf h}'))\coloneqq q_{m}(n;{\bf h})-q_{\ell}(n';{\bf h}),\footnote{ The notion $({\bf h},{\bf h}')$ refers to the vector $((h_{1},h'_{1}),\dots,(h_{s},h'_{s}))\in (\Z^{2L})^{s}$, which we use to simplify the notation.}\quad \text{and}\quad  g'_{m}(x;({\bf h},{\bf h}'))\coloneqq g_{m}(x;{\bf h}),\] while for $1\leq m\leq \ell-1$ we set  \[q'_{m+\ell}((n,n');({\bf h},{\bf h}'))\coloneqq q_{m}(n';{\bf h})-q_{\ell}(n';{\bf h}), \quad \text{and} \quad  g'_{m+\ell}(x;({\bf h},{\bf h}'))\coloneqq g_{m}(x;{\bf h}).\] Also, let $\q'=(q'_{1},\dots,q'_{2\ell-1}),$ $\g'=(g'_{1},\dots,g'_{2\ell-1})$ and $A'=(2L,s,2\ell-1,\q',\g')$.
	
	Since $\deg(q_{\ell})=\deg(A)$ and $\deg(q_{1})<\deg(A)$, we have that $\deg(q'_{1})=\deg(A')$ and moreover $g'_{1}=f$. So $A'$ is standard for $f$.
	On the other hand,
	since $A$ is non-degenerate, one can easily see that $q'_{1},\dots,q'_{2\ell-1}$ are essentially distinct (note that $q_{\ell}(n;{\bf h})-q_{\ell}(n';{\bf h})$ is essentially non-constant).  So  $A'$ is non-degenerate.
	
	Recall that $\overline{\mathbb{E}}^{\square}_{{\bf h}\in(\mathbb{Z}^{L})^{s}}=\overline{\mathbb{E}}^{\square}_{h_{1}\in\mathbb{Z}^{L}}\dots \overline{\mathbb{E}}^{\square}_{h_{s}\in\mathbb{Z}^{L}}.$
	By the fact that the action is measure preserving and the Cauchy-Schwarz inequality, we have that 
	\begin{equation}\nonumber
	\begin{split}
	& S(A,2\kappa)=\overline{\mathbb{E}}^{\square}_{{\bf h}\in(\mathbb{Z}^{L})^{s}}\F\vl\Bigl\Vert\mathbb{E}_{n\in I_{N}}\prod_{m=1}^{\ell}T_{q_{m}(n;{\bf h})}g_{m}(x;{\bf h})\Bigr\Vert^{2\kappa}_{L^{2}(\mu)}
	\\= & \overline{\mathbb{E}}^{\square}_{{\bf h}\in(\mathbb{Z}^{L})^{s}}\F\vl\Bigl\vert\mathbb{E}_{n,n'\in I_{N}}\int_{X}\prod_{m=1}^{\ell} T_{q_{m}(n;{\bf h})}g_{m}(x;{\bf h})\cdot T_{q_{m}(n';{\bf h})}g_{m}(x;{\bf h})\,d\mu\Bigr\vert^{\kappa}
	\\\leq & \overline{\mathbb{E}}^{\square}_{{\bf h}\in(\mathbb{Z}^{L})^{s}}\F\vl\Bigl\Vert\mathbb{E}_{n,n'\in I_{N}}\prod_{m=1}^{\ell}T_{q_{m}(n;{\bf h})-q_{\ell}(n';{\bf h})}g_{m}(x;{\bf h})
	\\ & \hspace{9cm} 
	\cdot\prod_{m=1}^{\ell-1} T_{q_{m}(n';{\bf h})-q_{\ell}(n';{\bf h})}g_{m}(x;{\bf h})\Bigr\Vert_{L^{2}(\mu)}^{\kappa}
		\\ = &\overline{\mathbb{E}}^{\square}_{({\bf h},{\bf h}')\in(\mathbb{Z}^{2L})^{s}}\F\vl\Bigl\Vert\mathbb{E}_{n,n'\in I_{N}}\prod_{m=1}^{2\ell-1}T_{q'_{m}((n,n');({\bf h},{\bf h}'))}g'_{m}(x;({\bf h},{\bf h}'))\Bigr\Vert_{L^{2}(\mu)}^{\kappa} 
		\\ \leq &  S(A',\kappa), 
	\end{split}
	\end{equation}
	where the last inequality holds because $(I_{N}\times I_{N})_{N\in\N}$ is a F\o lner sequence in $\Z^{L}\times \Z^{L}$.
	On the other hand, if 
	\[q_{i}(n;{\bf h})=\sum_{b,a_{1},\dots,a_{s}\in\mathbb{N}^{L}} h^{a_{1}}_{1}\dots h^{a_{s}}_{s}n^{b}\cdot u_{i}(b;a_{1},\dots,a_{s})\]
	for some $u_{i}(b;a_{1},\dots,a_{s})\in\mathbb{Q}^{d}$, then for $1\leq i\leq \ell-1$, we have
	\[q'_{i+\ell}(n,n';{\bf h},{\bf h}')=\sum_{b,a_{1},\dots,a_{s}\in\mathbb{N}^{L}} h^{a_{1}}_{1}\dots h^{a_{s}}_{s}{n'}^{b}\cdot (u_{i}(b;a_{1},\dots,a_{s})-u_{\ell}(b;a_{1},\dots,a_{s})),\]
	and for $1\leq i\leq \ell$,
	\[q'_{i}(n,n';{\bf h},{\bf h}')=\sum_{b,a_{1},\dots,a_{s}\in\mathbb{N}^{L}} h^{a_{1}}_{1}\dots h^{a_{s}}_{s}(n^{b}\cdot u_{i}(b;a_{1},\dots,a_{s})-{n'}^{b}\cdot u_{\ell}(b;a_{1},\dots,a_{s})).\]
	So, for all $b,b',a_{1},\dots,a_{s},a'_{1},\dots,a'_{s}\in\mathbb{N}^{L}$, similarly to the argument in the proof of Proposition~\ref{1234}, we have \[R_{\q}(b;a_{1},\dots,a_{s})=R_{\q'}((b,{\bf 0});(a_{1},{\bf 0}),\dots,(a_{s},{\bf 0}))\sim R_{\q'}(({\bf 0},b);(a_{1},{\bf 0}),\dots,(a_{s},{\bf 0}))\]
	and
	$R_{\q'}((b,b');(a_{1},a'_{1}),\dots,(a_{s},a'_{s}))=\{{\bf 0}\}$. This implies (\ref{321}) and finishes the proof.
\end{proof}

We are now ready to prove Proposition~\ref{pet} and close this section.

\begin{proof}[Proof of Proposition~\ref{pet}]
	Let $A$ denote the PET-tuple $(L,0,k,(p_{1},\dots,p_{k}),(f_{1},\dots,f_{k}))$. Then, for all $\kappa>0$,
	\[S(A,\kappa)=\F\vl\Bigl\Vert\ei\prod_{m=1}^{k}T_{p_{m}(n)}f_{m}\Bigr\Vert^{\kappa}_{L^{2}(\mu)}.\]
	By the assumption, $A$ is non-degenerate.
	We only prove (\ref{30}) for $f_{1}$ as the other cases are identical.
	
	We first assume that $A$ is standard for $f_{1}$. By Theorem \ref{PET}, there exist finitely many
	vdC-operations $\partial_{\rho_{1}},\dots,\partial_{\rho_{t}}$ such that $A'=\partial_{\rho_{t}}\dots\partial_{\rho_{1}}A$ is a  non-degenerate PET-tuple which is standard for $f_{1}$, and $\deg(A')=1$. 
	By Proposition~\ref{induction}, $S(A, 2^{t})\leq C\cdot S(A',1)$ for some $C>0$ depending only on the polynomials $p_{1},\dots,p_{k}$. We may assume that
	\begin{equation}\nonumber
	\begin{split}
	&S(A',1)=\overline{\mathbb{E}}^{\square}_{h_{1},\dots,h_{s}\in\mathbb{Z}^{L}}\F\vl\Bigl\Vert\ei \prod_{m=1}^{\ell}T_{\textbf{d}_{m}(h_{1},\dots,h_{s})\cdot n+r_{m}(h_{1},\dots,h_{s})}g_{m}(x;h_{1},\dots,h_{s})\Bigr\Vert_{L^{2}(\mu)}
	\end{split}
	\end{equation}
	for some $s,\ell\in\mathbb{N}^\ast$, functions $g_{1},\dots,g_{\ell}\colon X\times (\mathbb{Z}^{L})^{s}\to\mathbb{R},$ where $g_{1}(\cdot;h_{1},\dots,h_{s})=f_1,$ such that each $g_{m}(\cdot;h_{1},\dots,h_{s})$ is an $L^{\infty}(\mu)$ function bounded by $1,$ and polynomials $\textbf{d}_{m}\colon(\mathbb{Z}^{L})^{s}\to(\mathbb{Z}^{d})^{L}$ and $r_{m}\colon(\mathbb{Z}^{L})^{s}\to\mathbb{Z}^{d},$ $1\leq m\leq \ell$, where $\textbf{d}_{m}, r_{m}$ take values in vectors with integer coordinates because the vdC-operations send integer-valued polynomials to integer-valued polynomials. Let $\c_{1,1}=-\textbf{d}_{1}$ and $\c_{1,m}=\textbf{d}_{m}-\textbf{d}_{1}$ for $m\neq 1$.
	Since $A'$ is non-degenerate, we have that $\c_{1,1},\dots,\c_{1,s}\not\equiv{\bf 0}$.
		
	If $\ell\geq 2$,
	by Proposition~\ref{P:1},  we have that 
	\begin{eqnarray*}
	S(A',1) & \leq & C'\cdot \overline{\mathbb{E}}^{\square}_{h_{1},\dots,h_{s}\in\mathbb{Z}^{L}}\Vert T_{r_{1}(h_{1},\dots,h_{s})}f_1\Vert_{\{G'(\c_{1,i}(h_{1},\dots,h_{s}))\}_{1\leq i\leq \ell}} \\
	& = & C'\cdot \overline{\mathbb{E}}^{\square}_{h_{1},\dots,h_{s}\in\mathbb{Z}^{L}}\Vert f_1\Vert_{\{G'(\c_{1,i}(h_{1},\dots,h_{s}))\}_{1\leq i\leq \ell}}
	\end{eqnarray*}
	for some $C'>0$ depending only on $\ell$ (which depends only on the polynomials $p_{1},\dots,p_{k}$).
	If $\ell=1$, by Theorem~\ref{erg} and Lemma~\ref{replacement0} (iv), (vi), we have
 		\begin{eqnarray*}
 			S(A',1) 
 			& = &  \overline{\mathbb{E}}^{\square}_{h_{1},\dots,h_{s}\in\mathbb{Z}^{L}} \Vert \mathbb{E}(T_{r_{1}(h_{1},\dots,h_{s})}f_1\vert \mathcal{I}(\c_{1,1}(h_{1},\dots,h_{s})))\Vert_{2} \\
 				& = &  \overline{\mathbb{E}}^{\square}_{h_{1},\dots,h_{s}\in\mathbb{Z}^{L}} \Vert \mathbb{E}(f_1\vert \mathcal{I}(\c_{1,1}(h_{1},\dots,h_{s})))\Vert_{2} \\
 			& = &  \overline{\mathbb{E}}^{\square}_{h_{1},\dots,h_{s}\in\mathbb{Z}^{L}} \Vert{f_1}\Vert_{G'(\c_{1,1}(h_{1},\dots,h_{s}))}.
 		\end{eqnarray*}
	In both cases, we get (\ref{30}) since $S(A,2^{t})\leq C\cdot S(A',1)$.

	Suppose that 
	\[\c_{1,m}(h_{1},\dots,h_{s})=\sum_{a_{1},\dots,a_{s}\in\mathbb{N}^{L}}h^{a_{1}}_{1}\dots h^{a_{s}}_{s}\cdot \u_{1,m}(a_{1},\dots,a_{s}),\;\;\text{ and}\]
	\[\textbf{d}_{m}(h_{1},\dots,h_{s})=\sum_{a_{1},\dots,a_{s}\in\mathbb{N}^L}h^{a_{1}}_{1}\dots h^{a_{s}}_{s}\cdot \v_{m}(a_{1},\dots,a_{s})\]
	for some  $\u_{1,m}(a_{1},\dots,a_{s}),\v_{m}(a_{1},\dots,a_{s})\in(\mathbb{Q}^{d})^{L}$ with all but finitely many terms being {\bf 0} for each $m$. Write $\u_{1,m}(a_{1},\dots,a_{s})=(u_{1,m,1}(a_{1},\dots,a_{s}),\dots,u_{1,m,L}(a_{1},\dots,a_{s})),$ $\v_{m}(a_{1},\dots,a_{s})=(v_{m,1}(a_{1},\dots,a_{s}),$ $\dots,v_{m,L}(a_{1},\dots,a_{s})),$ and, for all $1\leq r\leq \ell,$ set \[U_{1,r}(a_{1},\dots,a_{s})\coloneqq \{u_{1,m,r}(a_{1},\dots,a_{s})\in\mathbb{Q}^{d}\colon 1\leq m\leq \ell\}\cup\{{\bf 0}\};\quad\text{and}\quad\]
	\[V_{r}(a_{1},\dots,a_{s})\coloneqq \{v_{m,r}(a_{1},\dots,a_{s})\in\mathbb{Q}^{d}\colon 1\leq m\leq \ell\}\cup\{{\bf 0}\}.\]
	Since $A'=\partial_{\rho_{t}}\dots\partial_{\rho_{1}}A$, by repeatedly using Proposition~\ref{1234}, for all $a_{1},\dots,a_{s}\in\mathbb{N}^{L}$ not all equal to ${\bf 0}$ and every $1\leq r\leq L$, there exists $v\in\mathbb{N}^{L}, v\neq {\bf 0}$ such that $V_{r}(a_{1},\dots,a_{s})\lesssim R_{v}$. By the relation between $\u_{1,m}$ and $\v_{m}$, we get $U_{1,r}(a_{1},\dots,a_{s})\sim V_{r}(a_{1},\dots,a_{s})$  and so $U_{1,r}(a_{1},\dots,a_{s})\lesssim R_{v}$.
	
	We now assume that $A=(L,0,k,(p_{1},\dots,p_{k}),(f_{1},\dots,f_{k}))$ is not standard for $f_{1}$. Since $A$ is semi-standard for $f_{1}$, by Proposition~\ref{ext},  there exists a PET-tuple $A'=(2L,0,\ell,\q,\g)$ which is non-degenerate and standard for $f_{1}$ such that $S(A,2\kappa)\leq S(A',\kappa)$ for all $\kappa>0$ and (\ref{321}) holds. Working with the PET-tuple $A'$ instead of $A$ as before (and using (\ref{321})), we get the result.
\end{proof}	

\section{Proof of Proposition \ref{pet3}}\label{s:pet3}

This last section is dedicated to the proof of Proposition~\ref{pet3}.
 If $s=0$, then there is nothing to prove. So we assume that $s\in\mathbb{N}^\ast$.
We remark that it is in this proposition  where the concatenation results (Theorem~\ref{ct0} and Corollary~\ref{ct}) are used.

Following the notation of Proposition~\ref{pet3}, for every ${\bf h}=(h_{1},\dots,h_{s})\in(\mathbb{Z}^{L})^{s}$ and $1\leq i\leq k$, we set
\[W_{i,{\bf h}}\coloneqq Z_{G(\c_{i,1}({\bf h})),\dots,G(\c_{i,t_{i}}({\bf h}))}(\X),\]
and for every subset $J\subseteq(\mathbb{Z}^{L})^{s},$
\[W_{i,J}\coloneqq \bigvee_{{\bf h}\in J}W_{i,{\bf h}}.\] The following lemma informs us that we can assume that the functions $f_i$ are measurable with respect to some $W_{i,J_i}.$

\begin{lemma}\label{pet2}
	Let the notation be as in Proposition~\ref{pet3}  with $s\in\mathbb{N}^\ast$. If (\ref{35}) holds
	for every $\mathbb{Z}^{d}$-system $(X,\mathcal{B},\mu,$ $ (T_{g})_{g\in\mathbb{Z}^{d}})$ and every $f_{1},\dots,f_{k}\in L^{\infty}(\mu)$, then for every $J_{1},\dots,J_{k}\subseteq (\mathbb{Z}^{L})^{s}$ of density $1,$ we have that 
	\begin{equation}\label{11}
	\begin{split}
	\mathbb{E}_{n\in\mathbb{Z}^{L}} T_{p_{1}(n)}f_1\cdot\ldots\cdot T_{p_{k}(n)}f_k=0,\;\; \text{ if $\;\;\mathbb{E}(f_i\vert W_{i,J_{i}})=0$ for some $1\leq i\leq k$.}
	\end{split}
	\end{equation}
\end{lemma}

\begin{proof}
By Lemma \ref{replacement0} (vi), we may assume without loss of generality that $t_{i}\geq 2$ in (\ref{35}).
Suppose that
	 $\mathbb{E}(f_i\vert W_{i,J_{i}})=0$ for some $1\leq i\leq k$. Then $\mathbb{E}(f_i\vert Z_{G(c_{i,1}({\bf h})),\dots,G(c_{i,t_{i}}({\bf h}))})=0$ for all ${\bf h}\in J_{i}$. %\footnote{Here we can assume without loss of generality that $t_{i}\geq 2,$ which is a technical assumption that allows us to use Lemma \ref{replacement0} (iv) conveniently.} 
	 Since $G'(c_{i,j}(\bold{h}))$ is a finite index subgroup of $G(c_{i,j}(\bold{h}))$, by Lemma \ref{replacement0} (iv), we have $\mathbb{E}(f_i\vert Z_{G'(c_{i,1}({\bf h})),\dots,G'(c_{i,t_{i}}({\bf h}))})=0$ and thus $\Vert f_{i}\Vert_{G'(c_{i,1}({\bf h})),\dots,G'(c_{i,t_{i}}({\bf h}))}=0$ for all ${\bf h}\in J_{i}$. Since $J_{i}$ is of density $1,$ the conclusion follows from (\ref{35}). 
\end{proof}

Before proving Proposition~\ref{pet3}, we continue with our main example (Example~\ref{Eg:1}).

\medskip

\noindent {\bf Third part of computations for Example~\ref{Eg:1}:} 
We are dealing with the $(T^{n^{2}+n}_{1},T^{n^{2}}_{2})$ case.
Applying (\ref{101}) to Lemma~\ref{pet2}, we have that
\begin{equation}\label{201}
\begin{split}
\mathbb{E}_{n\in\Z} T^{n^{2}+n}_{1}f_1\cdot T^{n^{2}}_{2}f_2=0,\;\; \text{ if $\;\;\mathbb{E}(f_i\vert W_{i,J_{i}})=0$ for $i=1$ or 2,}
\end{split}
\end{equation}
for all $J_{1},J_{2}\in\mathbb{Z}^{3}$ of density $1,$ where
$$W_{i,J}=\bigvee_{(h_{1},h_{2},h_{3})\in J}W_{i,(h_{1},h_{2},h_{3})}=\bigvee_{(h_{1},h_{2},h_{3})\in J}Z_{G(\c_{i,1}(h_{1},h_{2},h_{3})),\dots,G(\c_{i,7}(h_{1},h_{2},h_{3}))},\quad i=1,2,$$ 
where $\c_{i,j}\colon\Z^{3}\to\Z^{2}$ are the ones in the second part of computations for Example~\ref{Eg:1}.

Recall that $e_{1}=(1,0),$ $ e_{2}=(0,1),$ $ e=e_1-e_2$.
In this case, we have that $H_{1,1}=\mathbb{Z}e_{1}$, $H_{1,3}=H_{1,5}=H_{1,7}=\mathbb{Z}e$ and $H_{1,2}=H_{1,4}=H_{1,6}=\mathbb{Z}^{2}$.
Moreover,  $H_{2,1}=\mathbb{Z}e_{2}$, $H_{2,2}=H_{2,4}=H_{2,6}=\mathbb{Z}e$ and $H_{2,3}=H_{2,5}=H_{2,7}=\mathbb{Z}^2$.

From (\ref{101}) and Lemma \ref{pet2}, if $\mathbb{E}(f_1\vert W_{1,\mathbb{Z}^{3}})=0$, then  $\F\vl\Bigl\Vert\mathbb{E}_{n\in I_{N}} T_1^{n^2+n}f_1\cdot T_2^{n^2}f_2\Bigr\Vert_{L^{2}(\mu)}=0.$
In the general case, by decomposing $f_{1}$ and $\mathbb{E}(f_1\vert W_{1,\mathbb{Z}^{3}})+(f_{1}-\mathbb{E}(f_1\vert W_{1,\mathbb{Z}^{3}}))$,
we can deduce that
\begin{equation}\label{21}
\F\vl\Bigl\Vert\mathbb{E}_{n\in I_{N}} \Bigl(T_1^{n^2+n}f_1\cdot T_2^{n^2}f_2-T_1^{n^2+n}\mathbb{E}(f_1\vert W_{1,\mathbb{Z}^{3}})\cdot T_2^{n^2}f_{2}\Bigr)\Bigr\Vert_{L^{2}(\mu)}=0.
\end{equation}
Fix $\varepsilon>0$.
Since $W_{1,\Z^{3}}=\bigvee_{N=1}^{\infty}W_{1,[-N,N]^{3}}$, by approximation, there exists a finite subset $I$ of $\mathbb{Z}^{3}$ such that $\Vert\mathbb{E}(f_1\vert W_{1,\mathbb{Z}^{3}})-\mathbb{E}(f_1\vert W_{1,I})\Vert_{L^{1}(\mu)}<\varepsilon^{2}/2$.
Since $\Vert f_{1}\Vert_{L^{\infty}(\mu)},\Vert f_{2}\Vert_{L^{\infty}(\mu)}\leq 1$, for all $n\in\Z$,
\begin{equation}\nonumber
	\begin{split}
	&\Bigr\Vert\Bigl(T_1^{n^2+n}\mathbb{E}(f_1\vert W_{1,\mathbb{Z}^{3}})\cdot T_2^{n^2}f_{2}- T_1^{n^2+n}\mathbb{E}(f_1\vert W_{1,I})\cdot T_2^{n^2}f_{2}\Bigr)\Bigr\Vert^{2}_{L^{2}(\mu)}
	\\&=\int_{X}\Bigl(T_1^{n^2+n}\mathbb{E}(f_1\vert W_{1,\mathbb{Z}^{3}})\cdot T_2^{n^2}f_{2}- T_1^{n^2+n}\mathbb{E}(f_1\vert W_{1,I})\cdot T_2^{n^2}f_{2}\Bigr)^{2}\,d\mu
	\\&\leq\int_{X}2\Bigl\vert T_1^{n^2+n}\mathbb{E}(f_1\vert W_{1,\mathbb{Z}^{3}})- T_1^{n^2+n}\mathbb{E}(f_1\vert W_{1,I})\Bigr\vert\,d\mu
	=\int_{X}2\Bigl\vert \mathbb{E}(f_1\vert W_{1,\mathbb{Z}^{3}})- \mathbb{E}(f_1\vert W_{1,I})\Bigr\vert\,d\mu
	<\varepsilon^{2}.
	\end{split}
\end{equation}
So,
\begin{equation}\label{31}
\F\vl\Bigl\Vert\mathbb{E}_{n\in I_{N}} \Bigl(T_1^{n^2+n}\mathbb{E}(f_1\vert W_{1,\mathbb{Z}^{3}})\cdot T_2^{n^2}f_{2}- T_1^{n^2+n}\mathbb{E}(f_1\vert W_{1,I})\cdot T_2^{n^2}f_{2}\Bigr)\Bigr\Vert_{L^{2}(\mu)}
<\varepsilon.
\end{equation}
Note that $W_{1,I}$ is contained in the $(7\vert I\vert)$-step factor
$$W'_{1}\coloneqq Z_{(G(\c_{1,1}(h_{1},h_{2},h_{3})),\dots,G(\c_{1,7}(h_{1},h_{2},h_{3})))_{(h_{1},h_{2},h_{3})\in I}}.$$

We say that $(h'_{1},h'_{2},h'_{3})\in\mathbb{Z}^{3}$ is \emph{good} if for any $(h_{1},h_{2},h_{3})\in I$, any 
$$g\in\{-2h_{1}e_{1},2h_{2}e-2h_1e_1, 2h_{2}e,2h_{3}e-2h_1e_1,2h_{3}e,2(h_{2}+h_{3})e-2h_1e_1, 2(h_{2}+h_{3})e\}$$
(\emph{i.e.}, $g$ is the generator of one of $G(\c_{1,1}(h_{1},h_{2},h_{3})),\dots,G(\c_{1,7}(h_{1},h_{2},h_{3}))$)
and any action
\[g'\in\{-2h'_{1}e_{1},2h'_{2}e-2h'_1e_1, 2h'_{2}e,2h'_{3}e-2h'_1e_1,2h'_{3}e,2(h'_{2}+h'_{3})e-2h'_1e_1, 2(h'_{2}+h'_{3})e\},\]
(\emph{i.e.}, $g'$ is the generator of one of $G(\c_{1,1}(h'_{1},h'_{2},h'_{3})),\dots,G(\c_{1,7}(h'_{1},h'_{2},h'_{3}))$)
the set $$H\coloneqq \text{span}_{\mathbb{Q}}\{g,g'\}\cap\mathbb{Z}^{2}$$ satisfies the following: 
$$\left\{ \begin{array}{ll} H=\mathbb{Z}e_{1} & \; ,\text{ if } g=-2h_{1}e_{1}, g'=-2h'_{1}e_{1} 
\\ H=\mathbb{Z}e  & \; ,\text{ if } g\in\{2h_{2}e,2h_{3}e,2(h_{2}+h_{3})e\}, g'\in\{2h'_{2}e,2h'_{3}e,2(h'_{2}+h'_{3})e\}
\\ H=\mathbb{Z}^{2}  & \; ,\text{ otherwise}
\end{array} \right. .$$

Let $J$ be the set of all good tuples. Since $I$ is finite, it is not hard to show that $J$ is of density 1 (see also the claim in the proof of Proposition~\ref{pet3}). Similar to the way we obtained (\ref{21}), using (\ref{101}) and Lemma \ref{pet2}, we can deduce  
\begin{equation}\label{41}
\F\vl\Bigl\Vert\mathbb{E}_{n\in I_{N}} \Bigl(T_1^{n^2+n}\mathbb{E}(f_1\vert W_{1,I})\cdot T_2^{n^2}f_{2}- T_1^{n^2+n}\mathbb{E}(f_1\vert W_{1,J}\cap W_{1,I})\cdot T_2^{n^2}f_{2}\Bigr)\Bigr\Vert_{L^{2}(\mu)}=0.
\end{equation}
So,  (\ref{21}), (\ref{31}) and (\ref{41}) imply that
\begin{equation}\label{42}
	\begin{split}
		\F\vl\Bigl\Vert\mathbb{E}_{n\in I_{N}}\Bigl(T_1^{n^2+n} f_1\cdot T_2^{n^2}f_{2}
		-T_1^{n^2+n}\mathbb{E}(f_1\vert W_{1,J}\cap W_{1,I})\cdot T_2^{n^2}f_{2}\Bigr)\Bigr\Vert_{L^{2}(\mu)}<\varepsilon.
	\end{split}
\end{equation}

By the definition of good tuples and Corollary \ref{ct}, we have that 
\begin{equation*}
\begin{split}
W_{1,J}\cap W_{1,I}  \subseteq \bigvee_{(h'_{1},h'_{2},h'_{3})\in J}W'_{1}\cap W_{1,(h'_{1},h'_{2},h'_{3})}
 =\bigvee_{(h'_{1},h'_{2},h'_{3})\in J}Z_{(\mathbb{Z}e_{1})^{\times \vert I\vert},(\mathbb{Z}e)^{\times 9\vert I\vert},(\mathbb{Z}^{2})^{\times 39\vert I\vert}}\subseteq Z_{e^{\times \infty},e_{1}^{\times \infty}}.
\end{split}
\end{equation*}
So, (\ref{42}) implies that 
\begin{equation}\nonumber
\begin{split}
\F\vl\Bigl\Vert\mathbb{E}_{n\in I_{N}} T_1^{n^2+n}f_{1}\cdot T_2^{n^2}f_{2}\Bigr\Vert_{L^{2}(\mu)}<\varepsilon,\;\; \text{ if $\;\;\mathbb{E}(f_1\vert Z_{e_{1}^{\times\infty},e^{\times\infty}})=0$.}
\end{split}
\end{equation}
Since $\varepsilon>0$ is arbitrary,
\begin{equation}\nonumber
	\begin{split}
		\F\vl\Bigl\Vert\mathbb{E}_{n\in I_{N}} T_1^{n^2+n}f_{1}\cdot T_2^{n^2}f_{2}\Bigr\Vert_{L^{2}(\mu)}=0,\;\; \text{ if $\;\;\mathbb{E}(f_1\vert Z_{e_{1}^{\times\infty},e^{\times\infty}})=0$.}
	\end{split}
\end{equation}
Working analogously for the $T_2^{n^2}f_2$ term, we eventually get that
\begin{equation}\label{400}
\begin{split}
	\F\vl\Bigl\Vert\mathbb{E}_{n\in I_{N}} T_1^{n^2+n}f_{1}\cdot T_2^{n^2}f_{2}\Bigr\Vert_{L^{2}(\mu)}=0\; \text{ if $\;\mathbb{E}(f_i\vert Z_{e_{i}^{\times\infty},e^{\times\infty}})=0\;$ for $i=1$ or $2$.}
\end{split}
\end{equation}
We remark that (\ref{400}) is a stronger version of (\ref{40}) (\emph{i.e.},  in the continuation of Example~\ref{Eg:1}).

\begin{remark*}
	As it was mentioned before,  the characteristic factors described in Theorem~\ref{T:2} are not the optimal ones in general, but they are sufficient for the needs of our study.
\end{remark*}

We briefly explain the idea on proving \cref{pet3}. Under the assumptions of Proposition~\ref{pet3},
Lemma~\ref{pet2} says that one can assume that $f_{1}$ is measurable with respect to the factor $W_{1,J_{1}}$. However,
thanks to the freedom of the choices of $J_{1}$, we can use Lemma~\ref{pet2} to repeatedly choose different subsets $J_{1,1}, \dots$, $J_{1,r}$, for some $r\in \mathbb{N}^\ast$, and assume that $f_{1}$ is measurable with respect to the factor $W_{1,J_{1,1}}\cap W_{1,J_{1,2}}\cap\dots\cap W_{1,J_{1,r}}$. We then employ the concatenation theorems to estimate the intersection of $W_{1,J_{1,j}}$'s, and find a smaller factor characterizing the multiple average we aim to study.

\begin{proof}[Proof of Proposition \ref{pet3}]
    By Lemma \ref{replacement0} (vi), duplicating $G'(\bold{c}_{i,m}(h_{1},\dots,h_{s}))$ if necessary, we may assume without loss of generality that $t_{i}\geq 2$.
	If $s=0$, then there is nothing to prove. So we assume that $s\in\mathbb{N}^\ast$.
	Let $(X,\mathcal{B},\mu, (T_{g})_{g\in\mathbb{Z}^{d}})$ be a $\mathbb{Z}^{d}$-system, $f_{1},\dots,f_{k}\in L^{\infty}(\mu)$ and $s,t_{1},\dots,$ $t_{k},\c_{i,m},$ $1\leq i\leq k,$ $1\leq m\leq t_i$  be as in the statement. By Lemma~\ref{ag2}, 
	\[H_{i,m}=\text{span}_{\mathbb{Q}}\{G(\c_{i,m}(h_{1},\dots,h_{s}))\colon h_{1},\dots,h_{s}\in\mathbb{Z}^{L}\}\cap\mathbb{Z}^{d}.\]
	To show (\ref{36}), it suffices to show that if $\mathbb{E}(f_{i}\vert Z_{(H_{i,1})^{\times \infty},\dots, (H_{i,t_{i}})^{\times \infty}})=0$ for some $1\leq i\leq k$, then the left hand side of (\ref{36}) equals to 0. We assume without loss of generality that $i=1$.

	For every $r\in\mathbb{N}$, every finite subset $I\subseteq \Z^{L}$, and every tuple $(J_{1},\dots,J_{r})$, where $ J_i\subseteq(\mathbb{Z}^{L})^{s}$, $1\leq i\leq r$, denote 
	\[A_{I}(J_{1},\dots,J_{r})\coloneqq \mathbb{E}_{n\in I} T_{p_{1}(n)}\mathbb{E}(f_1\vert W_{1,J_{1}}\cap\dots\cap  W_{1,J_{r}})\cdot T_{p_{2}(n)}f_{2}\cdot \ldots\cdot T_{p_{k}(n)}f_{k},\]
	and in the degenerated case, set 
	\[A_{I}(\emptyset)\coloneqq \mathbb{E}_{n\in I} T_{p_{1}(n)}f_1\cdot T_{p_{2}(n)}f_{2}\cdot\ldots\cdot T_{p_{k}(n)}f_{k}.\]
	We say that a tuple $(J_{1},\dots,J_{r})$ of subsets of $(\mathbb{Z}^{L})^{s}$ is \emph{admissible} if
	for every ${\bf h}_{u}\in J_{u}, 1\leq u\leq r$ and every $1\leq m\leq t_{1}$, 	denoting
	\begin{equation}\label{0010}
	G_{K}\coloneqq \text{span}_{\mathbb{Q}}\{G(\c_{1,m}({\bf h}_{u}))\colon u\in K\}\cap\mathbb{Z}^{d}
	\end{equation}	
	for all $K\subseteq \{1,\dots,r\}$, the following holds: 
	for all $\emptyset\neq K'\subsetneq K\subseteq \{1,\dots,r\}$ such that $\max\{x\in K'\}<\min\{x\in K\backslash K'\}$, either $G_{K'}\subsetneq G_{K}$ or $G_{K'}=H_{1,m}$.\footnote{ We think of this as a notion of having ``full  rank''.}

	Fix $\varepsilon>0$.
	By (\ref{35}), we have that $$\F\vl\Bigl\Vert A_{I_{N}}(\emptyset)-A_{I_{N}}((\mathbb{Z}^{L})^{s})\Bigr\Vert_{L^{2}(\mu)}=0.$$ By an approximation argument similar to the one that we used to obtain (\ref{31}), there exists a finite subset $J_{1}'\subseteq (\mathbb{Z}^{L})^{s}$ such that $$\F\vl\Bigl\Vert A_{I_{N}}((\mathbb{Z}^{L})^{s})-A_{I_{N}}(J'_{1})\Bigr\Vert_{L^{2}(\mu)}<\varepsilon,$$ and so $$\F\vl\Bigl\Vert A_{I_{N}}(\emptyset)-A_{I_{N}}(J'_{1})\Bigr\Vert_{L^{2}(\mu)}<\varepsilon.$$
	Note that the induction basis is ensured as $J_1'$ is automatically admissible. Suppose now that for some $r\geq 1$, we have constructed finite subsets 
	$J'_{1},\dots,J'_{r}\subseteq(\mathbb{Z}^{L})^{s}$ such that:
	\begin{itemize}
		\item[(i)] $\F\vl\Bigl\Vert A_{I_{N}}(\emptyset)-A_{I_{N}}(J'_{1},\dots,J'_{r})\Bigr\Vert_{L^{2}(\mu)}<r\varepsilon$; and
		\item[(ii)] $(J'_{1},\dots,J'_{r})$ is admissible.
	\end{itemize}	
	We construct $J'_{r+1}$. We first claim that there exists $J_{r+1}\subseteq (\mathbb{Z}^{L})^{s}$ of density $1$ such that $(J'_{1},\dots,J'_{r},J_{r+1})$ is admissible.  For every ${\bf h}_{u}\in J'_{u}, 1\leq u\leq r$, $1\leq m\leq t_{1}$ and nonempty subset $K\subseteq \{1,\dots,r\}$, let 
	\[Q_{m;{\bf h}_{1},\dots,{\bf h}_{r};K} \coloneqq \text{span}_{\mathbb{Q}}\{G(\c_{1,m}({\bf h}_{u}))\colon u\in K\}\cap\mathbb{Z}^{d}.\]
	If $Q_{m;{\bf h}_{1},\dots,{\bf h}_{r};K}=H_{1,m}$, we let  $V_{m;{\bf h}_{1},\dots,{\bf h}_{r};K}=(\mathbb{Z}^{L})^{s}$; otherwise
	$V_{m;{\bf h}_{1},\dots,{\bf h}_{r};K}$
	denotes the set of ${\bf h}=(h_{1},\dots,h_{s})\in(\mathbb{Z}^{L})^{s}$ such that $G(\c_{1,m}({\bf h}))$ is not contained in $Q_{m;{\bf h}_{1},\dots,{\bf h}_{r};K}$.
	Let \[J_{r+1}\coloneqq\bigcap_{ \substack{ {\bf h}_{u}\in I_{u},\; 1\leq u\leq r,\;  1\leq m\leq t_{1},\;  K\subseteq \{1,\dots,r\}} }V_{m;{\bf h}_{1},\dots,{\bf h}_{r};K}.\]
	To show that $(J'_{1},\dots,J'_{r},J_{r+1})$ is admissible, fix ${\bf h}_{i}\in J'_{i}, 1\leq i\leq r$, ${\bf h}_{r+1}\in J_{r+1},$ $1\leq m\leq t_{1}$, and let $G_{K}$ be defined as in (\ref{0010}) for all $K\subseteq\{1,\dots,r+1\}$. Let  
	$\emptyset\neq K'\subsetneq K\subseteq \{1,\dots,r+1\}$ such that $\max\{x\in K'\}<\min\{x\in K\backslash K'\}$. We have the following three possible cases for $r+1$: 
	
	\textbf{Case (i): $r+1\notin K$.} Then  $r+1\notin K'$ and so $\emptyset\neq K'\subsetneq K\subseteq \{1,\dots,r\}$. Since $(I_{1},\dots,I_{r})$ is admissible,  either $G_{K'}\subsetneq G_{K}$ or $G_{K'}=H_{1,m}$. 
	
	\textbf{Case (ii): $r+1\in K'$}. This contradicts the assumption that $\max\{x\in K'\}<\min\{x\in K\backslash K' \}.$ So this case is not possible. 
	
	\textbf{Case (iii):} $r+1\in K$ but $r+1\notin K'$. Then $K'\subseteq \{1,\dots,r\}$ and so $J_{r+1}\subseteq V_{m;{\bf h}_{1},\dots,{\bf h}_{r};K'}$. If $G_{K'}\neq H_{1,m}$, then since ${\bf h}_{r+1}\in J_{r+1}\subseteq V_{m;{\bf h}_{1},\dots,{\bf h}_{r};K'}$, the subgroup $G(\c_{1,m}({\bf h}_{r+1}))$ (which is contained in $G_{K}$ since $r+1\in K$) is not contained in  $Q_{m;{\bf h}_{1},\dots,{\bf h}_{r};K'}=G_{K'}$. This implies that $G_{K}\neq G_{K'}$.

	In conclusion, we have that $(J'_{1},\dots,J'_{r},J_{r+1})$ is admissible.
	The second part of the claim is that $J_{r+1}$ is of density $1.$ Since $J'_{1},\dots,J'_{r}$ are finite sets, it suffices to show that every $V_{m;{\bf h}_{1},\dots,{\bf h}_{r};K}$ is of density $1.$
	If $Q_{m;{\bf h}_{1},\dots,{\bf h}_{r};K}=H_{1,m}$, then $V_{m;{\bf h}_{1},\dots,{\bf h}_{r};K}=(\mathbb{Z}^{L})^{s}$ and we are done.
	Now assume that $Q_{m;{\bf h}_{1},\dots,{\bf h}_{r};K}\neq H_{1,m}$.
	By Lemma~\ref{ag}, the set \[V_{m;{\bf h}_{1},\dots,{\bf h}_{r};K}=\{{\bf h}\in(\mathbb{Z}^{L})^{s}\colon G(\c_{1,m}({\bf h}))\nsubseteq Q_{m;{\bf h}_{1},\dots,{\bf h}_{r};K}\}\] is either of density $1,$ or is empty and \[Q_{m;{\bf h}_{1},\dots,{\bf h}_{r};K}=\text{span}_{\mathbb{Q}}\{G(\c_{i,m}(h_{1},\dots,h_{s}))\colon h_{1},\dots,h_{s}\in\mathbb{Z}^{L}\}\cap\mathbb{Z}^{d}=H_{1,m}.\] 
	By our assumption, $V_{m;{\bf h}_{1},\dots,{\bf h}_{r};K}$ is of density $1.$ This finishes the proof of the claim.
	
	\
	
	By Lemma \ref{pet2},
	$A(J'_{1},\dots,J'_{r})=A(J'_{1},\dots,J'_{r},J_{r+1})$. By an approximation argument, there exists a finite subset $J'_{r+1}\subseteq J_{r+1}$ such that 
	$$\F\vl\Bigl\Vert A_{I_{N}}(J'_{1},\dots,J'_{r},J_{r+1})-A_{I_{N}}(J'_{1},\dots,J'_{r},J'_{r+1})\Bigr\Vert_{L^{2}(\mu)}<\varepsilon.$$ Using the induction hypothesis, we get   $$\F\vl\Bigl\Vert A_{I_{N}}(\emptyset)-A_{I_{N}}(J'_{1},\dots,J'_{r},J'_{r+1})\Bigr\Vert_{L^{2}(\mu)}<(r+1)\varepsilon.$$  So (i) holds for $r+1$. Since $(J'_{1},\dots,J'_{r},J_{r+1})$ is admissible, so is $(J'_{1},\dots,J'_{r},J'_{r+1})$, hence (ii) holds for $r+1$. In conclusion, there exist a tuple  
	$(J'_{1},\dots,J'_{dt_{1}})$  of finite subsets of $(\mathbb{Z}^{L})^{s}$ such that 
	\[\F\vl\Bigl\Vert A_{I_{N}}(\emptyset)-A_{I_{N}}(J'_{1},\dots,J'_{dt_{1}})\Bigr\Vert_{L^{2}(\mu)}<dt_{1}\varepsilon\]
	 and $(J'_{1},\dots,J'_{dt_{1}})$ is admissible. Note that
	\begin{equation}\nonumber
	\begin{split}
	W_{1,J'_{1}}\cap\dots\cap W_{1,J'_{dt_{1}}} & =\bigcap_{u=1}^{dt_{1}}  \bigvee_{ \substack{{\bf h}_{u}\in J'_{u} }}W_{1,{\bf h}_{u}}
	\\&=\bigcap_{u=1}^{dt_{1}} \bigvee_{\substack{{\bf h}_{u}\in J'_{u}}} Z_{G(\c_{1,1}({\bf h}_{u})),\dots,G(\c_{1,t_{1}}({\bf h}_{u}))}
	\\&\subseteq  \bigcap_{u=1}^{dt_{1}}Z_{\{G(\c_{1,m}({\bf h}_{u}))\}_{1\leq m\leq t_{1},{\bf h}_{u}\in J'_{u}}},
	\end{split}
	\end{equation}
	where we  used \cref{replacement0} (vii)  in the last inclusion.	For each $1\leq u\leq dt_{1}$, pick some $1\leq m_{u}\leq t_{1}$ and ${\bf h}_{u}\in J'_{u}$. Consider the set
	\[P\coloneqq \text{span}_{\mathbb{Q}}\{G(\c_{1,m_{u}}({\bf h}_{u}))\colon 1\leq u\leq dt_{1}\}\cap \mathbb{Z}^{d}.\]
	By the pigeon-hole principle, there exist $1\leq m\leq t_{1}$ and $1\leq u_{1}<\dots<u_{d}\leq dt_{1}$ such that $m_{u_{1}}=\dots=m_{u_{d}}=m$. 
	For all $1\leq i\leq d$, let $K_{i}=\{u_{1},\dots,u_{i}\}\subseteq\{1,\dots,dt_{1}\}$ and
	\[P_{i}\coloneqq \text{span}_{\mathbb{Q}}\{G(\c_{1,m_{u}}({\bf h}_{u}))\colon u\in K_{i}\}\cap \mathbb{Z}^{d}.\]
	Since $(J'_{1},\dots,J'_{dt_{1}})$ is admissible, for all $1\leq i\leq d-1$, either $P_{i}=H_{1,m}$ or the dimension of $P_{i+1}$ is higher than that of $P_{i}$. Since the dimension of $P_{i}$ can not exceed $d$, we must have that $P_{i}$ contains $H_{1,m}$ for some $1\leq i\leq d$. As $P_{i}\subseteq P$, we have that
	$P$ also contains $H_{1,m}$.
	By Corollary~\ref{ct},
	\begin{equation}\nonumber
	\begin{split}
	W_{1,J'_{1}}\cap\dots\cap W_{1,J'_{dt_{1}}}
	\subseteq  \bigcap_{u=1}^{dt_{1}}Z_{\{G(\c_{1,m}(\textbf{h}_{u}))\}_{1\leq m\leq t_{1},\textbf{h}_{u}\in J'_{u}}}
	\subseteq Z_{H_{1,1}^{\times\infty},\dots,H_{1,t_{1}}^{\times\infty}}.
	\end{split}
	\end{equation}
	Since $\mathbb{E}(f_{1}\vert Z_{(H_{1,1})^{\times \infty},\dots, (H_{1,t_{1}})^{\times \infty}})=0$, $A(J'_{1},\dots,J'_{dt_{1}})=0$ and so $$\F\vl\Bigl\Vert A_{I_{N}}(\emptyset)\Bigr\Vert_{L^{2}(\mu)}<dt_{1}\varepsilon.$$ Since $\varepsilon$ is chosen arbitrary, the left hand side of (\ref{36}) is equal to $\F\vl\Bigl\Vert A_{I_{N}}(\emptyset)\Bigr\Vert_{L^{2}(\mu)}=0,$ which finishes the proof.
\end{proof}

\end{document}